\documentclass[12pt,a4paper,reqno]{amsart}

\usepackage[utf8]{inputenc}
\usepackage[T1]{fontenc}
\usepackage[english]{babel}
\usepackage[margin=2.5cm]{geometry}
\usepackage[shortlabels]{enumitem} 
\usepackage{graphicx}
\usepackage{caption}
\usepackage{subcaption} 
   
\usepackage[colorlinks]{hyperref}
\usepackage{tikz}

\usepackage{amsmath,amssymb,amsfonts,amsthm}
\usepackage{mathtools,mathrsfs}  
\usepackage[capitalize,nameinlink]{cleveref}
\allowdisplaybreaks{} 
 
\definecolor{myBlue}{HTML}{1F77B4}
\definecolor{myGreen}{HTML}{2CA02C} 
\definecolor{myOrange}{HTML}{FF7F0E}
\definecolor{myRed}{HTML}{D62728}

\hypersetup{
    colorlinks = True,
    linkcolor  = myRed, 
    citecolor  = myGreen,
    urlcolor   = myBlue,
}

\providecommand{\N}{\mathbb{N}}
\providecommand{\R}{\mathbb{R}}
\providecommand{\Z}{\mathbb{Z}}
\providecommand{\C}{\mathbb{C}}

\newcommand{\scrC}{\mathscr{C}}

\newcommand{\rmL}{\mathrm{L}} 
\newcommand{\rmH}{\mathrm{H}}

\newcommand{\ttK}{\mathtt{K}} 

\newcommand{\ttT}{\mathtt{T}}

\newcommand{\di}[1]{\mathop{}\!\mathrm{d}#1}
\DeclareMathOperator{\OO}{\mathcal{O}}
\DeclareMathOperator{\oo}{\mathcal{\scriptstyle{}O}}
\DeclareMathOperator{\diver}{div}
\DeclareMathOperator{\curl}{\mathbf{curl}}

\DeclareMathOperator{\spa}{span}
\DeclareMathOperator{\Card}{Card}

\DeclareMathOperator{\sign}{sign}

\newcommand{\abs}[1]{\left\lvert#1\right\rvert}
\newcommand{\opnorm}[1]{{\left\vert\kern-0.25ex\left\vert\kern-0.25ex\left\vert#1\right\vert\kern-0.25ex\right\vert\kern-0.25ex\right\vert}}
\newcommand{\plr}[1]{\left(#1\right)}
\newcommand{\clr}[1]{\left[#1\right]}

\newcommand{\skp}[2]{\left\langle#1, #2\right\rangle}
\newcommand{\skpL}[2]{\left(#1, #2\right)_c}
\newcommand{\skpH}[2]{\left\langle#1, #2\right\rangle_{|\sigma|}}

\newcommand{\norm}[1]{\left\lVert#1\right\rVert}
\newcommand{\normL}[1]{\left\lVert#1\right\rVert_c}
\newcommand{\normH}[1]{\left\lVert#1\right\rVert_{|\sigma|}}

\newcommand{\setst}[2]{\left\lbrace#1\ \middle\vert\ #2\right\rbrace}
\newcommand{\setwt}[2]{\left\lbrace#1\ {:}\ #2\right\rbrace}

\newcommand{\eps}{\varepsilon}
\newcommand{\sm}{\setminus}
\newcommand{\ov}{\overline}
\newcommand{\les}{\lesssim}

\newcommand{\las}{\ell}
\newcommand{\As}{A_\ell}

\theoremstyle{definition}
\newtheorem{defn}{Definition}[section]

\newtheorem{assumption}{Assumption}

\crefname{assumption}{Assumption}{Assumptions}

\theoremstyle{plain}
\newtheorem{lem}[defn]{Lemma}
\newtheorem{prop}[defn]{Proposition}
\newtheorem{thm}[defn]{Theorem}
\newtheorem{cor}[defn]{Corollary}

\theoremstyle{remark}
\newtheorem{rem}[defn]{Remark}

\makeatletter
\renewcommand\subsubsection{
    \@startsection{subsubsection}{3}%
    \z@{.5\linespacing\@plus.7\linespacing}{-.5em}%
    {\normalfont\bfseries}
}
\renewcommand\paragraph{
    \@startsection{paragraph}{4}%
    \z@\z@{-\fontdimen2\font}%
    {\normalfont\bfseries}}
\makeatother

\begin{document}

% Title ===================================================
\title[Sign-changing nonlinear Helmholtz equations]{Nonlinear Helmholtz equations with sign-changing diffusion coefficient}

\author[R. Mandel]{Rainer Mandel}
\address{
  Karlsruhe Institute of Technology,
  Institute for Analysis,
  Englerstra{\ss}e 2,
  D-76131 Karlsruhe, Germany.
}
\email{Rainer.Mandel@kit.edu}

\author[Z. Moitier]{Zo{\"\i}s Moitier}
\address{
  Karlsruhe Institute of Technology,
  Institute for Analysis,
  Englerstra{\ss}e 2,
  D-76131 Karlsruhe, Germany.
}
\email{Zois.Moitier@kit.edu}

\author[B. Verf\"urth]{Barbara Verf\"urth}
\address{
  Karlsruhe Institute of Technology,
  Institute for Applied and Numerical Mathematics
  Englerstra{\ss}e 2,
  D-76131 Karlsruhe, Germany.
}
\email{Barbara.Verfuerth@kit.edu}

\date{\today}

% 35B32: Bifurcations in context of PDEs
% 58E07: Variational problems in abstract bifurcation theory in infinite-dimensional spaces
% 47A10: General theory of linear operators, Spectrum, resolvent
\subjclass[2010]{35B32; 47A10}
\keywords{Bifurcation Theory; Helmholtz equation; Sign-changing; \( \ttT \)-coercivity}

\begin{abstract}
  In this paper, we study nonlinear Helmholtz equations with sign-changing diffusion coefficients on  bounded
  domains. The existence of an orthonormal basis of eigenfunctions is established making use of weak \( \ttT \)-coercivity theory.  All eigenvalues are proved to be bifurcation points and the bifurcating branches are
  investigated both theoretically and numerically.  In a one-dimensional model example we obtain the existence
  of infinitely many bifurcating branches that are mutually disjoint, unbounded, and consist of solutions
  with a fixed nodal pattern.
\end{abstract}

\maketitle

\section{Introduction}\label{sec:intro}

In this paper, we are interested in nonlinear Helmholtz equations of the form
\begin{equation}\label{eq:NLeq}
  -\diver(\sigma(x) \, \nabla u) - \lambda \, c(x) \, u = \kappa(x) \, u^3
  \quad \text{in } \Omega,
  \qquad u \in \rmH_0^1(\Omega)
\end{equation}
where \( \Omega \subset \R^N \) is a bounded domain and the diffusion coefficient \( \sigma \) is sign-changing.
As we will explain in \cref{subsec:Physics}, such problems  occur in the study of time-harmonic wave propagation through metamaterials with negative permeability and nonlinear Kerr-type permittivity.
Up to now, the linear theory dealing with the well-posedness of such problems for right-hand sides \( f(x) \) instead of \( \kappa(x) \, u^3 \) has been studied to some extent both analytically and
numerically~\cite{BBDhiaZwoelf_Timeharmonic, BonCheCia_TCoercivity, CheCia_2013, bonnet-ben_dhia_use_2016, CarCheCia_Eigenvalue,
  BBDhia_MeshRequirements}.
Here, the main difficulty is that the differential operator \( u \mapsto -\diver(\sigma(x) \, \nabla u) \) is not elliptic on the whole domain \( \Omega \).
Accordingly, the standard theory for elliptic boundary value problems based on the Lax-Milgram Lemma does not apply.
In the papers~\cite{BBDhiaZwoelf_Timeharmonic, BonCheCia_TCoercivity} the (weak) \( \ttT \)-coercivity approach was introduced to develop a solution theory for such linear problems.
The guiding idea of this method is to require that the strongly indefinite bilinear form
\[
  \rmH_0^1(\Omega) \times \rmH_0^1(\Omega) \to \R, \qquad
  (u,v) \mapsto \int_\Omega \sigma(x) \, \nabla u \cdot \nabla v \di{x} - \lambda \int_\Omega c(x) \, u v \di{x}
\]
satisfies the assumptions of the Lax-Milgram Lemma up to some compact perturbation once \( v \) is replaced by \( \ttT v \) for some isomorphism \( \ttT : \rmH_0^1(\Omega)\to \rmH_0^1(\Omega) \).
Our intention is to combine this approach with methods from nonlinear analysis to study the nonlinear Helmholtz equation \cref{eq:NLeq}.

\medskip

Under reasonable assumptions on \( \Omega \), \( \sigma \), and \( c \) our main contributions  are the following:
\begin{enumerate}[(i)]
  \item\label{item:contrib_1} There is an orthonormal basis \( \plr{\phi_j}_{j\in\Z} \) of \( \rmL^2(\Omega) \) that consists of eigenfunctions of the linear differential operator \( u\mapsto -{c(x)}^{-1} \diver(\sigma(x) \, \nabla u) \).
        The corresponding eigenvalue sequence \( \plr{\lambda_j}_{j\in\Z} \) is unbounded from above and from
        below. Moreover, the eigenfunctions are dense in \( \rmH_0^1(\Omega) \).

  \item\label{item:contrib_2}
        If \( \kappa\in \rmL^\infty(\Omega) \) then each of the eigenvalues \( \lambda_j \) is a bifurcation point
        for \cref{eq:NLeq} with respect to the trivial solution family \( \setwt{(0, \lambda)}{\lambda \in \R} \).
        By definition, this means that for all \( j\in\Z \) there is a sequence \( \plr{u_j^n,
          \lambda_j^n}_{n \in \N} \subset \rmH_0^1(\Omega) \sm \{0\} \times \R \) of nontrivial solutions such
        that \( (u_j^n, \lambda_j^n) \to (0, \lambda_j) \) in \( \rmH_0^1(\Omega) \times \R \) as \( n\to\infty \).
        Our numerical illustrations in \cref{sec02} indicate that these solutions may be located on a smooth
        and unbounded curve in \( \rmH_0^1(\Omega) \times \R \) going through the point \( (0, \lambda_j) \).
        We can also prove this for some one-dimensional model problem.

  \item\label{item:contrib_3} In a one-dimensional case, for any given \( \lambda \in \R \) there are
        infinitely many nontrivial solutions for \cref{eq:NLeq} provided that \( \kappa\in \rmL^\infty(\Omega) \)
        is uniformly positive or uniformly negative. This result is obtained using variational methods instead of
        bifurcation theory.
\end{enumerate}

\medskip

A few comments regarding \cref{item:contrib_1,item:contrib_2,item:contrib_3} are in order.
As to~\ref{item:contrib_1}, the existence of an orthonormal basis of eigenfunctions in \( \rmL^2(\Omega) \) may be considered as a known fact in view of~\cite[Section~1]{CarCheCia_Eigenvalue}.
However, we could not find a reference in the literature that covers our setting, so we briefly review this in \cref{sec03}.
On the other hand, the Weyl law asymptotics seem to be new.
Our proofs rely on the weak \( \ttT \)-coercivity approach developed in~\cite{BBDhiaZwoelf_Timeharmonic, BonCheCia_TCoercivity}.
Under this assumption, the linear theory turns out to be analogous to the linear (Fredholm) Theory for elliptic boundary value problems.
The construction of isomorphisms \( \ttT : \rmH_0^1(\Omega) \to \rmH_0^1(\Omega) \), however, is a research
topic on its own and depends on the precise setting, notably on the nature of the interface where the sign of \( \sigma \) jumps, see, \emph{e.g.}~\cite{BonCheCia_TCoercivity, BBDhia_MeshRequirements}.

\medskip

Our main bifurcation theoretical results from~\ref{item:contrib_2} also rely on the functional analytical
framework given by the weak \( \ttT \)-coercivity approach.
The task is to detect nontrivial solutions of \cref{eq:NLeq} that bifurcate from the trivial solution family.
Being given the linear theory and the Implicit Function Theorem, one knows that such bifurcations can only occur at \( \lambda = \lambda_j \) for some \( j \in \Z \).
To prove the occurrence of bifurcations we resort to variational bifurcation theory (see below for references).
Here the main difficulty comes from the fact  that the associated energy functional
\begin{equation}\label{eq:def_Phi}
  \Psi_\lambda(u)
  \coloneqq \frac{1}{2} \int_\Omega \sigma(x) \abs{\nabla u(x)}^2  \di{x}
  - \frac{\lambda}{2} \int_\Omega c(x) \, {u(x)}^2 \di{x}
  - \frac{1}{4} \int_\Omega \kappa(x) {u(x)}^4  \di{x}
\end{equation}
is strongly indefinite due to the sign change of \( \sigma \).
Strong indefiniteness means that the quadratic part of \( \Psi_\lambda \) is positive definite on an infinite-dimensional subspace of \( \rmH_0^1(\Omega) \), and it is negative definite on another
infinite-dimensional subspace.
As a consequence, standard results in this area going back to B\"ohme~\cite[Satz~II.1]{Boehme}, Marino~\cite{Marino_Bif}, and Rabinowitz~\cite[Theorem~11.4]{Rab_Minimax} do not apply.
Instead, we demonstrate how to apply the more recent variational bifurcation theory developed by Fitzpatrick, Pejsachowicz, Recht, Waterstraat~\cite{FitzPejRech,PejWat_Bifurcation}.
Better results are obtained for eigenvalues with odd geometric multiplicity, which is based on Rabinowitz' Global Bifurcation Theorem~\cite{Rab_SomeGlobal}.
In our one-dimensional model case we significantly improve and numerically illustrate our bifurcation results with the aid of bifurcation diagrams (\cref{sec02}).
The latter provide a qualitative picture of the bifurcation scenario given that each point in such a diagram corresponds to \( (\lambda, \norm{u}) \) where \( (u, \lambda) \) solves \cref{eq:NLeq}.

\medskip

As to~\ref{item:contrib_3}, the strong indefiniteness of \( \Psi_\lambda \) also makes it harder to prove the existence of critical points.
Note that critical points \( u \in \rmH_0^1(\Omega) \) satisfy \( \Psi_\lambda'(u) = 0 \), which is equivalent to \cref{eq:NLeq}.
In the case of positive diffusion coefficients, the Symmetric Mountain Pass
Theorem~\cite[Theorem~10.18]{AmbMal} applies and yields infinitely many nontrivial solutions.
In the context of strongly indefinite functionals, an analogous result was established only recently by Szulkin and Weth~\cite[Section~5]{SzuWet_Nehari}.
We will show how to apply their abstract results under reasonable extra assumptions in order to obtain
infinitely many nontrivial solutions of \cref{eq:NLeq} for any given \( \lambda \in \R \) assuming
\( \kappa\geq \alpha>0 \) or \( \kappa\leq -\alpha<0 \).

\medskip

Before commenting on the physical background of \cref{eq:NLeq} we wish to emphasize that our main goal is to
bring (weak) \( \ttT \)-coercivity theory and nonlinear analysis together.
Accordingly, we do not aim for the most general assumptions for our results to hold true.
For instance, we avoid  technicalities related to the regularity of the interface where the sign of \( \sigma \) jumps.
Similarly, we content ourselves with the special nonlinearity \( \kappa(x) \, u^3 \).
Only little effort is needed to generalize our bifurcation results as well as our variational existence
results to more general nonlinearities in all space dimensions \( N \geq 1 \).

\subsection{Physical motivation}\label{subsec:Physics}

We comment on the physical background of \cref{eq:NLeq}.
The propagation of electromagnetic waves with a fixed temporal frequency parameter
\( \omega\in\R \) is governed by the time-harmonic Maxwell's equations
\begin{equation}\label{eq:Maxwell-time}
  i \omega D - \curl H = 0
  \qquad \text{and} \qquad
  i \omega B + \curl E = 0.
\end{equation}
Here, charges and currents are assumed to be absent.
The symbols \( E, D : \R^3 \to \C^3 \) denote the electric field and the electric induction and \( H, B : \R^3 \to \C^3 \) represent the
magnetic field and the magnetic induction, respectively.
In nonlinear Kerr media the constitutive relations between these fields are given by
\begin{equation}\label{eq:constitutive}
  D = \eps(x) \, E + \chi(x) \, |E|^2 E
  \qquad \text{and} \qquad
  B = \mu(x) \, H
\end{equation}
where \( \eps, \mu, \chi \) are real-valued, see~\cite[Chapter~4]{Boyd}. In physics, these quantities are
called permittivity, permeability and third-order susceptibility of the given medium, respectively.
Plugging in this ansatz into \cref{eq:Maxwell-time} one finds
\begin{equation}\label{eq:Maxwell-time2}
  \curl\plr{ {\mu(x)}^{-1} \curl E } = \omega^2 \eps(x) \, E + \omega^2 \chi(x) \, |E|^2 E.
\end{equation}
Now we assume that the propagation of electromagnetic waves is considered in a closed waveguide \( \Omega
\times \R \) having a bounded cross-section \( \Omega \subset \R^2 \) and all material parameters only depend
on the cross-section variable.
With an abuse of notation, this means \( \mu(x) = \mu(y) \), \( \eps(x) = \eps(y) \), and \( \chi(x) =
\chi(y) \) where \( x = (y, z) \in \Omega \times \R \).
This is a natural assumption for the modelling of layered cylindrical waveguides.
If then the electric field is of the special form \( E(x) = (0, 0, u(y)) \) with \( u \) real-valued, we infer
\[
  -\diver({\mu(y)}^{-1} \, \nabla u) - \omega^2 \eps(y) \, u = \omega^2 \chi(y) \, u^3
  \quad \text{in } \Omega,
  \qquad u \in \rmH_0^1(\Omega),
\]
which corresponds to \cref{eq:NLeq} in the two-dimensional case.
For complex-valued \( u \) the nonlinearity is given by \( |u|^2u \).
Sign-changing diffusion coefficients \( \sigma(y) \coloneqq {\mu(y)}^{-1} \) may occur if one of the layers of the waveguide is filled with a
negative-index metamaterial (NIM) where the permeability \( \mu \) may be negative, see, \emph{e.g.},~\cite{OBrPen_physics}.
Note that any such solution \( u \) determines \( E, D, B, H \) via \cref{eq:Maxwell-time2,eq:constitutive}.
We mention that one may equally solve for the magnetic field, which leads to Helmholtz-type problems with
nonlinear (\emph{i.e.}\ solution-dependent) and possibly sign-changing diffusion coefficients.

\subsection{Notation}

In the following, we equip the Hilbert spaces \( \rmH_0^1(\Omega) \) and \( \rmL^2(\Omega) \) with the inner products
\[
  \skpH{u}{v} \coloneqq \int_\Omega \abs{\sigma(x)} \, \nabla u \cdot \nabla v \di{x},
  \qquad \text{and} \qquad
  \skpL{u}{v} \coloneqq \int_\Omega c(x) \, u v \di{x},
\]
respectively.
Our assumptions on \( \Omega \), \( c \), and \( \sigma \) will imply that the associated norms \( \normH{\cdot} \) and \( \normL{\cdot} \) are equivalent to the standard norms on these spaces.
Moreover, we introduce the bilinear form
\[
  a : \rmH_0^1(\Omega) \times \rmH_0^1(\Omega) \to \R,
  \quad
  (u,v) \mapsto \int_\Omega \sigma(x) \, \nabla u \cdot \nabla v \di{x}.
\]

\subsection{Outline}

The paper is organized as follows.
\Cref{sec:mainresults} contains a mathematically rigorous statement of our main results dealing with bifurcations for \cref{eq:NLeq} from the trivial solution family.
These results are illustrated numerically in \cref{sec02} with the aid of bifurcation diagrams.
Those illustrate the evolution of solutions along the branches as well the global behavior of the latter.
In \cref{sec03}, we set up the linear theory that we need to prove our main bifurcation theoretical results in \cref{sec04}.
Finally, \cref{sec:var_method} contains further existence results for nontrivial solutions of \cref{eq:NLeq} obtained by a variational approach.
The proof is based on the Critical Point Theory from~\cite[Chapter~4]{SzuWet_Nehari}.

\section{Main results}\label{sec:mainresults}

We now come to the precise formulation of our main results for \cref{eq:NLeq}, \emph{i.e.},
\begin{equation*}
  -\diver(\sigma(x) \, \nabla u) - \lambda \, c(x) \, u = \kappa(x)u^3
  \quad \text{in } \Omega,
  \qquad u \in \rmH_0^1(\Omega).
\end{equation*}
Here, \( \Omega \) and the coefficient functions \( \sigma \), \( c \), \( \kappa \) will be chosen as follows:

\begin{assumption}\label{hyp:sigma_c}
  \hfill
  \begin{enumerate}
    \item \( \Omega \subset \R^N \) for \( N \in \{ 1, 2, 3 \} \) is a bounded domain and there are
          nonempty open subsets \( \Omega_+, \Omega_- \subset \Omega \) such that \( \ov{\Omega_+ \cup \Omega_-} = \ov{\Omega} \) and \( \Omega_+ \cap \Omega_- = \varnothing \).

    \item \( \sigma > 0 \) on \( \Omega_+ \), \( \sigma < 0 \) on \( \Omega_- \) and \( \abs{\sigma} + \abs{\sigma}^{-1} \in \rmL^\infty(\Omega) \).

    \item \( c  \in \rmL^\infty(\Omega) \) with \( c(x)  \geq \alpha > 0 \) for almost all \( x \in \Omega \).

    \item \( \kappa \in \rmL^\infty(\Omega) \).
  \end{enumerate}
\end{assumption}

\begin{assumption}\label{hyp:w_T_cor}
  There is a bounded linear invertible operator \( \ttT : \rmH_0^1(\Omega) \to \rmH_0^1(\Omega) \) such that the bilinear form \( (u,v) \mapsto a(u,\ttT v) +  \skpH{\ttK u}{v} \) is continuous and coercive on \( \rmH_0^1(\Omega) \times \rmH_0^1(\Omega) \) for some compact operator \( \ttK : \rmH_0^1(\Omega) \to \rmH_0^1(\Omega) \).
\end{assumption}

Later, in \cref{cor:ONB}, we  show that \cref{hyp:sigma_c,hyp:w_T_cor} ensure the existence of an orthonormal basis \( \plr{\phi_j}_{j\in\Z} \) of \( \plr{\rmL^2(\Omega), \skpL{\cdot}{\cdot}} \) consisting of eigenfunctions associated to the linear differential operator \( u \mapsto -{c(x)}^{-1} \diver(\sigma(x) \, \nabla u) \) appearing in \cref{eq:NLeq}.
Due to the sign-change of \( \sigma \) the corresponding sequence of eigenvalues \( \plr{\lambda_j}_{j\in\Z} \) can be indexed in such a way that \( \lambda_j \to \pm\infty \) holds as \( j\to\pm\infty \).
In \cref{thm:bif} below we show that nontrivial solutions of \cref{eq:NLeq} bifurcate from the trivial branch \( \setwt{(0, \lambda)}{\lambda \in \R} \) at any of these eigenvalues.
If the eigenvalue comes with an odd-dimensional eigenspace, we even find that the bifurcating nontrivial solutions  lie on connected sets
\( \mathcal{C}_j \subset \rmH_0^1(\Omega) \times \R \) that are  unbounded or return to the trivial solution branch \( \setwt{(0, \lambda)}{\lambda \in \R} \) at some other bifurcation point.
As in~\cite{Rab_SomeGlobal}, for any given \( j \in \Z \), the set \( \mathcal{C}_j \) is defined as the connected component of \( (0, \lambda_j) \) in \( \mathcal{S} \), which in turn is defined as the closure of all nontrivial solutions of \cref{eq:NLeq} in the space \( \rmH_0^1(\Omega) \times \R \).
Our first main result reads as follows:

\begin{thm}\label{thm:bif}
  Assume~\ref{hyp:sigma_c} and~\ref{hyp:w_T_cor}.
  Let \( \plr{\lambda_j}_{j \in \Z} \) denote the unbounded sequence of eigenvalues from \cref{cor:ONB}.
  Then each \( (0, \lambda_j) \) is a bifurcation point for \cref{eq:NLeq}.
  If \( \lambda_j \) has odd geometric multiplicity, then the connected component \( \mathcal{C}_j \) in \( \mathcal{S} \) containing \( (0, \lambda_j) \) satisfies Rabinowitz' alternative:
  \begin{enumerate}[(I)]
    \item \( \mathcal{C}_j \) is unbounded in \( \rmH_0^1(\Omega) \times \R \)  or
    \item \( \mathcal{C}_j \) contains another trivial solution \( (0,\lambda_k) \) with \( k \neq j \).
  \end{enumerate}
\end{thm}

\begin{rem}\label{rem:AssumptionA2}
  \hfill
  \begin{enumerate}[(a)]
    \item In \cref{hyp:sigma_c} we may as well assume \( c(x) \leq -\alpha < 0 \); it suffices to replace \( (c, \lambda) \) by \( (-c, -\lambda) \).
          On the other hand we cannot assume \( c \) to be sign-changing since we will need that \( \skpL{\cdot}{\cdot} \) is an inner product on \( \rmL^2(\Omega) \).
          In \cref{rem:T1_assumption} (c), we show that one cannot expect our results to hold for general sign-changing \( c \in \rmL^\infty(\Omega) \).

    \item We need not require a priori smoothness properties of \( \Omega \) or the interface \( \Gamma \coloneqq \ov{\Omega_+} \cap \ov{\Omega_-} \), but imposing those are natural when it comes to verify \cref{hyp:w_T_cor}, see the Theorems 2.1, 3.1, 3.3, 3.7, 3.10 from~\cite{BonCheCia_TCoercivity}.
          It is known that \cref{hyp:w_T_cor} does not always hold, for example in 2D if \( \sigma_- / \sigma_+ \in \clr{-3, -\frac{1}{3}} \) and the interface \( \Gamma \) has a right angle corner, see~\cite{BonCheCia_TCoercivity, bonnet-ben_dhia_use_2016}.
  \end{enumerate}
\end{rem}

We strengthen our result in some 1D model example where we can show the following:
\begin{itemize}
  \item  \cref{hyp:sigma_c,hyp:w_T_cor} are satisfied.

  \item The eigenvalues \( (\lambda_j) \) are simple and in particular have odd geometric multiplicity.

  \item  The eigenpairs \( (\phi_j, \lambda_j) \) are almost explicitly known.

  \item  The case (II) in Rabinowitz' Alternative  is ruled out, hence all \( \mathcal{C}_j \) are
        unbounded.

\end{itemize}

The setting is as follows: Assume that \( \ov{\Omega} = \ov{\Omega_-} \cup \ov{\Omega_+} \) is an interval with precisely two non-void sub-intervals \( \Omega_- = (a_-, 0) \) and \( \Omega_+ = (0, a_+) \) with \( a_- < 0 < a_+ \).
The coefficient function \( c \) and \( \sigma \) satisfy \( c(x) = c_\pm \) resp.~\( \sigma(x) = \sigma_\pm \) on \( \Omega_\pm \) where \( c_\pm > 0 \) and \( \sigma_+ > 0 > \sigma_- \) are constants.
For such domains and coefficients we consider the nonlinear problem
\begin{equation}\label{eq:NLeq1D}
  -\frac{\di{}}{\di{x}}\plr{\sigma(x) \, u'} - \lambda c(x) \, u =  \kappa(x) u^3
  \quad \text{in } \Omega,
  \qquad u \in \rmH_0^1(\Omega).
\end{equation}

\begin{cor}\label{cor:bif1D}
  Assume that \( \Omega \), \( c \), \( \sigma \) are as above, and \( \kappa \in \rmL^\infty(\Omega) \).
  Let \( {(\lambda_j)}_{j\in\Z} \) denote the unbounded sequence of simple eigenvalues
  from \cref{cor:ONB} ordered according to
  \( \ldots < \lambda_{-2} < \lambda_{-1} < 0 < \lambda_1 < \lambda_2 < \ldots \) and \( \lambda_{-1} < \lambda_0 < \lambda_1 \) with
  \[
    \lambda_0<0 \; \Leftrightarrow \; \frac{\sigma_+a_-}{a_+\sigma_-} < 1,
    \qquad
    \lambda_0 = 0 \; \Leftrightarrow \;  \frac{\sigma_+a_-}{a_+\sigma_-} = 1.
    \qquad
    \lambda_0>0 \; \Leftrightarrow \; \frac{\sigma_+a_-}{a_+\sigma_-} > 1,
  \]
  Then the connected component \( \mathcal{C}_j \subset \rmH_0^1(\Omega) \times \R \) in \( \mathcal{S} \) containing \( (0, \lambda_j) \) is unbounded, and we have \( \mathcal{C}_j \cap \mathcal{C}_k = \emptyset \) for \( j \neq k \).
  All \( (u,\lambda)\in \mathcal C_j \) with \( u\neq 0 \) have the following property:
  \begin{enumerate}[(i)]
    \item\label{item:nodal_neg} If \( j \leq -1 \) then \( u \) has \( |j| \) interior zeros in \( \Omega_- \) and satisfies \( \abs{u'} > 0 \) on \( \ov{\Omega_+} \).

    \item\label{item:nodal_zero} If \( j = 0 \) then \( u \) has no interior zeros in \( \Omega \) and satisfies \( \abs{u'} > 0 \) on \( \ov{\Omega_\pm} \).

    \item\label{item:nodal_pos} If \( j \geq 1 \) then \( u \) has \( j \) interior zeros in \( \Omega_+ \) and satisfies \( \abs{u'} > 0 \) on \( \ov{\Omega_-} \).
  \end{enumerate}
\end{cor}

The seemingly complicated ordering of the eigenvalues is exclusively motivated by the nodal patterns given by
\cref{item:nodal_neg,item:nodal_zero,item:nodal_pos}.
Here, \( \abs{u'} > 0 \) on \( \ov{\Omega_\pm} \) means that the continuous extension of \( \abs{u'} : \Omega_\pm \to \R \) to \( \ov{\Omega_\pm} \) is positive.
We stress that nontrivial solutions \( u \) are smooth away from the interface \( x=0 \) and continuous at \( x=0 \), but they are not continuously differentiable at this point.
In fact, \( \sigma u' \) is continuous on \( \ov{\Omega} \) so that \( u'(0) \) does not exist in the classical sense.
In the following \cref{sec02} our results are illustrated with the aid of bifurcation diagrams.

\section{Visualization of bifurcation results via PDE2path}\label{sec02}

In this section, we illustrate our theoretical results  of \cref{thm:bif} and \cref{cor:bif1D} with numerical bifurcation diagrams.
These diagrams show the value of \( \lambda \) on the \( x \)-axis and the \( \rmL^2 \)-norm of solutions \( u \)
for that \( \lambda \) on the \( y \)-axis. Thereby, (numerical) bifurcation diagrams allow to get an overview of
the ``structure'' of solutions and, in particular, to visualize the connected components \( \mathcal{C}_j \)
of \cref{thm:bif}. The results were obtained with the package pde2path~\cite{UWR14, DRUW14}, version 2.9b and using Matlab 2018b.
The code to reproduce the numerical results is available on Zenodo with DOI \href{https://doi.org/10.5281/zenodo.5707422}{\texttt{10.5281/zenodo.5707422}}.

\subsection{One-dimensional example}

We consider \( \Omega=(-5, 5) \) with \( \Omega_- = (-5, 0) \), \( \Omega_+ = (0,5) \) and \( c \equiv 1 \).
The diffusion coefficient \( \sigma \) is chosen piecewise constant, set \( \sigma_+ = 1 \) and compare two different values for \( \sigma_- \), namely \( \sigma_- \in \{-2, -1.005\} \).
We consider \cref{eq:NLeq1D} in this special case, \emph{i.e.},
\[
  -\frac{\di{}}{\di{x}}\big(\sigma(x) \, u'\big) - \lambda  u =  u^3 \quad
  \text{in } \Omega, \qquad
  u\in H_0^1(\Omega).
\]
We choose a tailored finite element mesh which is refined close to \( \Gamma = \{0\} \) in the following way.
We start with an equidistant mesh with \( h = 2^{-9} \), \emph{i.e.}, \( \Omega \) is divided into \( 5120 \) equal subintervals.
Then, we refine all intervals which are closer than \( 0.1 \) to \( \Gamma \) five times by halving them.
This finally means that  intervals close to  \( \Gamma \) are only \( 2^{-14} \) long.
We point out that this finely resolved mesh is required to faithfully represent the interface behavior at \( \Gamma=\{0\} \), especially for \( \sigma_-=-1.005 \).
An insufficient mesh resolution does not only influence the numerical quality  of the eigenfunctions or solutions along the branches, but also the (qualitative picture) of the bifurcation diagram.
We validated our results by assuring that a further refinement of the mesh (halving all intervals) leads to
the same results and conclusions.

\subsubsection{Bifurcation diagrams and eigenfunctions for different contrasts}

We first investigate whether \( \frac{\sigma_+}{\sigma_-} \approx  -1 \) influences the bifurcation diagrams.
For this, we allow \( \lambda \) to vary in the interval \( [-10, 15] \).
The bifurcation diagrams are depicted in \cref{fig:1d:bif} for \( \sigma_-=-2 \) and \( \sigma_-=-1.005 \).

\begin{figure}[hbtp]
  \centering
  \includegraphics[width=0.95\textwidth,trim=20mm 0mm 25mm 5mm,
    clip=true,keepaspectratio=false]{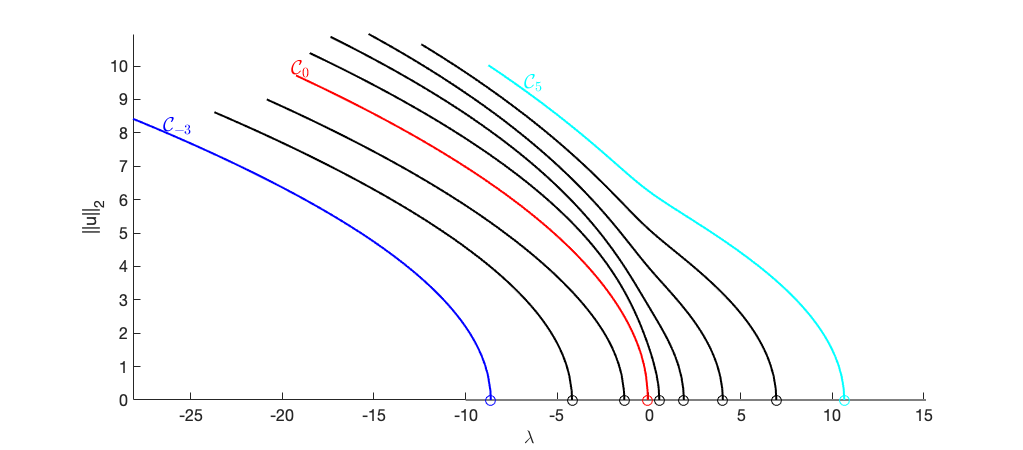}
  \includegraphics[width=0.95\textwidth,trim=20mm 0mm 25mm 5mm,
    clip=true,keepaspectratio=false]{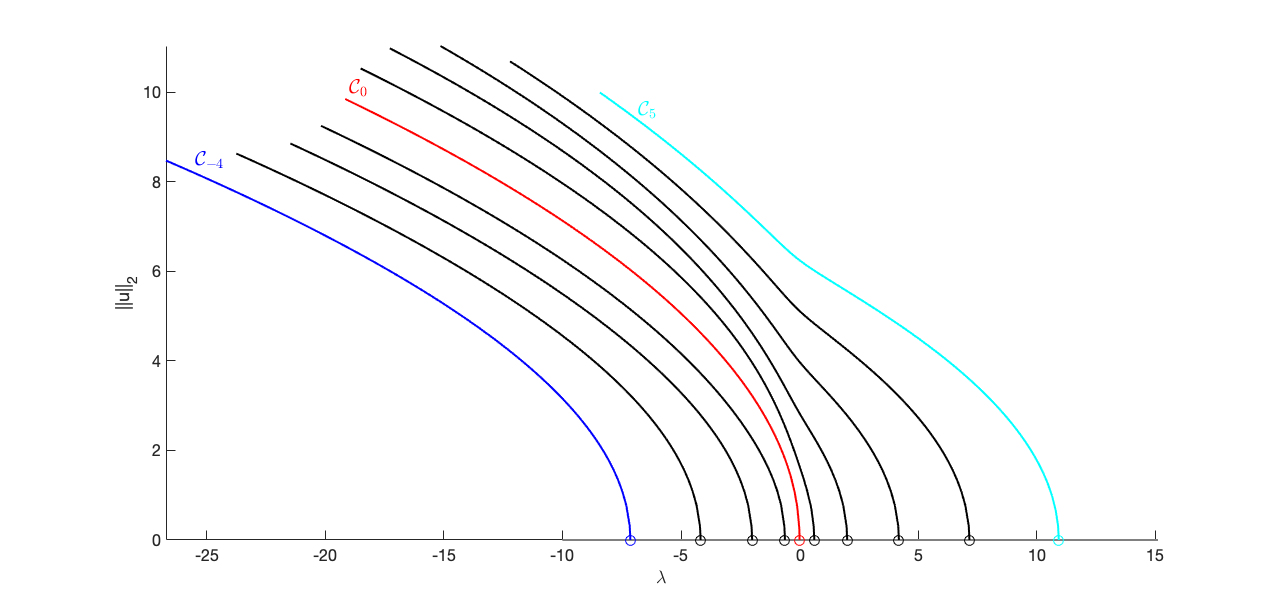}
  \caption{Bifurcation diagrams for \( \sigma_-=-2 \) (top) and \( \sigma_-=-1.005 \) (bottom).}%
  \label{fig:1d:bif}
\end{figure}

Qualitatively, they are quite similar with clearly separated, apparently unbounded branches without secondary
bifurcations.
Note that the bending direction of the branches to the left is determined by the sign of the
nonlinear term and can be predicted by the bifurcation formulae (I.6.11) in~\cite{Kielhoefer_Bifurcation}.
The first striking phenomenon due to the sign-changing coefficient is the occurrence of eigenvalues and, hence, bifurcation points, with negative value.
In fact, for sign-changing \( \sigma \), there are two families  of eigenvalues diverging to \( \pm \infty \),  see \cref{thm:bif}.
We use the following labeling of branches (\emph{cf.} \cref{fig:1d:bif}): The branch starting closest to zero
is labeled as \( \mathcal{C}_0 \) and the branches for negative and positive bifurcation points are labeled as
\( \mathcal{C}_{-i} \) and \( \mathcal{C}_{i} \) with \( i \in \mathbb N \), respectively. The absolute  value of
\( i \) increases as \( |\lambda| \to \infty \).
In our setting  this labeling of the branches is consistent with the notation introduced in \cref{sec03}.

Besides the eigenvalues, we also study the eigenfunctions by considering the solutions at the first point of each branch in \cref{fig:1d:eigfct}.
We display the branch name according to \cref{fig:1d:bif} as well as the value of \( \lambda \) at the bifurcation point.

\begin{figure}[hbtp]
  \includegraphics[width=0.29\textwidth, trim=33mm 80mm 37mm 85mm, clip=true]{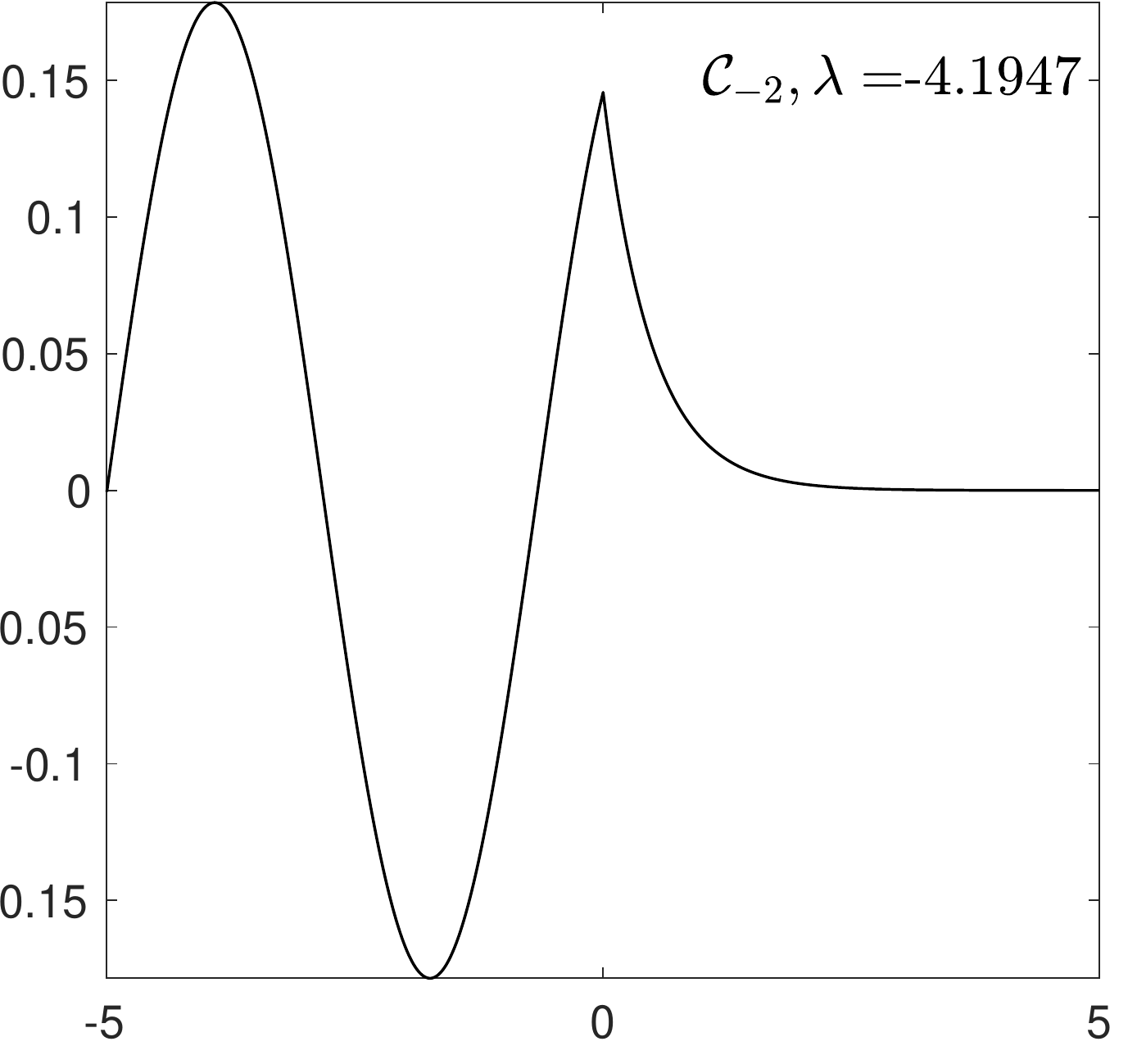}%
  \hspace{2ex}%
  \includegraphics[width=0.315\textwidth, trim=25mm 80mm 33mm 85mm, clip=true]{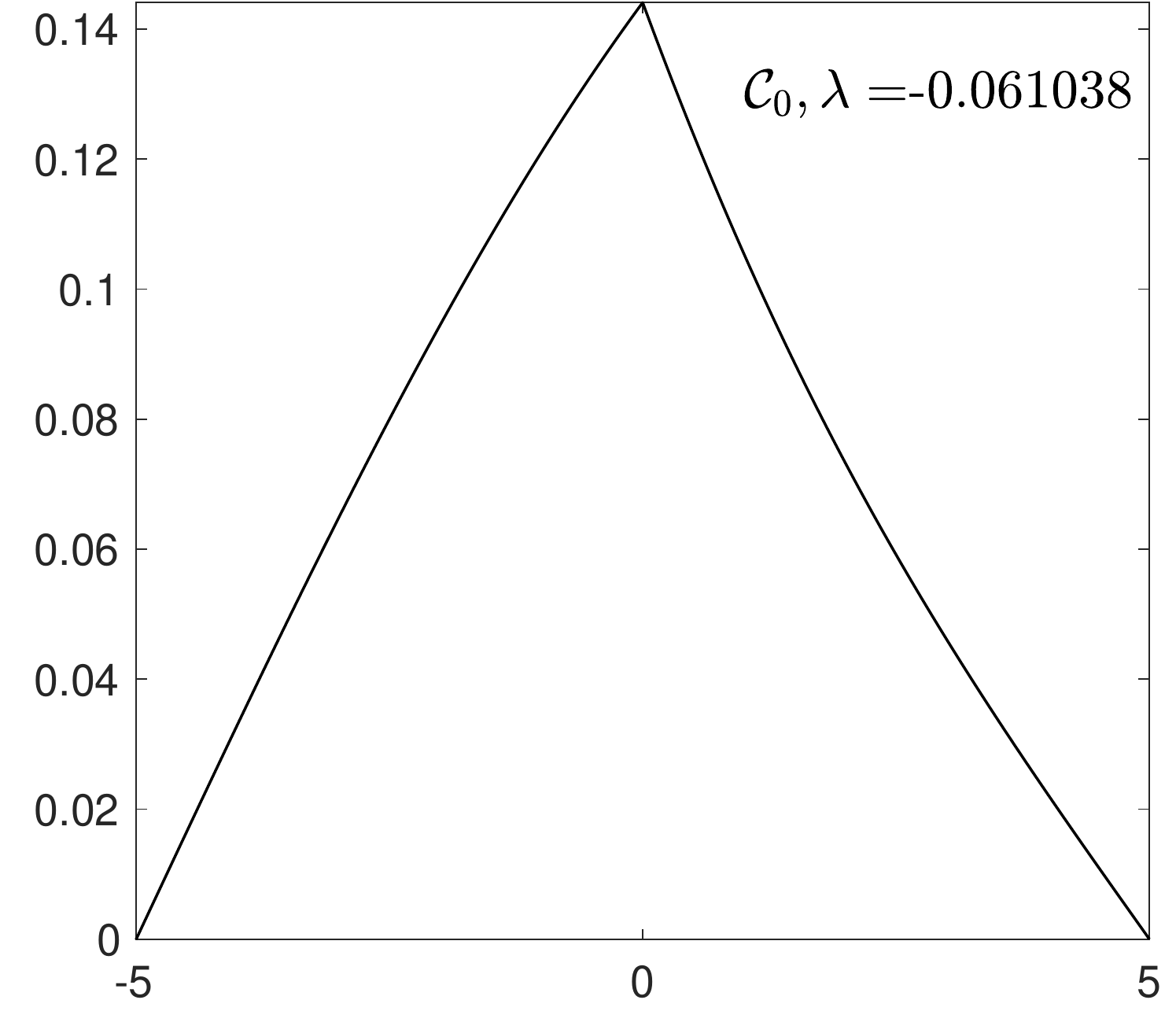}%
  \hspace{2ex}%
  \includegraphics[width=0.29\textwidth, trim=33mm 80mm 37mm 85mm, clip=true]{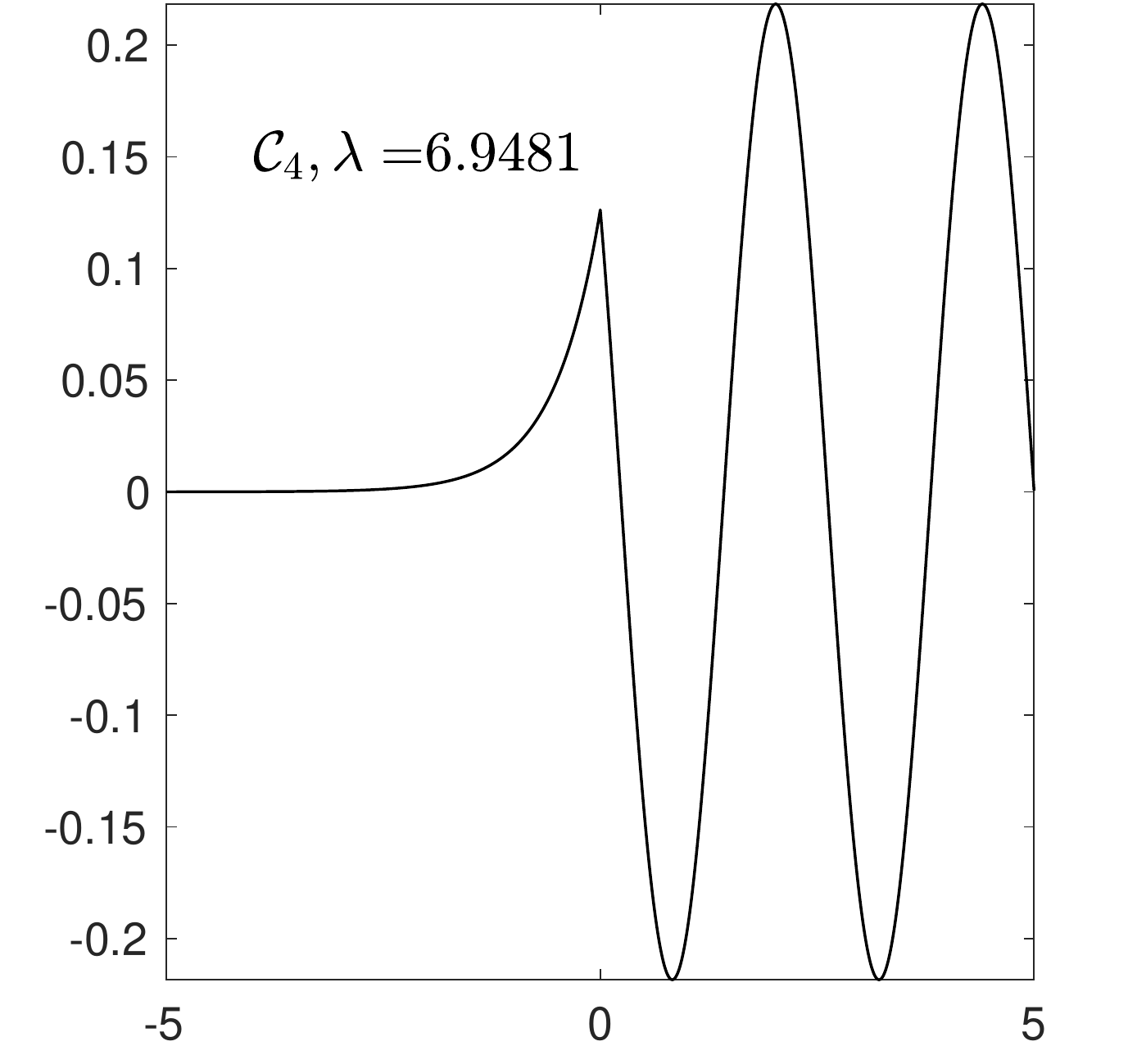}
  \includegraphics[width=0.29\textwidth, trim=19mm 80mm 24mm 85mm, clip=true]{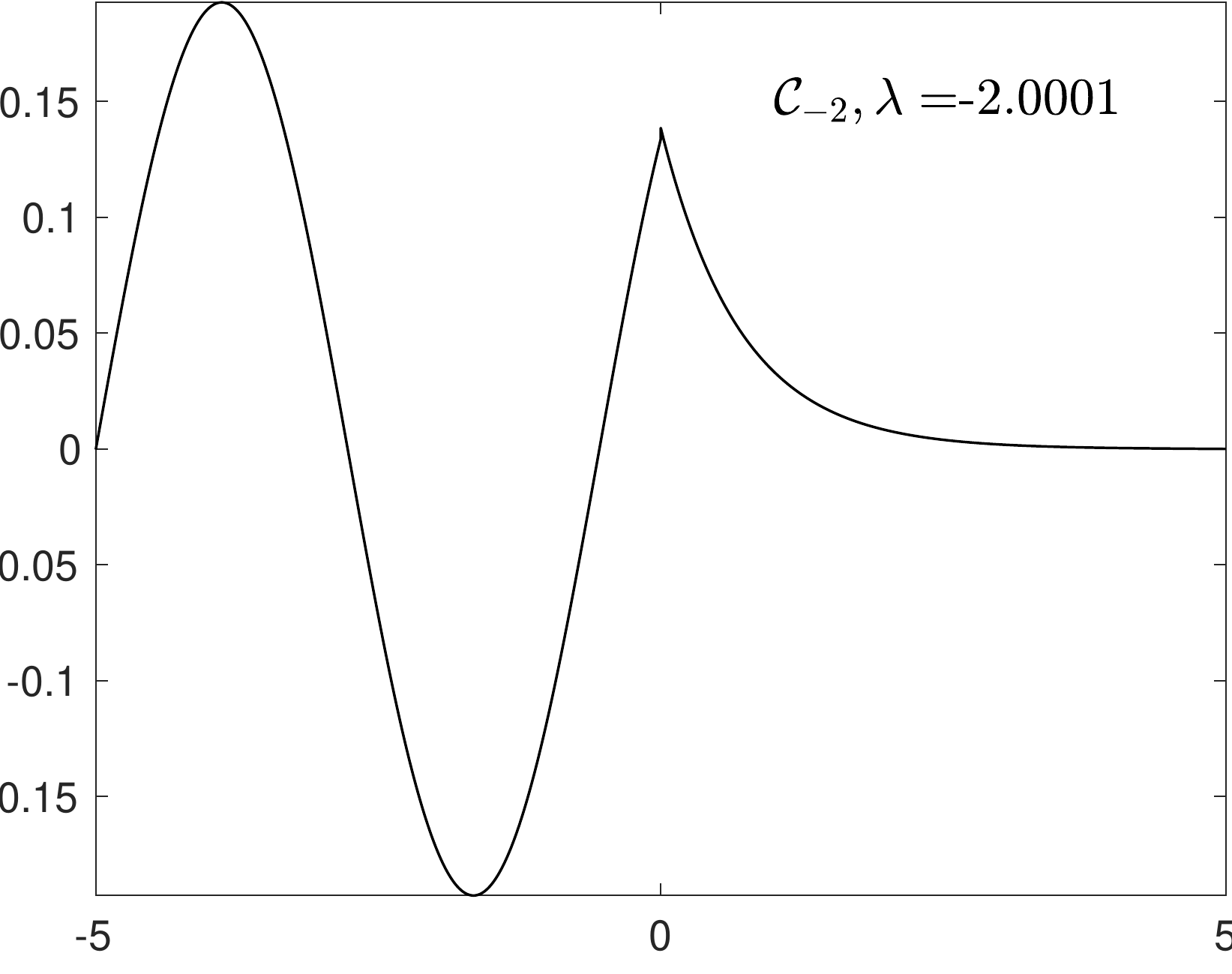}%
  \hspace{2ex}%
  \includegraphics[width=0.315\textwidth, trim=12mm 80mm 20mm 85mm, clip=true]{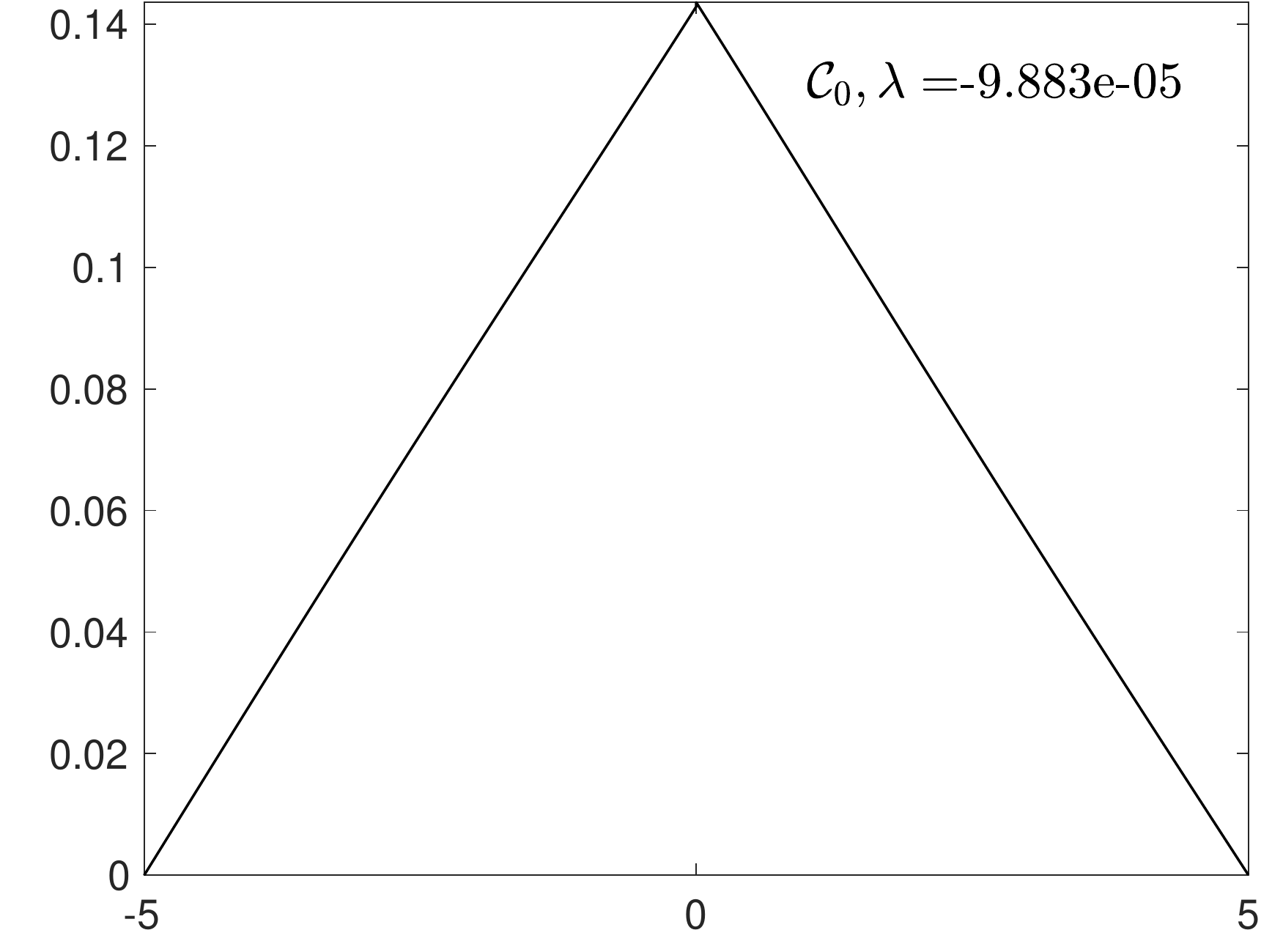}%
  \hspace{2ex}%
  \includegraphics[width=0.29\textwidth, trim=19mm 80mm 24mm 85mm, clip=true]{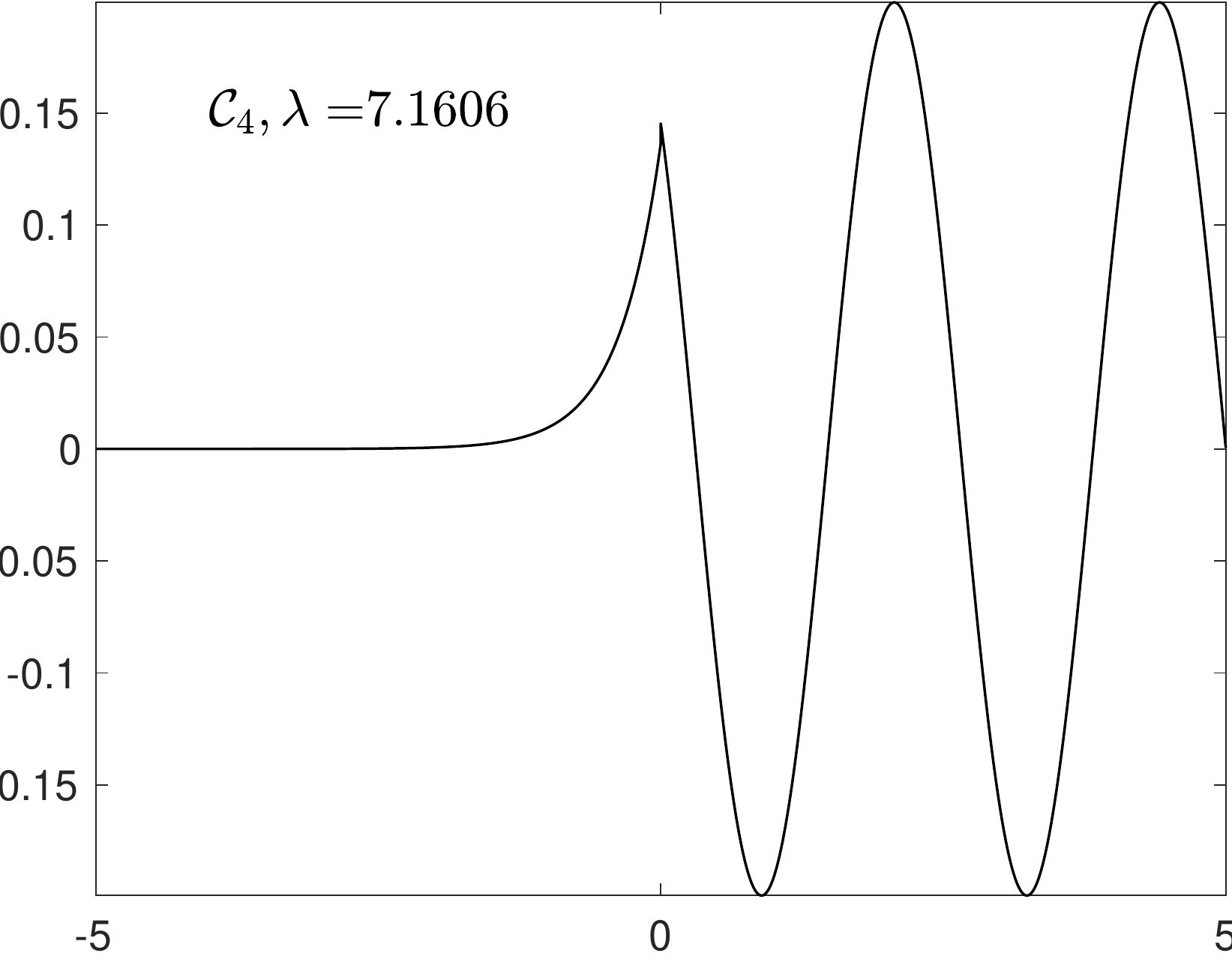}
  \caption{First solution on branches \( \mathcal{C}_{-2} \), \( \mathcal{C}_0 \) and \( \mathcal{C}_4 \)
    (from left to right) for \( \sigma_-=-2 \) (top row) as well as \( \sigma_-=-1.005 \) (bottom row).}%
  \label{fig:1d:eigfct}
\end{figure}

As (partly) expected from~\cite{CarCheCia_Eigenvalue}, we make the following observations.
Firstly, the solutions are concentrated (w.r.t.~the \( \rmL^2 \)-norm) on the ``oscillatory part'',
which  is \( \Omega_- \) for negative eigenvalues (left column of \cref{fig:1d:eigfct}) and \( \Omega_+ \) for
positive eigenvalues (right column of \cref{fig:1d:eigfct}).
The eigenvalue closest to zero (from which \( \mathcal{C}_0 \) emanates) plays a special role (middle column of \cref{fig:1d:eigfct}).
Secondly, with increasing \( |\lambda| \), the number of maxima and minima increases as one observes also for
the eigenfunctions of the Laplacian.
Thirdly, the transmission condition at \( \Gamma \) requires the (normal) derivative of \( u \) to change
sign, such that the solutions have a ``tip'' at the interface.
Taking a closer look at the bifurcation values and the corresponding solutions in \cref{fig:1d:eigfct}, we
note that \( \mathcal{C}_0 \) starts much closer to zero for \( \sigma_-=-1.005 \) than for \( \sigma_-=-2 \).
This illustrates the theoretical expectation that due to the symmetry of the domain \( \Omega \), we have an
eigenvalue approaching zero for the contrast going to \( -1 \).
Moreover, we observe a certain shrinking of the negative bifurcation values towards zero when the contrast
approaches \( -1 \).

\subsubsection{Patterns of solutions along branches}

We now take a closer look at how solutions evolve along branches --- depending on whether the corresponding bifurcation value is negative, close to zero or positive.
According to the previous discussion, we focus on \( \sigma_-=-1.005 \) in the following because it shows the phenomena in a particularly pronounced form and is close to the interesting ``critical'' contrast of \( -1 \).
In general, we observe that a certain limit pattern or profile of the solution evolves on each branch which remains qualitatively stable (values of maxima, minima and plateaus of course change with \( \lambda \)).
As example for a  negatively indexed bifurcation branch away from zero, we consider \( \mathcal{C}_{-2} \), \emph{cf.} \cref{fig:1d:bif}.
The first, \( 50 \)th, and \( 100 \)th solution on the branch are depicted in \cref{fig:1d:solb3}. As described above, the solution concentrates in \( \Omega_- \) where it oscillates, while it decays exponentially in \( \Omega_+ \).
This profile remains stable over the branch, but we note that the maxima and minima become wider along the branch.
This widening of the extrema in \( \Omega_- \) is also noted for the other branches emanating from a negative bifurcation point. Yet, the more oscillations occur for the branches as \( \lambda \to -\infty \), the less pronounced the effect becomes because we have more extrema over the same interval.
We emphasize that this effect of widening extrema is specific to the sign-changing case and especially to bifurcations starting at negative \( \lambda \).
\begin{figure}[hbtp]
  \includegraphics[width=0.31\textwidth, trim=19mm 80mm 24mm 85mm, clip=true]{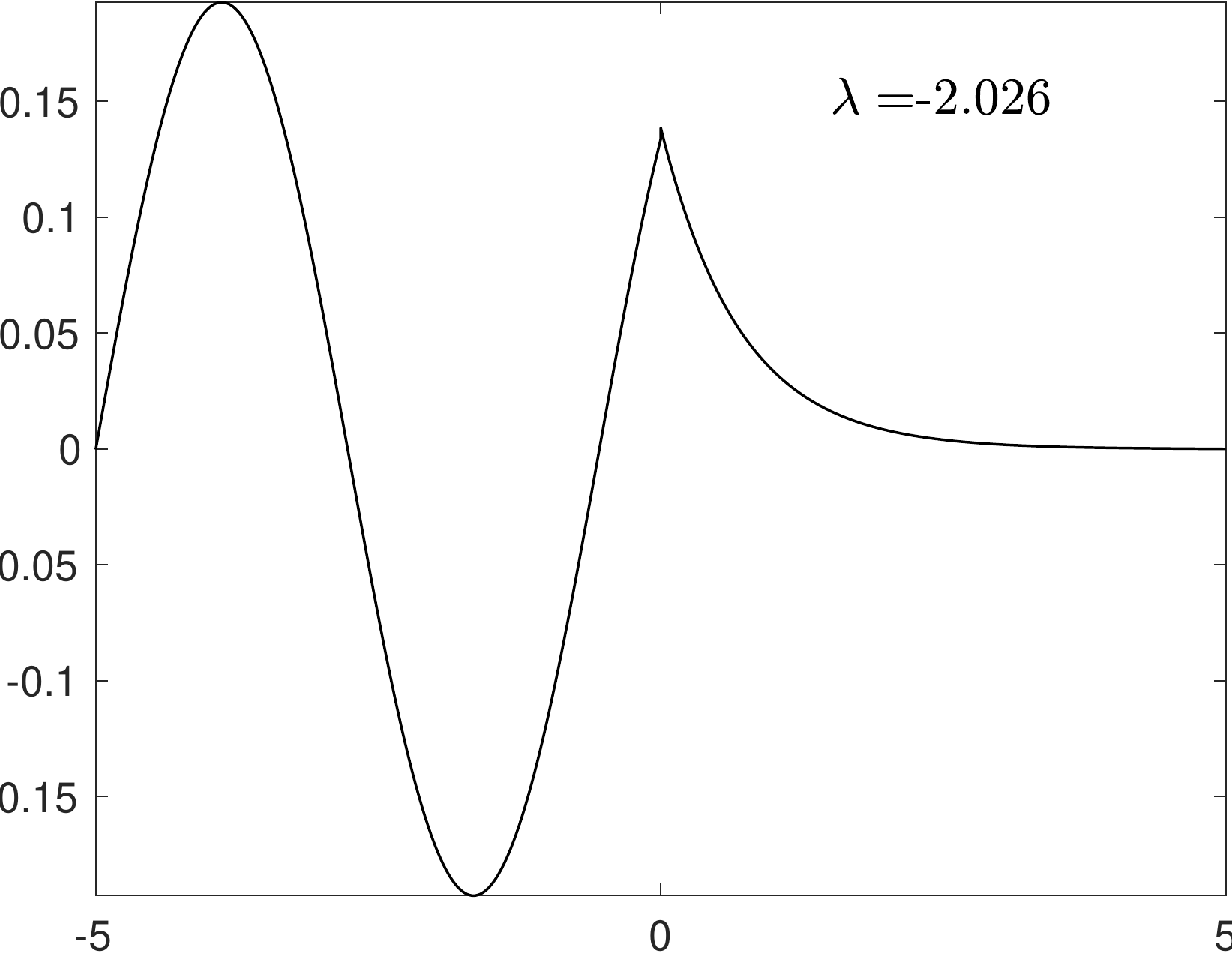}%
  \includegraphics[width=0.31\textwidth, trim=20mm 80mm 24mm 85mm, clip=true]{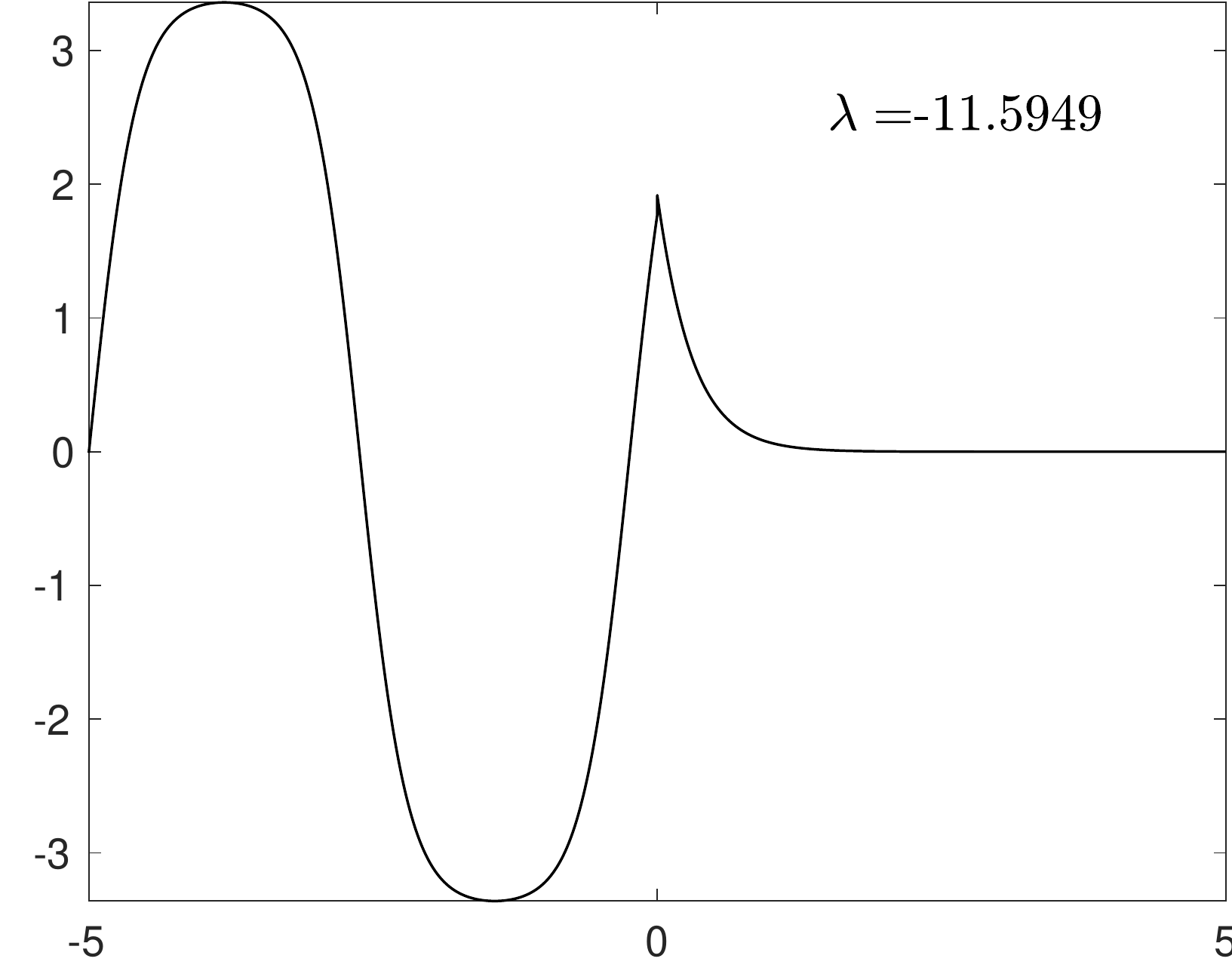}%
  \includegraphics[width=0.31\textwidth, trim=20mm 80mm 24mm 85mm, clip=true]{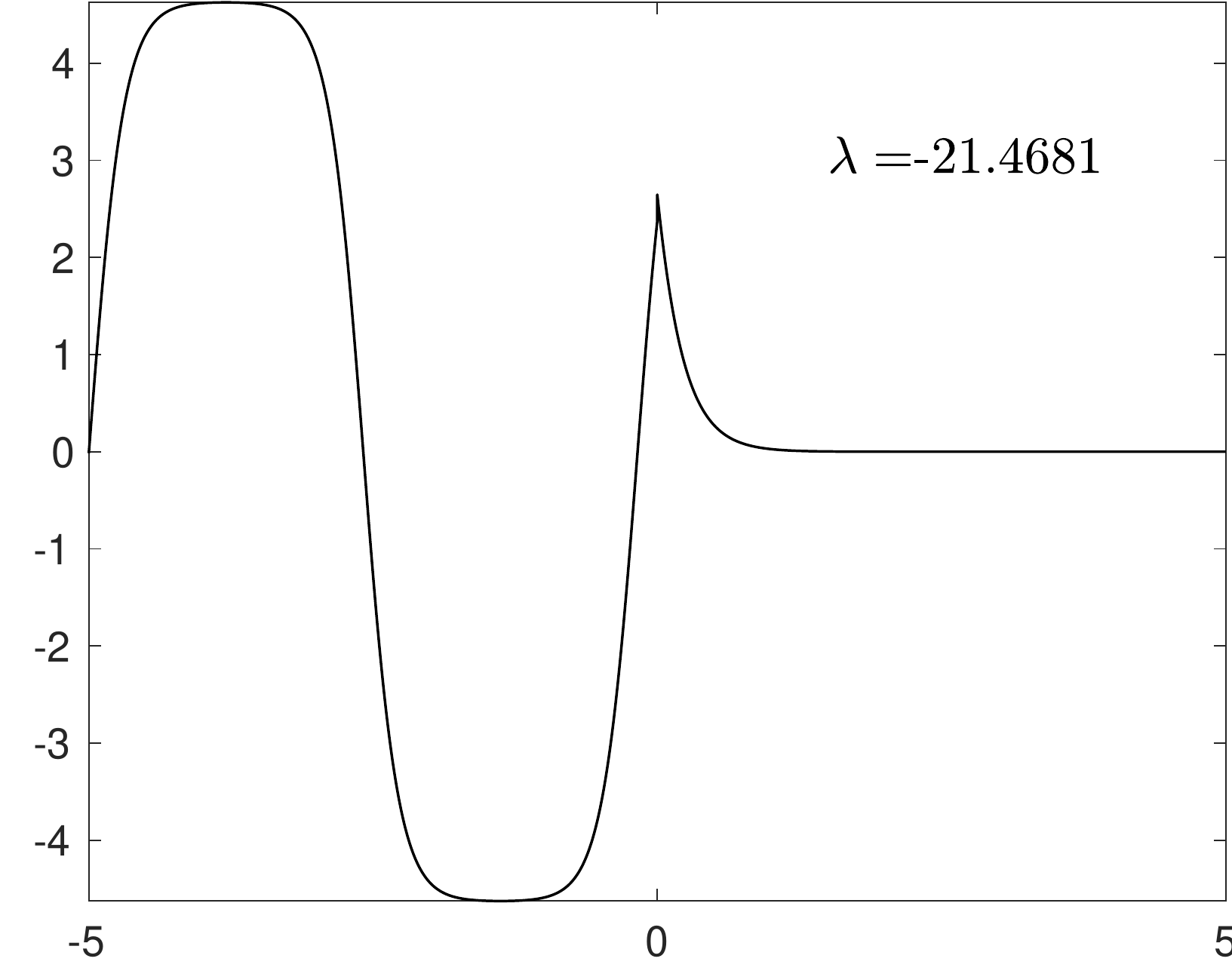}%
  \caption{Solution at first, \( 50 \)th, \( 100 \)th point of \( \mathcal{C}_{-2} \) for
    \( \sigma_-=-1.005 \).}%
  \label{fig:1d:solb3}
\end{figure}

As an example for a  positively indexed  bifurcation branch away from zero, we study the
branch \( \mathcal{C}_5 \), cf. \cref{fig:1d:bif}.
As expected, we observe in \cref{fig:1d:solb10} that the first solution concentrates on \( \Omega_+ \), where it oscillates as typical for an eigenfunction of the Laplacian, and shows an exponential decay in \( \Omega_- \).
The oscillatory pattern in \( \Omega_+ \) is preserved along the branch. The behavior in \( \Omega_- \) changes when \( \lambda \) gets negative: Instead of an exponential decay to zero, we now see an exponential decay to (almost) a plateau (with value \( \pm\sqrt{-\lambda} \)) and a sharp transition to the zero boundary value.
Once this pattern is established, it remains stable as well.
This appearance of a plateau different from zero is also a specific phenomenon of the sign-changing case.
\begin{figure}[hbtp]
  \includegraphics[width=0.23\textwidth, trim=19mm 78mm 24mm 85mm, clip=true]{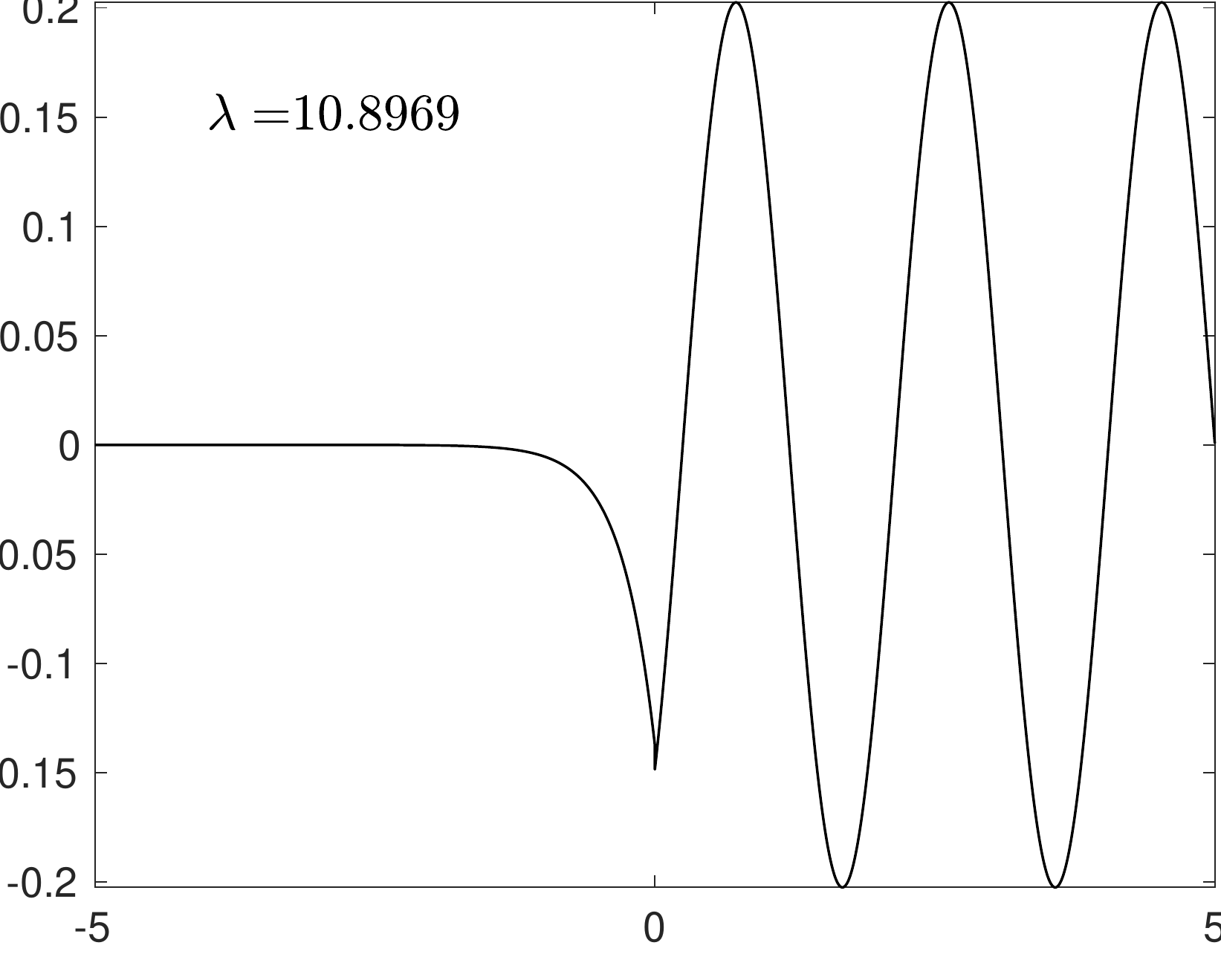}%
  \hspace{2ex}%
  \includegraphics[width=0.23\textwidth, trim=20mm 80mm 24mm 85mm, clip=true]{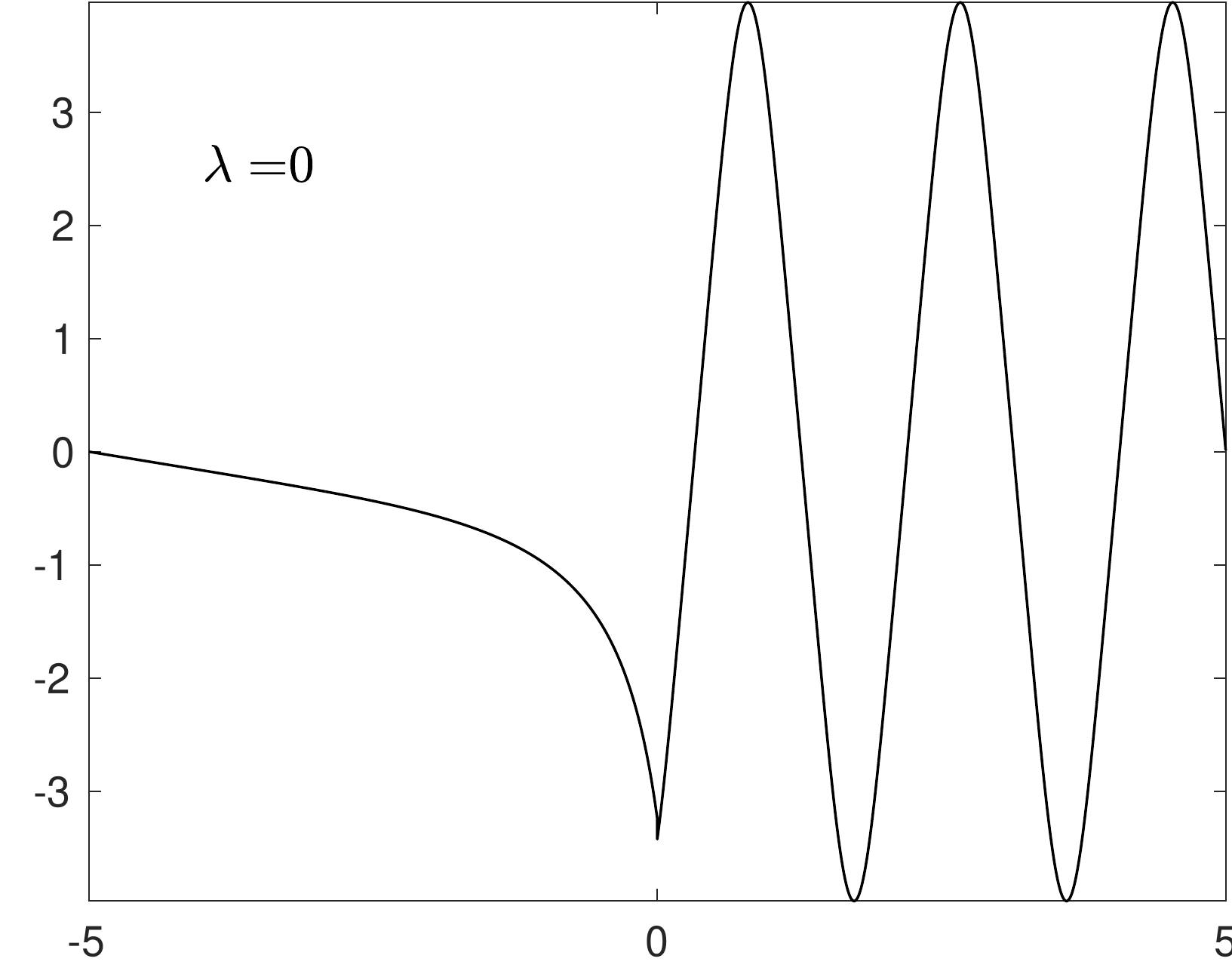}%
  \hspace{2ex}%
  \includegraphics[width=0.23\textwidth, trim=20mm 80mm 24mm 85mm, clip=true]{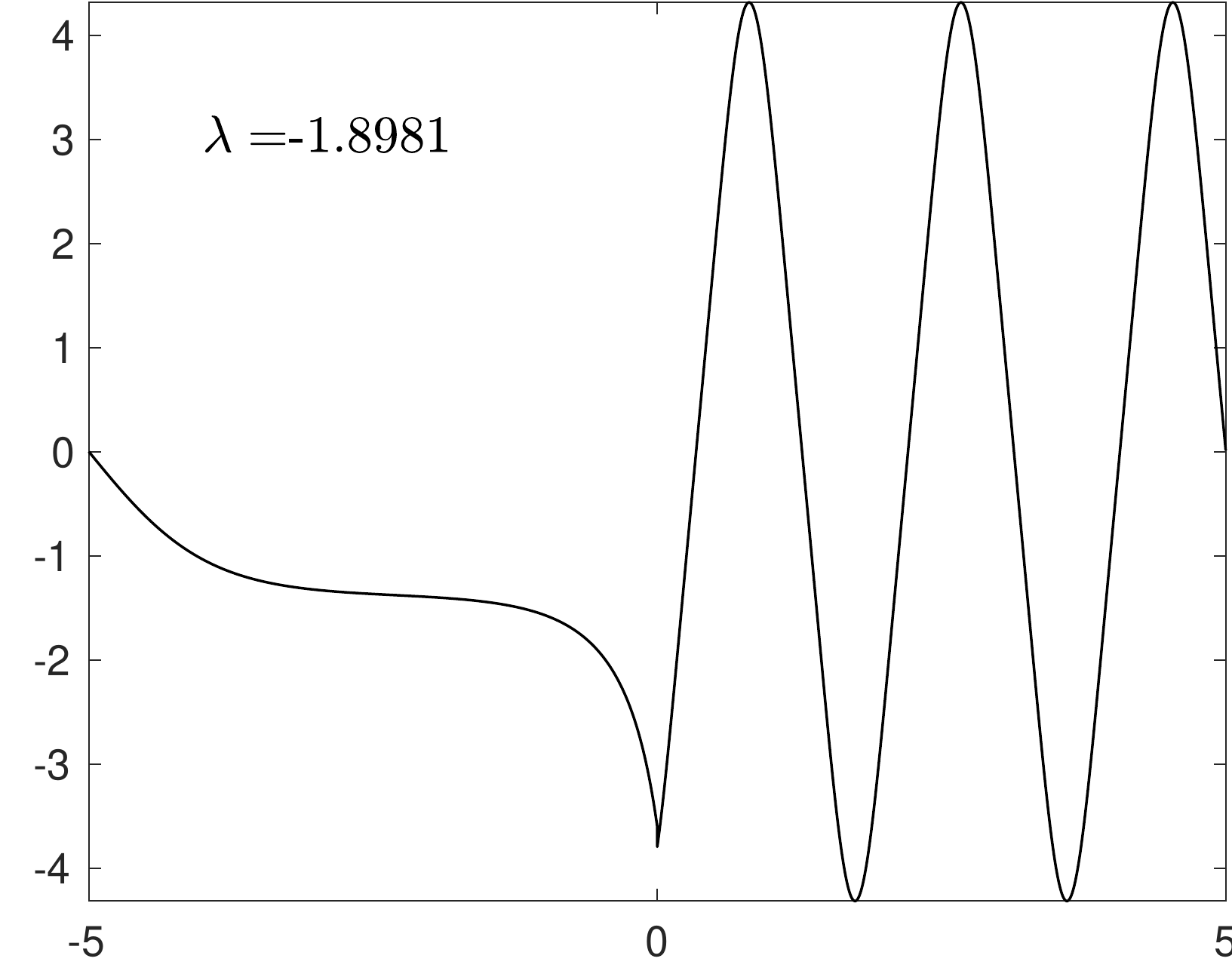}%
  \hspace{2ex}%
  \includegraphics[width=0.23\textwidth, trim=20mm 80mm 24mm 85mm, clip=true]{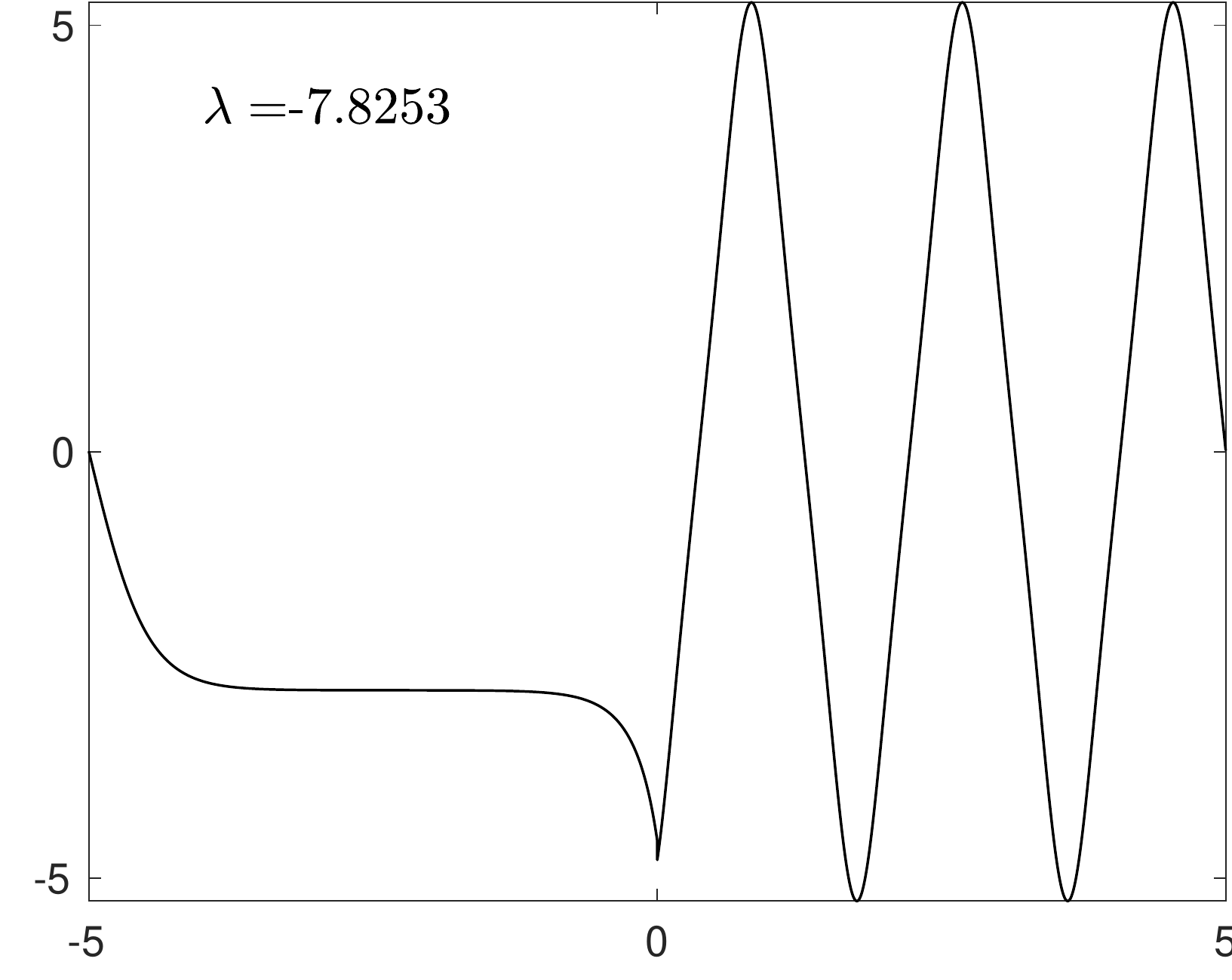}%
  \caption{Solution at first, \( 60 \)th, \( 70 \)th, \( 100 \)th point of \( \mathcal{C}_5 \) for
    \( \sigma_-=-1.005 \).}%
  \label{fig:1d:solb10}
\end{figure}

The occurrence of a plateau in \( \Omega_- \) is also observed in \cref{fig:1d:solb5} for the branch \( \mathcal{C}_0 \) closest to zero, cf. \cref{fig:1d:bif}.
While the first solution has a similar shape in \( \Omega_+ \) and \( \Omega_- \) with a linear decay in each
subdomain, the ensuing solutions on the branch quickly evolve a plateau in \( \Omega_- \) and an exponential
decay in \( \Omega_+ \). This pattern then remains stable along the branch.  
All in all, we observe a certain stability of profiles along branches. The form of the profiles depends on
where the bifurcation starts. Moreover, we always recognize a concentration to the oscillatory part
and further the establishment of plateaus different from zero in \( \Omega_- \). As already emphasized, both
effects are specific to the sign-changing case.
This qualitative description of solutions seems to transfer to other contrasts, but the
bifurcation points closest to zero  and the (quantitative) decay in \( \Omega_{\pm} \) significantly depend  on
the contrast as already discussed above.
\begin{figure}[hbtp]
  \includegraphics[width=0.31\textwidth, trim=19mm 80mm 24mm 85mm, clip=true]{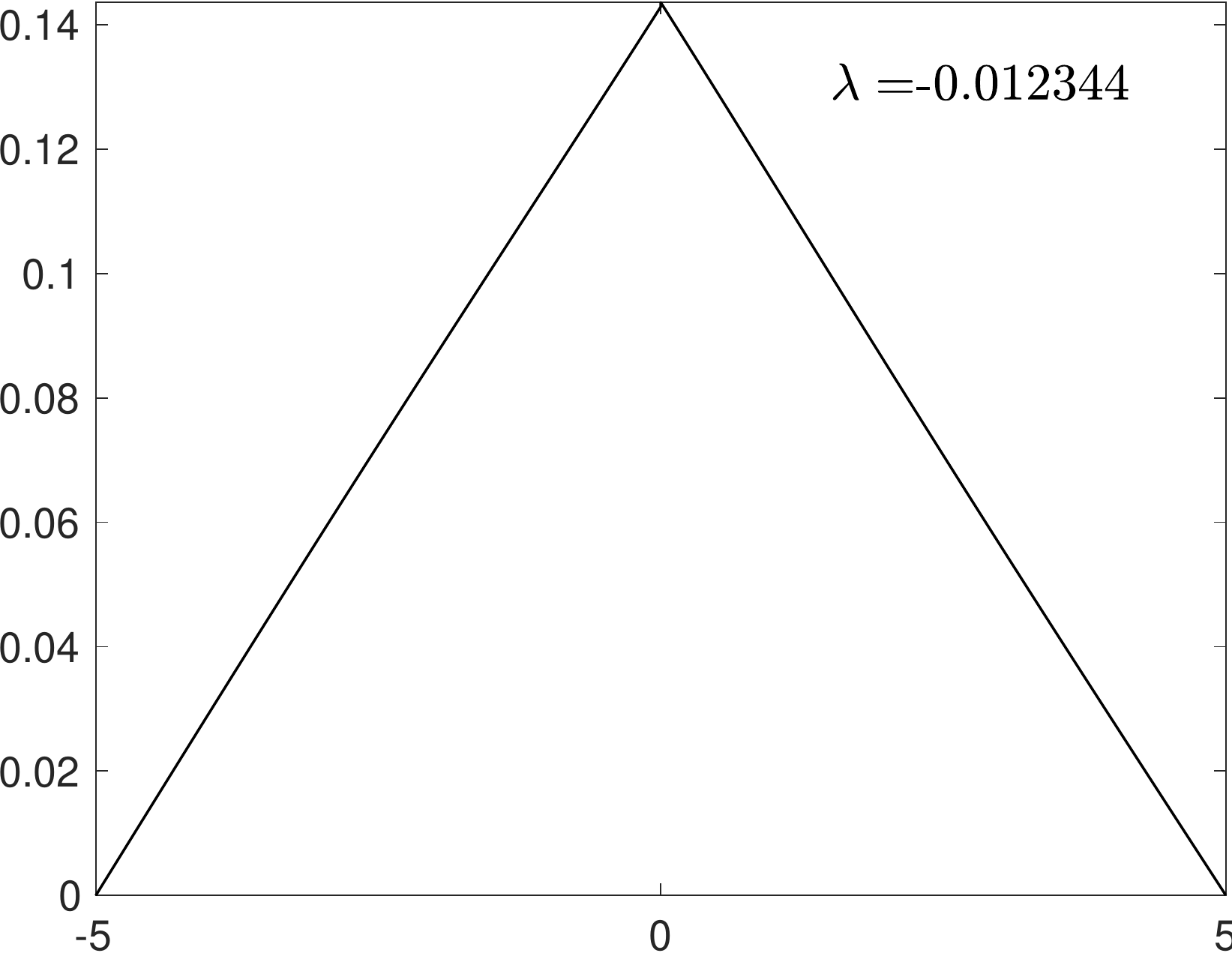}%
  \hspace{2ex}%
  \includegraphics[width=0.31\textwidth, trim=20mm 80mm 24mm 85mm, clip=true]{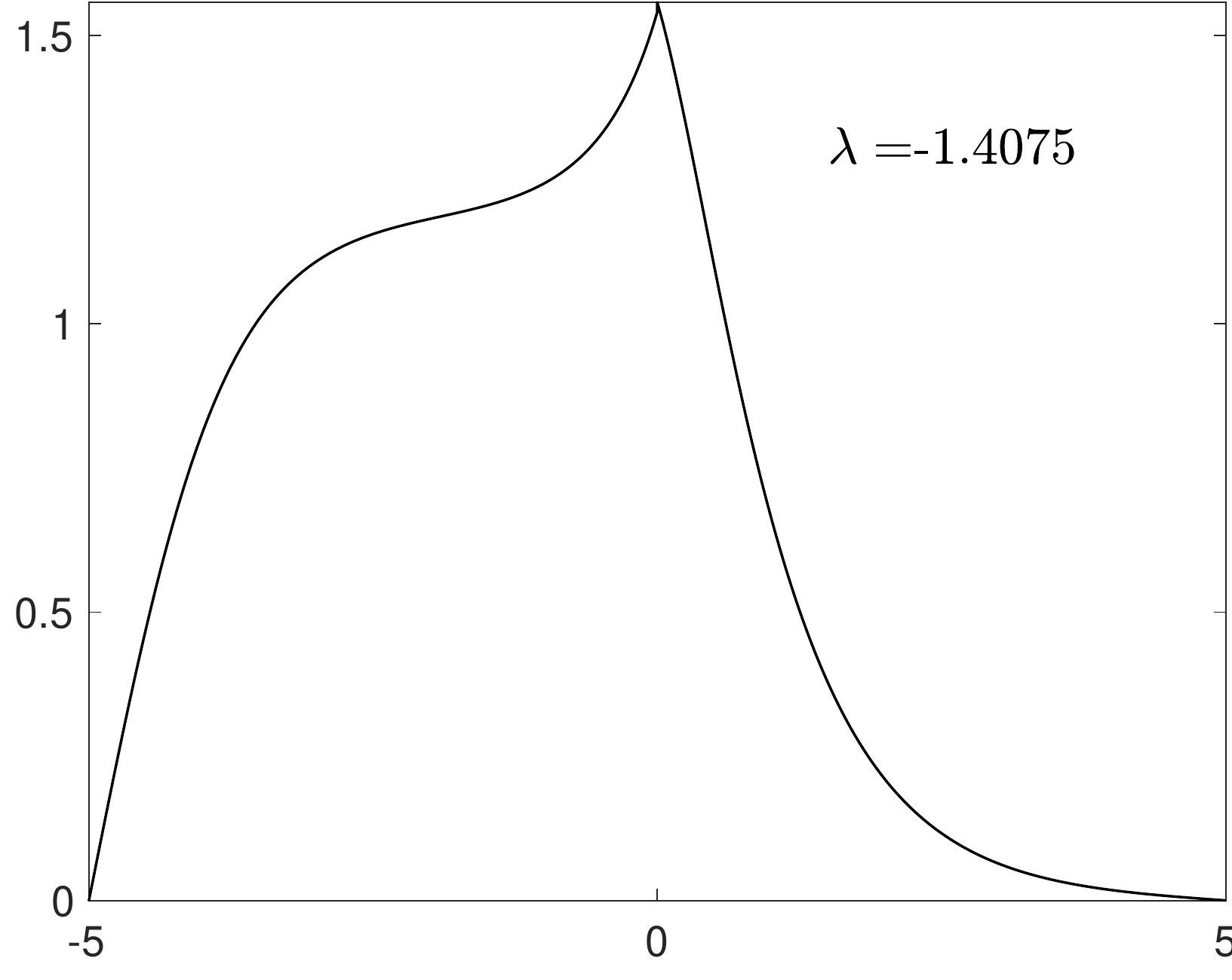}%
  \hspace{2ex}%
  \includegraphics[width=0.31\textwidth, trim=20mm 80mm 24mm 85mm, clip=true]{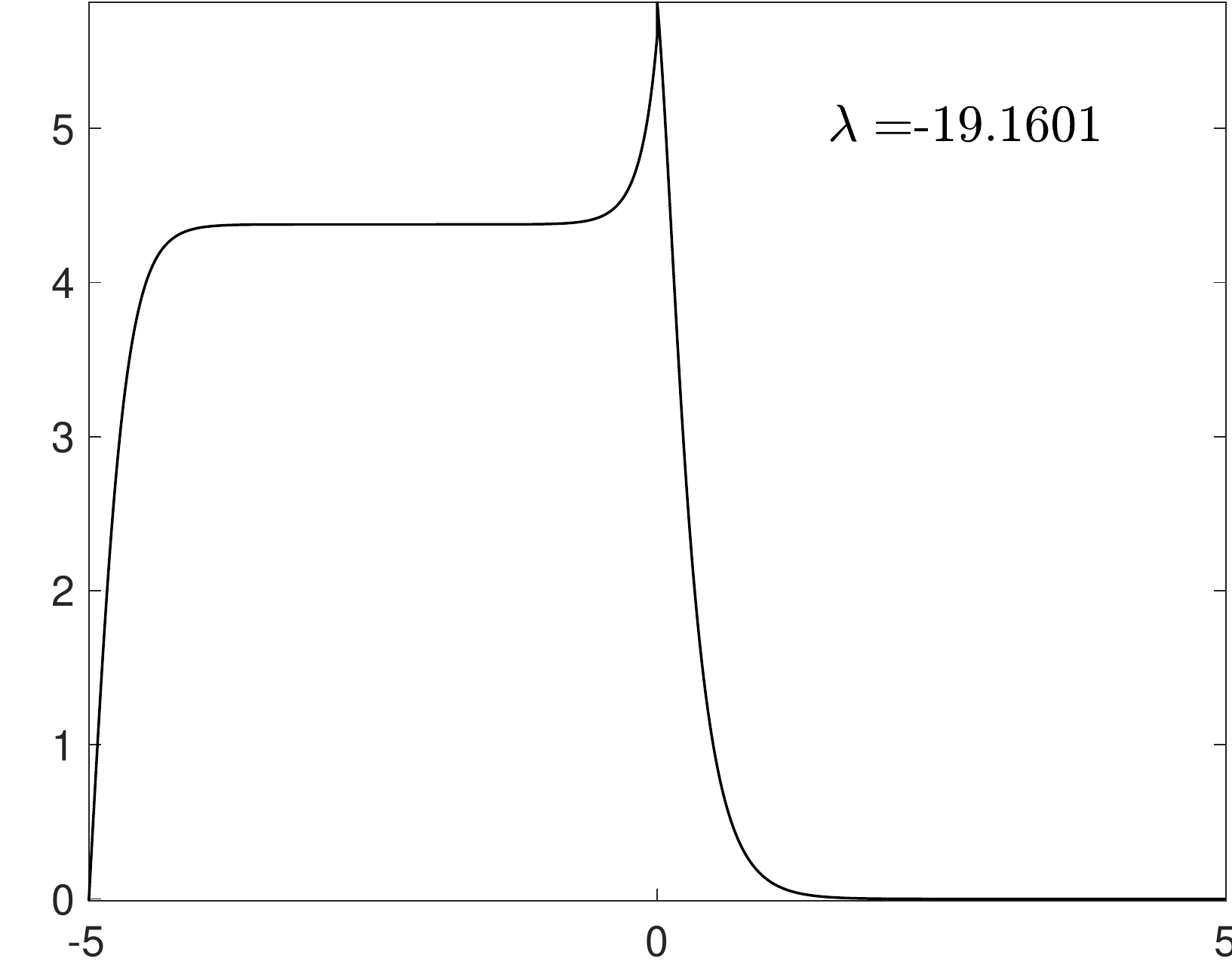}%
  \caption{Solution at first, \( 10 \)th,  \( 100 \)th point of \( \mathcal{C}_0 \) for
    \( \sigma_-=-1.005 \).}%
  \label{fig:1d:solb5}
\end{figure}

\subsection{Two-dimensional example}

We consider \( \Omega={(-2,2)}^2 \) with \( \Omega_-=(-2,0)\times (-2,2) \), as well
as \( \sigma_+=1 \) and \( \sigma_-=-2 \).
The finite element mesh is tailored similar to the one-dimensional experiment: We start with a symmetric uniform mesh with \( h=2^{-4} \) and refine three times all elements in the strip of width \( 0.1 \) around the interface \( \Gamma = \{0\}\times(-2, 2) \).
Note that our mesh satisfies the symmetry conditions laid out in~\cite{BBDhia_MeshRequirements}, which may be challenging in more complicated geometries though.

We focus on the behavior of solutions in this numerical experiment and let \( \lambda \) vary in \( [-12, 15] \).
There are three different types of eigenfunctions either concentrated on \( \Omega_- \), on \( \Gamma \), or on \( \Omega_+ \).
In contrast to the one-dimensional case, there are several different eigenfunctions concentrated on \( \Gamma \).
As before, the eigenfunctions concentrated on \( \Omega_- \) or \( \Gamma \) are associated with negative values of
\( \lambda \).
In \cref{fig:2d:Omegaminus,fig:2d:Gamma,fig:2d:Omegaplus}, we show the evolution of solutions along a branch for each of the three types described above.

\begin{figure}[hbtp]
  \includegraphics[width=0.31\textwidth, trim=12mm 75mm 22mm 83mm, clip=true]{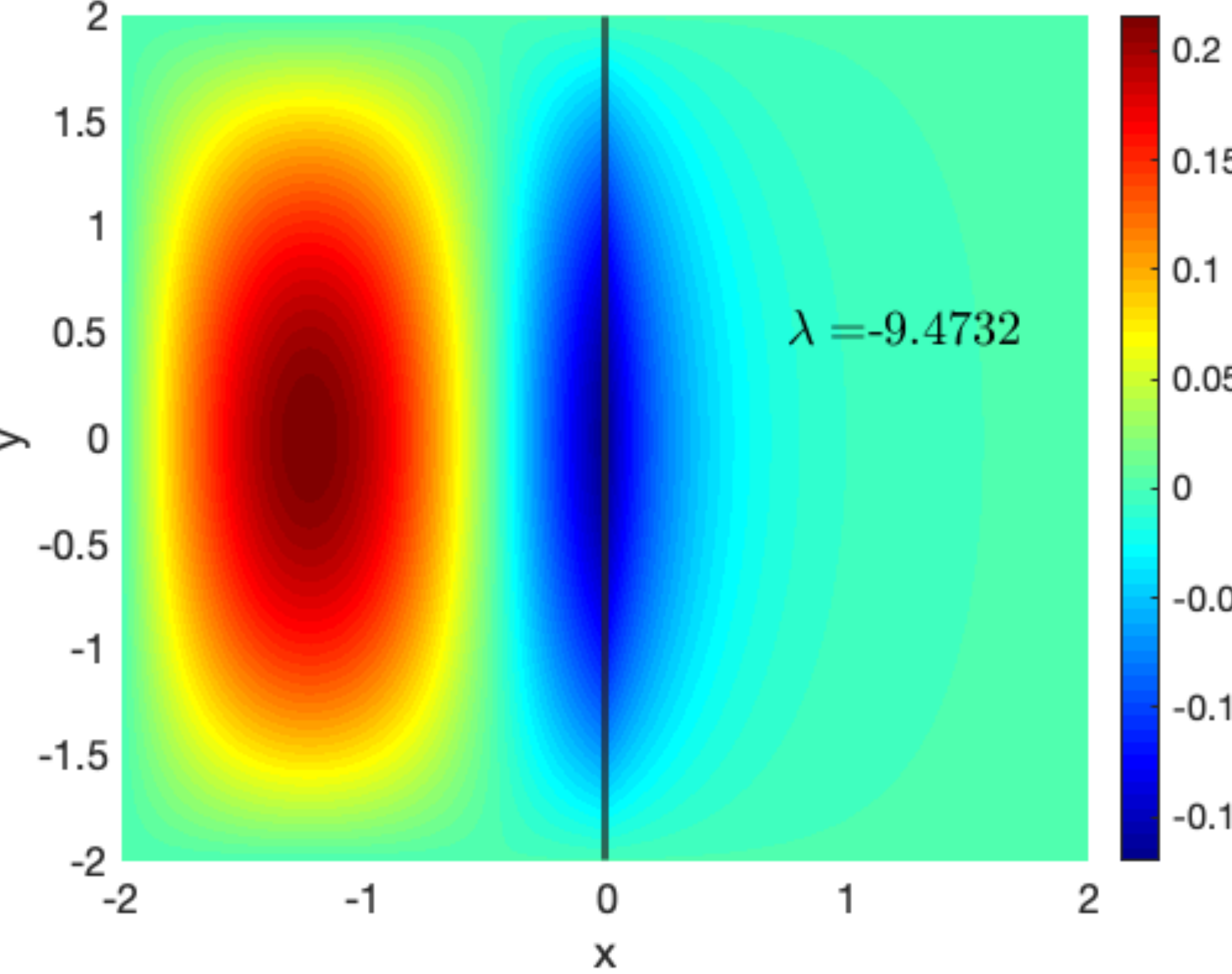}%
  \hspace{2ex}%
  \includegraphics[width=0.31\textwidth, trim=12mm 75mm 22mm 83mm, clip=true]{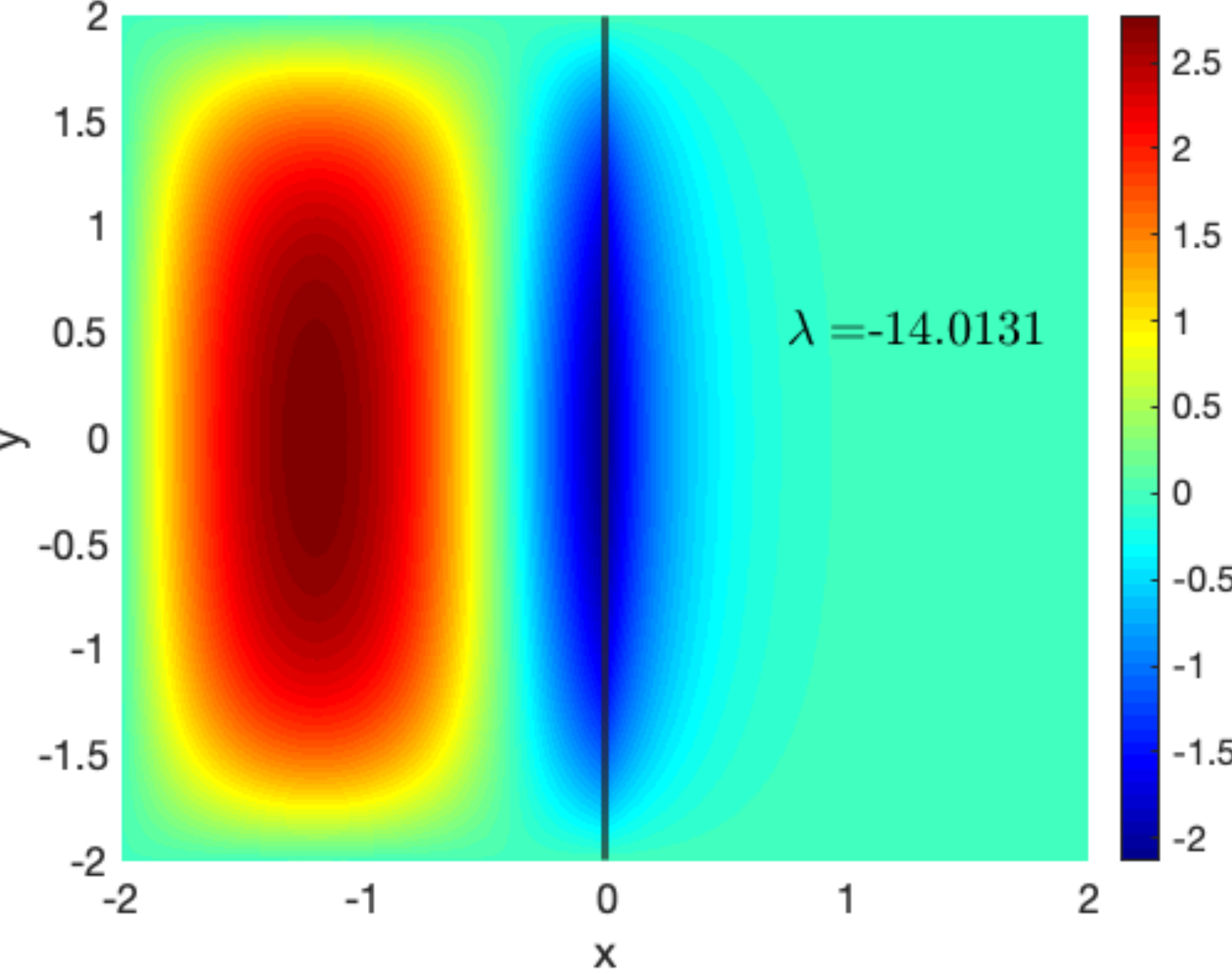}%
  \hspace{2ex}%
  \includegraphics[width=0.31\textwidth, trim=12mm 75mm 22mm 83mm, clip=true]{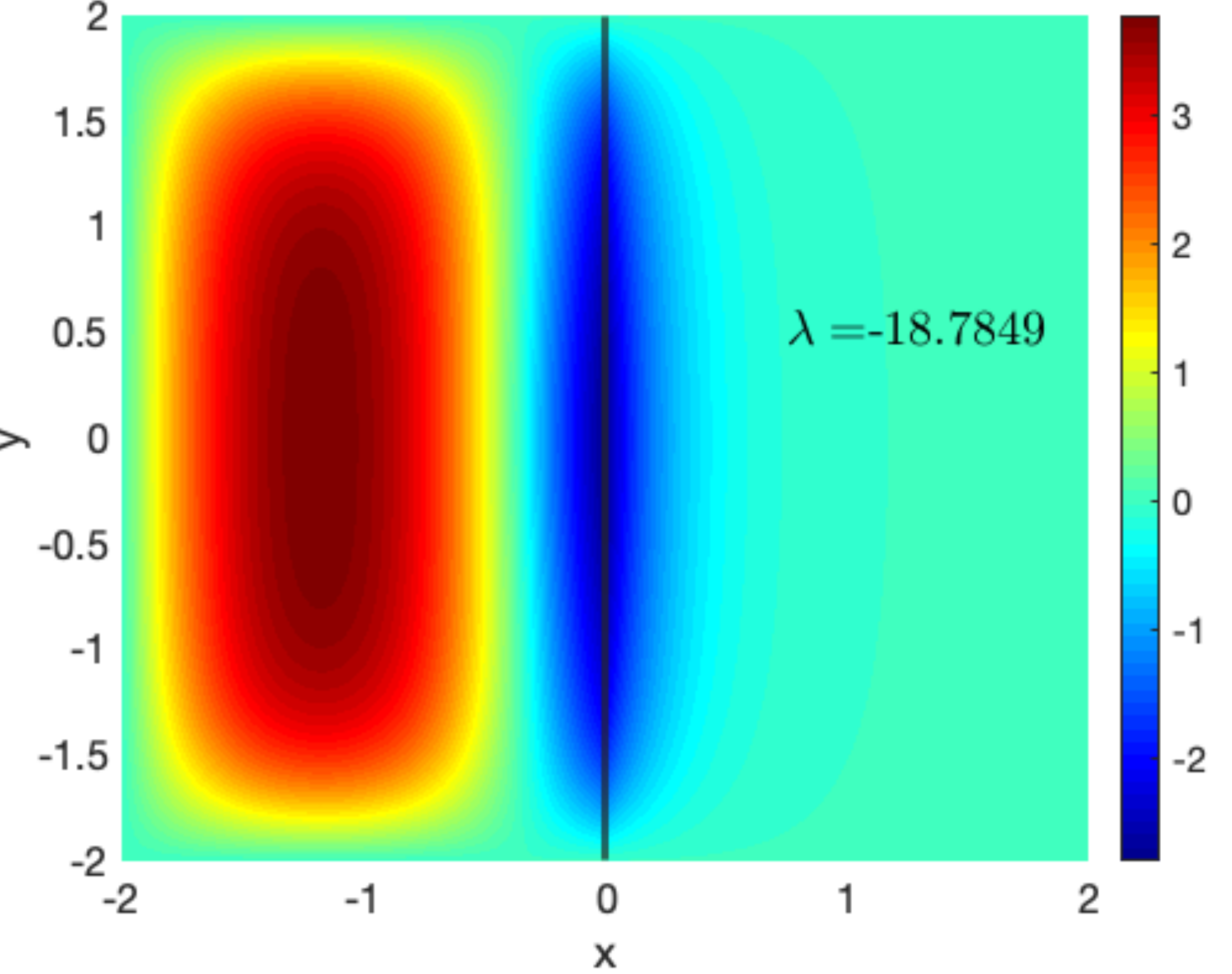}%
  \caption{Solution at first, \( 25 \)th, and \( 50 \)th point of branch associated with an eigenfunction concentrated on \( \Omega_- \).}%
  \label{fig:2d:Omegaminus}
\end{figure}

Similar to the one-dimensional case, we observe a widening of the extrema along the branch with concentration in \( \Omega_- \) in \cref{fig:2d:Omegaminus}, in particular in the \( y \)-direction.

\begin{figure}[hbtp]
  \includegraphics[width=0.31\textwidth, trim=12mm 75mm 22mm 83mm, clip=true]{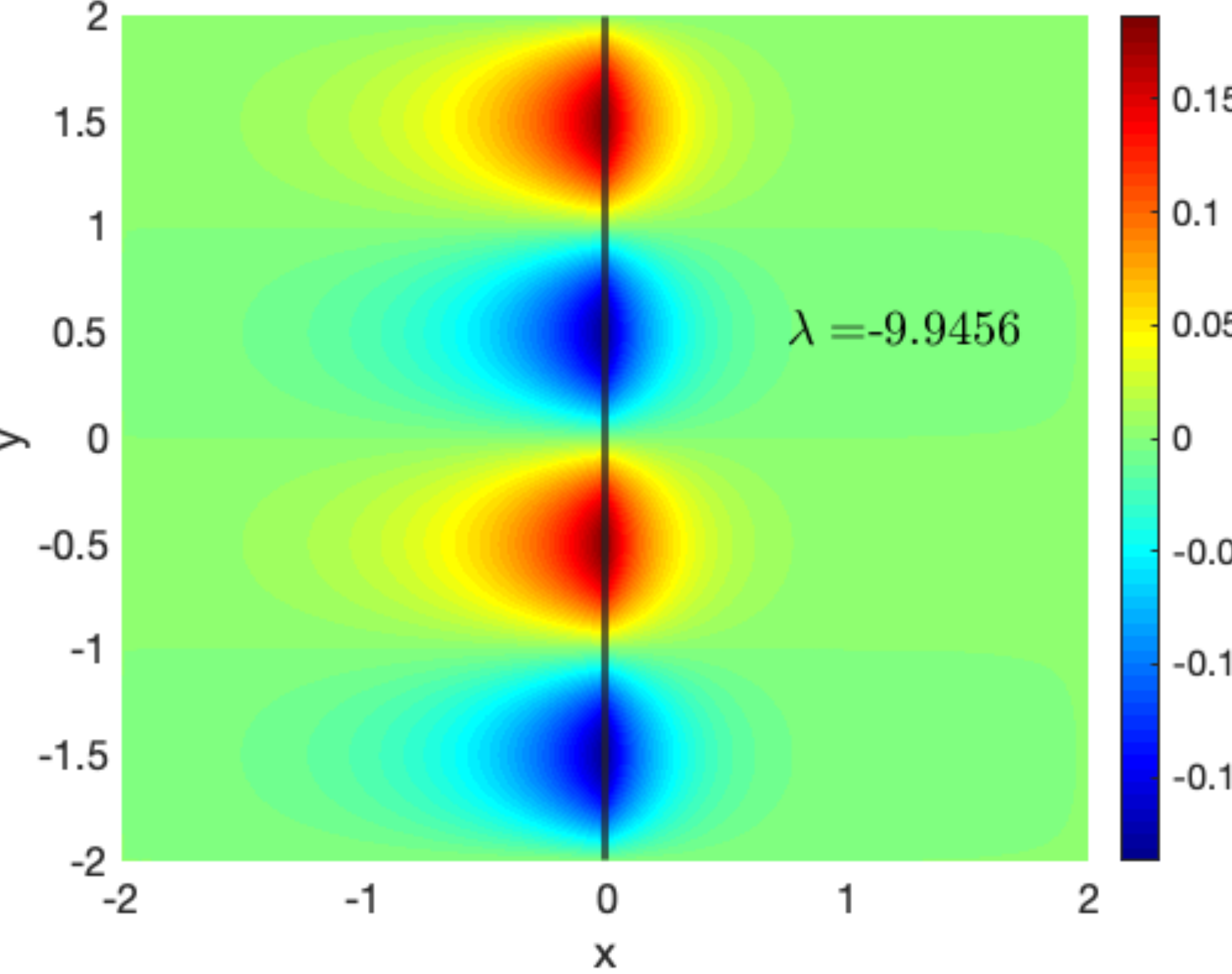}%
  \hspace{2ex}%
  \includegraphics[width=0.31\textwidth, trim=12mm 75mm 22mm 83mm, clip=true]{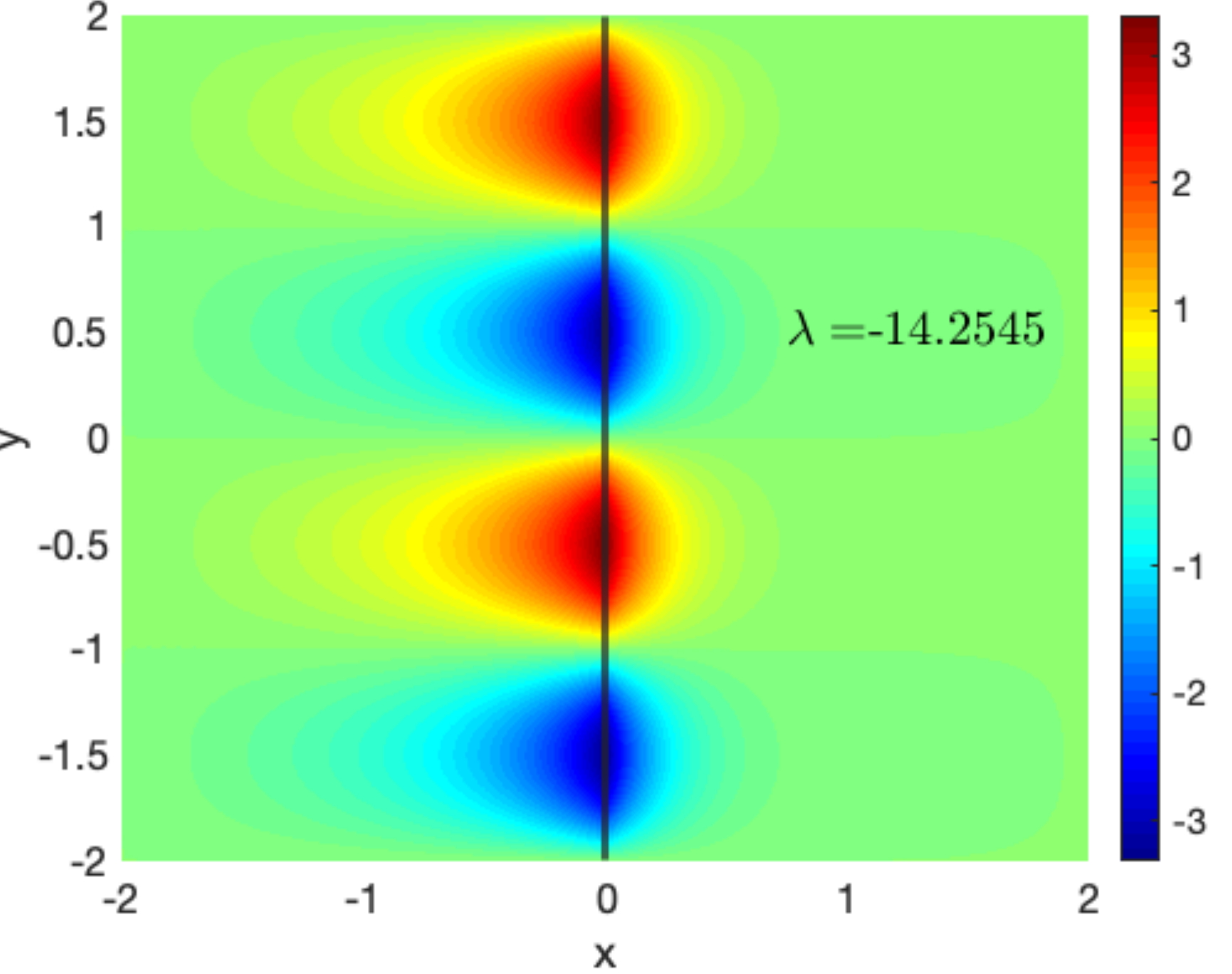}%
  \hspace{2ex}%
  \includegraphics[width=0.31\textwidth, trim=12mm 75mm 22mm 83mm, clip=true]{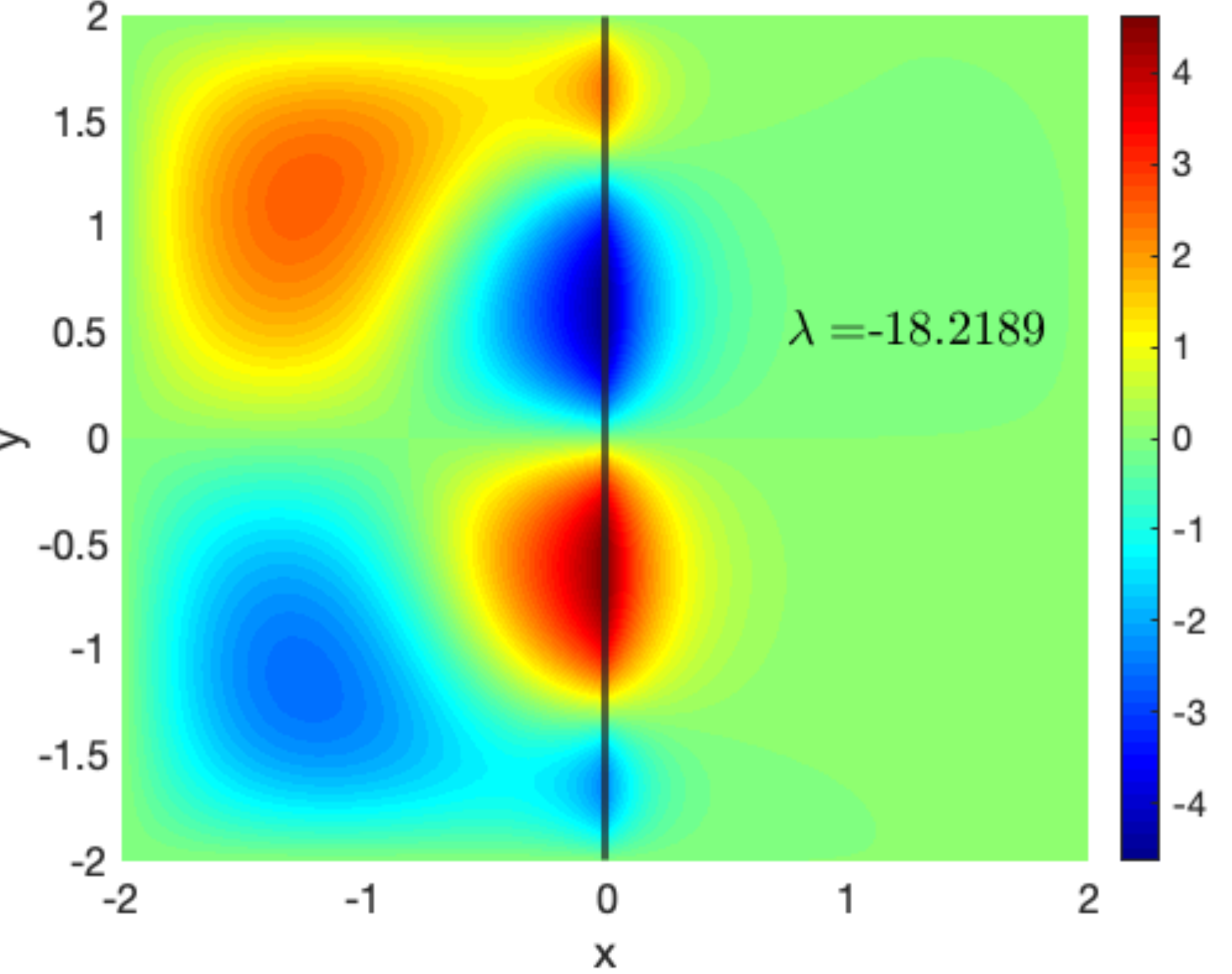}%
  \caption{Solution at first, \( 25 \)th, and \( 50 \)th point of branch associated with an eigenfunction concentrated on \( \Gamma \).}%
  \label{fig:2d:Gamma}
\end{figure}

Furthermore, plateaus in \( \Omega_- \) evolve for negative \( \lambda \) in \cref{fig:2d:Gamma,fig:2d:Omegaplus}.
Due to the second space dimension in the problem, we can have two (or more) different plateaus evolving in \( \Omega_- \).
In \cref{fig:2d:Gamma} for a branch with concentration on \( \Gamma \), we note that the plateaus and the transition between them seems to slightly change the oscillatory pattern on \( \Gamma \) as well.
While the two maxima have almost the same height for the first and \( 25 \)th point (\cref{fig:2d:Gamma} left and middle), one maximum becomes predominant for the \( 50 \)th point on the branch, see \cref{fig:2d:Gamma} right.

\begin{figure}[hbtp]
  \includegraphics[width=0.31\textwidth, trim=12mm 75mm 22mm 83mm, clip=true]{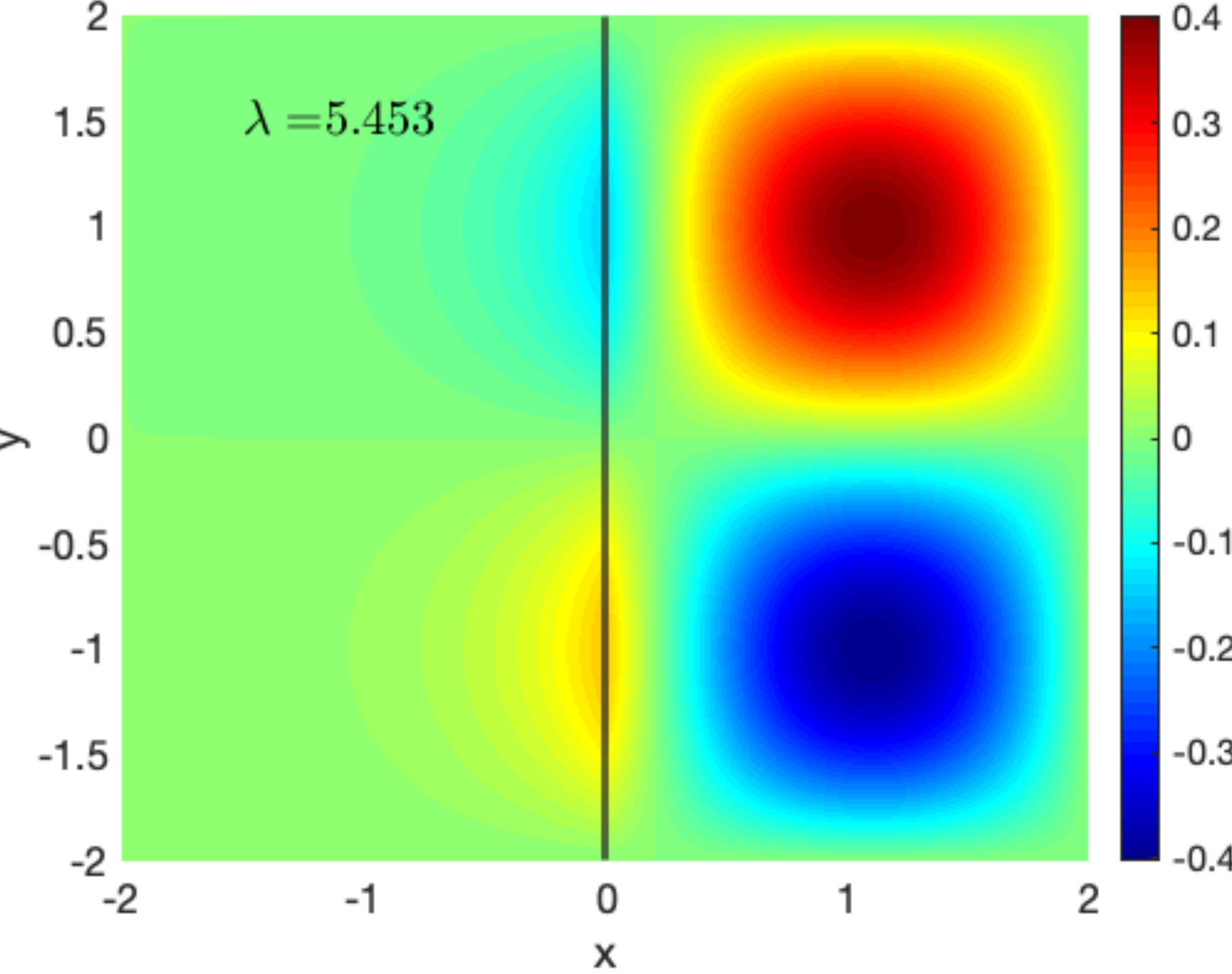}%
  \hspace{2ex}%
  \includegraphics[width=0.31\textwidth, trim=12mm 75mm 22mm 83mm, clip=true]{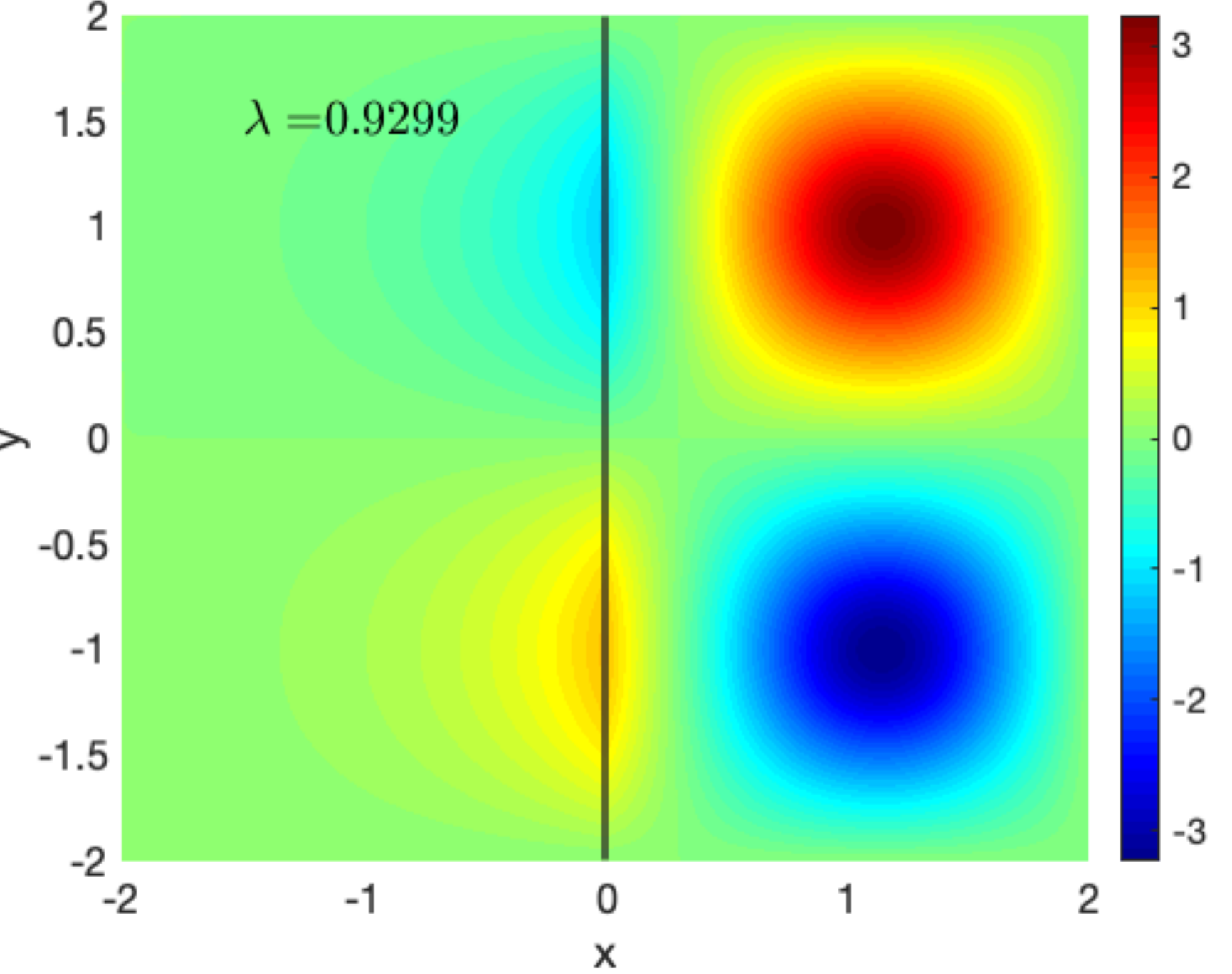}%
  \hspace{2ex}%
  \includegraphics[width=0.31\textwidth, trim=12mm 75mm 22mm 83mm, clip=true]{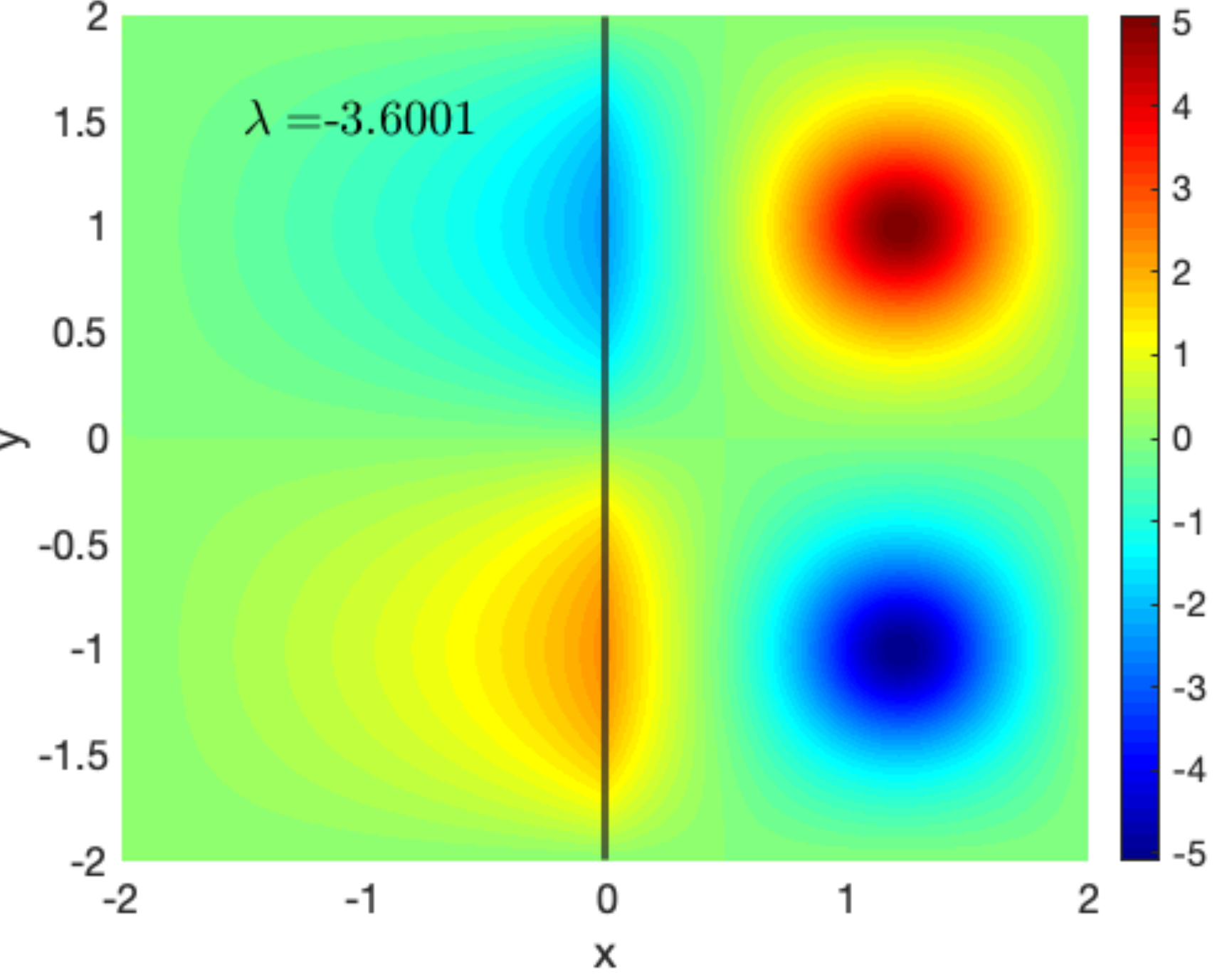}%
  \caption{Solution at first, \( 25 \)th, and \( 50 \)th point of branch associated with an eigenfunction concentrated on \( \Omega_+ \).}%
  \label{fig:2d:Omegaplus}
\end{figure}
Finally, for \cref{fig:2d:Omegaplus} and a branch with concentration on \( \Omega_+ \), we emphasize that the solution in \( \Omega_+ \) evolves like a solution of the standard Laplacian along a branch.
In particular, the extrema become thinner, \emph{i.e.}, more spatially localized, which should be contrasted with solutions concentrated in \( \Omega_- \) in \cref{fig:2d:Omegaminus}.

\section{Linear Theory}\label{sec03}

In this section we want to describe the linear theory for weakly \( \ttT \)-coercive problems.
As pointed out earlier, this theory is essentially well-known~\cite{BBDhiaZwoelf_Timeharmonic, BonCheCia_TCoercivity,CarCheCia_Eigenvalue}.
Since it is short and rather self-contained, we provide the details here, which will moreover allow us to fix the required notation.
Furthermore, we prove some Weyl law asymptotics that have not appeared in the literature yet.
We want to deal with linear problems of the   form
\begin{equation}\label{eq:problem}
  \int_\Omega \sigma(x) \, \nabla u \cdot \nabla v \di{x}
  - \lambda \int_\Omega c(x) \, u \, v \di{x}
  = F(v),
  \qquad \forall v \in \rmH_0^1(\Omega).
\end{equation}
The a priori unknown solution \( u \) is to be found in the Sobolev space \( \rmH_0^1(\Omega) \) and the
coefficient functions \( \sigma \) and \( c \) are assumed to satisfy the conditions~\ref{hyp:sigma_c}
and~\ref{hyp:w_T_cor} from \cref{sec:intro}. To develop a solution theory for the variational problem
\cref{eq:problem} both in \( \rmH_0^1(\Omega) \) and \( \rmL^2(\Omega) \) we assume \( F\in
\rmH^{-1}(\Omega) = \rmH_0^1(\Omega)' \). We may rewrite \cref{eq:problem} as
\begin{equation*}%\label{eq:problemweak}
  a(u,v) - \lambda \skpL{u}{v} = F(v) \qquad \forall v \in \rmH_0^1(\Omega).
\end{equation*}
We introduce the bounded linear operator \( A : \rmH_0^1(\Omega) \to \rmH_0^1(\Omega) \) and the compact
linear operator \( C : \rmH_0^1(\Omega) \to \rmH_0^1(\Omega) \) via  the relations
\begin{equation}\label{eq:def_op_AC}
  \skpH{Au}{v} \coloneqq a(u,v)
  \quad \text{and} \quad
  \skpH{Cu}{v} \coloneqq \skpL{u}{v}
  \qquad \text{for } u, v \in \rmH_0^1(\Omega).
\end{equation}
This is possible by Riesz' Representation Theorem and \( \sigma, c \in  \rmL^\infty(\Omega) \), see \cref{hyp:sigma_c}.
We will also use that the operator \( C \) can be written as \( C= \tilde C \iota \) where \( \tilde{C} : \rmL^2(\Omega) \to \rmH_0^1(\Omega) \) is a bounded linear operator and \( \iota : \rmH_0^1(\Omega)\to \rmL^2(\Omega) \) denotes the embedding operator
which is compact by the Rellich-Kondrachov Theorem.

\begin{prop}\label{prop:Aproperties}
  Under \cref{hyp:sigma_c,hyp:w_T_cor}, there exists \( \las \in \R \)
  such that the bounded linear operator \( \As \coloneqq A + \las C : \rmH_0^1(\Omega) \to \rmH_0^1(\Omega) \)
  is self-adjoint and invertible.
\end{prop}
%------------
\begin{proof}
  The self-adjointness follows from \( a(u, v) = a(v, u) \) and
  \[
    \skpH{C u}{v}
    = \skpL{u}{v}
    = \skpL{v}{u}
    = \skpH{Cv}{u}
    = \skpH{u}{C v}
  \]
  for all \( u, v \in \rmH_0^1(\Omega) \). To prove the invertibility of \( \As \) define the family of operators
  \( z \mapsto A_z \coloneqq A + z C \) for \( z \in \C \) on the complex Hilbert space \( \rmH_0^1(\Omega; \C) \). The
  bilinear form associated with \( A_z \) is given by \( (u, v) \mapsto a(u, v) + z \skpL{u}{v} \).
  From \cref{hyp:w_T_cor} and the Lax-Milgram Lemma we infer that \( \ttT^* A + \ttK \) is invertible.
  Moreover, we have the relation
  \[
    A_z = \plr{\ttT^*}^{-1} \clr{\ttT^* A + \ttK} - \plr{\ttT^*}^{-1} \ttK + z C.
  \]
  Here, the first summand is invertible while the other two summands are compact.
  Therefore, \( \setwt{A_z}{z \in \C} \) is a holomorphic family of zero index  Fredholm operators.
  For \( z \in \C \sm \R \), the operator \( A_z \) is injective. Indeed, if \( A_z u = 0 \) then \( \skpH{A_z u}{\overline{u}} = 0 \) and
  \[
    0
    = \Im\plr{\skpH{A_z u}{\overline{u}}}
    = \Im\plr{a(u, \overline{u}) + z \norm{u}_c^2}
    = \Im(z) \norm{u}_c^2.
  \]
  So, we have \( \ker(A_z) = \{0\} \), which implies that  \( A_z \) has a bounded inverse as an injective
  Fredholm operator.
  Using the analytic Fredholm theorem on \( \C \), see~\cite[Theorem C.8]{DyaZwo19}, the set \( \setwt{A_z^{-1}}{z \in \C} \) is a meromorphic family of operators with poles of finite rank.
  Therefore, the operator \( {(A + z C)}^{-1} \) exists for all \( z \in \C \sm \Lambda \) for a discrete set
  \( \Lambda \subset \R \).
  In particular, there exists \( \las \in \R \) such that \( \As \) is an invertible Fredholm operator.
\end{proof}

Regarding \cref{eq:problem} as an equation in \( \rmH_0^1(\Omega) \), we thus obtain the following:
\begin{prop}\label{prop:H1Theory}
  Let \cref{hyp:sigma_c,hyp:w_T_cor} hold as well as \( F\in H^{-1}(\Omega) \).
  Then \cref{eq:problem} is equivalent to
  \begin{equation}\label{eq:probleminH01}
    \As u - (\lambda+\las) C u = \mathcal F,
    \qquad u \in \rmH_0^1(\Omega)
  \end{equation}
  where \( \mathcal F\in \rmH_0^1(\Omega) \) is uniquely determined via
  \( \skpH{\mathcal F}{v}= F(v) \) for all \( v\in \rmH_0^1(\Omega) \).
\end{prop}
To prove the existence of an orthonormal basis of eigenfunctions for \cref{eq:problem} we now turn towards
an alternative formulation in \( \rmL^2(\Omega) \). From \cref{prop:Aproperties} and \cref{prop:H1Theory} we
obtain that \cref{eq:problem} is equivalent to
\[
  u - (\lambda + \las)  K_c u  = KF,  \qquad u \in \rmL^2(\Omega)
\]
where
\begin{equation}\label{eq:def_K}
  K_c \coloneqq \iota \As^{-1} \tilde C
  \quad \text{and} \quad
  K F \coloneqq \iota \As^{-1} \mathcal F.
\end{equation}
The compact operator  \( K_c : \rmL^2(\Omega) \to \rmL^2(\Omega) \) is self-adjoint with respect to \( \skpL{\cdot}{\cdot} \) because of
\[
  \skpL{K_c u}{v}
  = \skpH{K_c u}{C v}
  = \skpH{\As^{-1} C u}{C v}
  = \skpH{C u}{\As^{-1} C v}
  = \skpL{u}{K_c v}
\]
for all \( u,v \in \rmL^2(\Omega) \). We have thus proved the following.

\begin{prop}\label{prop:LinearTheory}
  Let \cref{hyp:sigma_c,hyp:w_T_cor} hold as well as \( F\in H^{-1}(\Omega) \).
  Then \cref{eq:problem} is equivalent to
  \[
    u - (\lambda + \las)  K_c u  = KF,
    \qquad u \in \rmL^2(\Omega)
  \]
  where \( K_c, K : \rmL^2(\Omega) \to \rmL^2(\Omega) \) are the compact operators given by \cref{eq:def_K}.
  Moreover, \( K_c \) is self-adjoint with respect to \( \skpL{\cdot}{\cdot} \).
  In particular, the linear problem \cref{eq:problem} satisfies the Fredholm Alternative in
  \( \rmL^2(\Omega) \) in the sense of~\cite[Appendix D, Theorem~5]{Evans_PDEs}.
\end{prop}

The Spectral Theorem for compact self-adjoint operators~\cite[Appendix D, Theorem~7]{Evans_PDEs} provides an orthonormal basis of eigenfunctions as pointed out in~\cite{CarCheCia_Eigenvalue}.
For notational simplicity we use \( \Z^* \coloneqq \Z \sm \{0\} \) as index set for this basis.

\begin{prop}\label{prop:ONB}
  Let \cref{hyp:sigma_c,hyp:w_T_cor} hold.
  Then there is an \( (\rmL^2(\Omega), \skpL{\cdot}{\cdot}) \)-orthonormal basis \( \plr{\phi_j}_{j \in \Z^*} \) of eigenfunctions with associated eigenvalue sequence \( {(\mu_j)}_{j \in \Z^*} \) of the operator \( K_c \) such that
  \[
    \mu_{-1} \leq \mu_{-2} \leq \cdots \leq \mu_{-n}
    \nearrow 0 \swarrow
    \mu_n \leq \cdots \leq \mu_2 \leq \mu_1.
  \]
  In addition, the family \( \plr{\phi_j}_{j \in \Z^*} \) is dense in \( \rmH_0^1(\Omega) \) and
  \begin{equation}\label{eq:orthogonality}
    \skpL{\phi_i}{\phi_j} = \delta_{i, j}
    \quad \text{and} \quad
    a(\phi_i, \phi_j) = \plr{\mu_i^{-1} - \las} \delta_{i, j},
    \qquad \text{for } i, j \in \Z^*.
  \end{equation}
  Moreover, there is \( D > 0 \) such that
  \begin{equation}\label{eq:WeylAsymptotics}
    \abs{\mu_j} \leq D \abs{j}^{-\frac{2}{N}},
    \qquad \text{for all } j \in \Z^*.
  \end{equation}
\end{prop}

\begin{proof}
  By \cref{prop:LinearTheory} the compact operator \( K_c \) is self-adjoint on \( (\rmL^2(\Omega), \skpL{\cdot}{\cdot}) \).
  Therefore, the spectral theorem for self-adjoint compact operators~\cite[Appendix D, Theorem~7]{Evans_PDEs} yields an orthonormal basis \( {(\phi_j)}_{j\in\Z^*} \) of \( (\rmL^2(\Omega), \skpL{\cdot}{\cdot}) \) consisting of eigenfunctions of \( K_c \) where the corresponding real eigenvalue sequence \( {(\mu_j)}_{j\in\Z^*} \) converges to zero.
  We claim that \( \mu_j \neq 0 \) holds for all \( j\in\Z^* \).
  Indeed, assuming the contrary, we get  \( K_c \phi_j = 0 \) and thus \( C \phi_j = 0 \), which in turn implies \( \skpL{\phi_j}{v} = \skpH{C \phi_j}{v} = 0 \) for all \( v \in \rmH_0^1(\Omega) \) because of \cref{eq:def_op_AC}.
  But this is impossible given that \( {(\phi_j)}_{j\in\Z^*} \) is an orthonormal basis in \( \rmL^2(\Omega) \) with respect to \( \skpL{\cdot}{\cdot} \) and \( \rmH_0^1(\Omega) \) is dense in \( \rmL^2(\Omega) \).

  \medskip

  We now prove \cref{eq:orthogonality}. Since \( \mu_j \neq 0 \), the relation \( K_c\phi_j = \mu_j
  \phi_j \) implies \( \phi_j \in \rmH_0^1(\Omega) \) and \( C \phi_j = \mu_j \As \phi_j \).
  Using \cref{eq:def_op_AC} we get for all \( v \in \rmH_0^1(\Omega) \)
  \begin{equation}\label{eq:skp_phi}
    \skpL{\phi_j}{v}
    = \skpH{C \phi_j}{v}
    = \mu_j \skpH{ \As \phi_j}{v}
    = \mu_j \plr{a(\phi_j, v) + \las \skpL{\phi_j}{v}}.
  \end{equation}
  In particular, choosing \( v = \phi_i \) in \cref{eq:skp_phi}, we obtain \( a(\phi_j,\phi_i) =
  (\mu_j^{-1} - \las) \delta_{i,j} \).

  \medskip

  To show that \( \plr{\phi_j}_{j \in \Z^*} \) is dense in \( \rmH_0^1(\Omega) \), consider any \( u \in \rmH_0^1(\Omega) \) such that \( \skpH{u}{\phi_j} = 0 \) for
  all \( j \in \Z^* \). We want to show \( u=0 \). Using \( v = \As^{-1} u \) in \cref{eq:skp_phi}, we get
  \[
    \skpL{\phi_j}{\As^{-1} u}
    = \mu_j \skpH{ \As \phi_j}{\As^{-1} u}
    = \mu_j \skpH{\phi_j}{u} = 0
    \quad \text{for all } j \in \Z^*.
  \]
  However, \( \plr{\phi_j}_{j \in \Z^*} \) is an orthonormal basis of \( (\rmL^2(\Omega), \skpL{\cdot}{\cdot})
  \), which implies  \( \As^{-1} u = 0 \) and thus \( u = 0 \).
  Therefore, the family \( {(\phi_j)}_{j\in\Z^*} \) is dense in \( \rmH_0^1(\Omega) \).

  \medskip

  We finally prove the Weyl law asymptotics \cref{eq:WeylAsymptotics}. This is based on the
  Courant-Fischer min-max characterization for the eigenvalues \( \mu_j \) in terms of \( K_c \). In fact, the
  formula
  \begin{equation}\label{eq:RayleighQuotient}
    \frac{\skpL{K_c \phi}{\phi}}{\norm{\phi}_c^2}
    = \frac{ \sum_{j\in\Z^*} \mu_j c_j^2 }{\sum_{j\in\Z^*} c_j^2}
    \quad \text{for } \phi = \sum_{j\in\Z^*} c_j \phi_j
    \text{ and } {(c_j)}_{j\in\Z^*} \in \las^2(\Z^*)
  \end{equation}
  implies  for \( j\in\N \)
  \begin{equation}\label{eq:minmax_mu}
    \mu_j = \max_{\substack{X \subset \rmL^2(\Omega) \\ \dim(X) = j}} \min_{\phi \in X \setminus \{0\}} \frac{\skpL{K_c \phi}{\phi}}{\norm{\phi}_c^2}
    \quad \text{and} \quad
    \mu_{-j} = \min_{\substack{X \subset \rmL^2(\Omega) \\ \dim(X) = j}} \max_{\phi \in X \setminus \{0\}} \frac{\skpL{K_c \phi}{\phi}}{\norm{\phi}_c^2}.
  \end{equation}
  To prove the upper bound for \( |\mu_j|,j\in\Z^* \) we use
  \[
    \abs{\skpL{K_c \phi}{\phi}}
    = \abs{ \skpH{\As^{-1} \tilde C \phi}{\tilde C \phi} }
    \leq \opnorm{\As^{-1}} \, \normH{\tilde C \phi}^2
    = \opnorm{\As^{-1}} \skpL{\iota \tilde C \phi}{\phi}
  \]
  for all \( \phi \in \rmL^2(\Omega) \). Here, the equalities hold due
  to \cref{eq:def_op_AC} and \cref{eq:def_K}.
  So the definition of \( \tilde C \) implies that \( |\mu_j|,|\mu_{-j}| \)   is bounded from above
  by \( \opnorm{\As^{-1}} {\zeta_j(\Omega)}^{-1} \) where \( \zeta_j(\Omega) \) is the \( j \)-th smallest
  Dirichlet eigenvalue of \( \phi\mapsto -c^{-1} \diver\plr{ |\sigma| \nabla \phi} \) on \( \Omega \) and
  \( j\in\N \).
  Since \( c \) and \( |\sigma| \) are uniformly positive and bounded, \( \zeta_j(\Omega) \) is bounded from below by some
  \( j \)-independent multiple of the \( |j| \)-th smallest eigenvalue of the Dirichlet-Laplacian over \( \Omega \). We
  thus conclude that there is \( D > 0 \) such that
  \[
    \abs{\mu_j}
    \leq \opnorm{\As^{-1}} {\zeta_j(\Omega)}^{-1}
    \leq D \, |j|^{-\frac{2}{N}}.
    \qquad (j\in\Z^*)
  \]
  This finishes the proof of \cref{eq:WeylAsymptotics}.
\end{proof}

To facilitate the application of this result we add a corollary.

\begin{cor}\label{cor:ONB}
  Let \cref{hyp:sigma_c,hyp:w_T_cor} hold.
  Then, there is a sequence \( \plr{\lambda_j, \phi_j}_{j \in \Z} \) consisting of all eigenpairs of the
  differential operator \( \phi \mapsto -{c(x)}^{-1} \diver(\sigma(x) \nabla \phi) \) on \( \rmH_0^1(\Omega) \) that satisfies
  \[
    -\infty \swarrow \cdots \leq \lambda_{-j} \leq \cdots \leq \lambda_{-1} \leq \lambda_0 \leq \lambda_1 \leq
    \cdots \leq \lambda_j \leq \cdots \nearrow +\infty
  \]
  and \( \plr{\phi_j}_{j \in \Z} \) is an orthonormal basis of \( (\rmL^2(\Omega), \skpL{\cdot}{\cdot}) \)
  which is dense in \( \rmH_0^1(\Omega) \). Moreover, there are constants \( m,M> 0 \) such that
    \[
      1+|\lambda_j| \geq m(1+|j|)^{\frac{N}{2}} \;(j\in\Z)
      \quad\text{and}\quad
      \Card\setst{j \in \Z}{-\Lambda \leq \lambda_j \leq \Lambda}
      \leq M \, \Lambda^{\frac{N}{2}},
      \;(\Lambda \geq 1).
    \]
\end{cor}
%------------
\begin{proof}
  We choose the eigenpairs \( \plr{\lambda_j,\phi_j}_{j \in \Z} \) such that
  \[
    \setwt{\plr{\lambda_j,\phi_j}}{j \in \Z} =
    \setwt{\plr{\mu_j^{-1} - \las,\phi_j}}{j \in \Z^*}
  \]
  where the map \( j \mapsto \plr{\lambda_j,\phi_j} \) is injective and \( j\mapsto \lambda_j \) is nondecreasing.
  Then, using the estimates for \( \mu_j \) from \cref{eq:WeylAsymptotics}, we find 
  $1+|\lambda_j| \geq m(1+|j|)^{\frac{N}{2}}$ for some constant $m>0$. Moreover,  $|\lambda_j|\leq
  \Lambda$ and $\Lambda \geq 1$ implies
  \[
    \Lambda
    \geq \abs{\mu_j^{-1} - \ell}
    \geq \abs{\mu_j}^{-1} - \abs{\ell}
    \geq D^{-1} \abs{j}^{\frac{2}{N}} - \abs{\ell}.
  \]
  Hence,
  \[
    \Card\setst{j \in \Z}{-\Lambda \leq \lambda_j \leq \Lambda}
    \leq \Card\setst{j \in \Z}{\abs{j}^{\frac{2}{N}} \leq D (\Lambda + \abs{\ell})}
    \leq M \, \Lambda^{\frac{N}{2}}
  \]
  for some constant \( M > 0 \) as claimed. 
\end{proof}

\begin{rem}\label{rem:T1_assumption}
  \hfill
  \begin{enumerate}[label=(\alph*)] % ChkTex 36
    \item In \cref{cor:ONB}, the ordering of the eigenvalues \( \plr{\lambda_j}_{j \in \Z} \) is fixed up to translations of the indices and permutations within eigenspaces.
          The former  ambiguity can be removed by specifying \( \lambda_0 \).
          A natural way to do this is to require that \( \lambda_0 \) has the smallest absolute value.
          As mentioned earlier, we do not choose such an ordering in our 1D model example from \cref{cor:bif1D} because it is in general not consistent with the \( j \)-dependent nodal patterns.

    \item A reasonable min-max formula for the eigenvalues  \( \plr{\lambda_j}_{j \in \Z} \) in terms of the bilinear form \( a \) does not seem to exist despite the simple formula \( a(\phi,\phi)= \sum_{j\in\Z}\lambda_j c_j^2 \) for \( \phi=\sum_{j\in\Z} c_j\phi_j \).
          In fact, the bilinear form \( (u,v) \mapsto a(u,v) \) has a totally isotropic subspace of infinite dimension, for example
          \[
            \spa\setwt{\abs{\lambda_j}^{\frac{1}{2}} \phi_{-j} + \abs{\lambda_{-j}}^{\frac{1}{2}} \phi_j}{j
            \in \N_0,\ \lambda_{-j} \lambda_j \leq 0} \subset \setst{u \in \rmH_0^1(\Omega)}{a(u, u)  = 0}.
          \]

    \item In~\cite[Section~1]{CarCheCia_Eigenvalue}, the authors  provide some explicit one-dimensional example showing that all statements in this section may be false when \( c \in \rmL^\infty(\Omega) \) is sign-changing.
          In fact, they showed that for some tailor-made \( \sigma \) as in \cref{hyp:sigma_c} and \( c \coloneqq \sigma \) the operator \( u \mapsto -{c(x)}^{-1} \diver(\sigma(x) \, \nabla u) \) may have the whole complex plane as
          spectrum.
          In particular, the spectral theory of (compact) self-adjoint operators does not apply in this context.\label{test}

    \item We mention some similarities and differences concerning the spectral properties of the differential operator \( u\mapsto -{c(x)}^{-1}\diver(\sigma(x)\nabla u) \) for
          \begin{center}
            (I)\; sign-changing \( \sigma \) and \( c=1 \), \qquad
            (II)\; \( \sigma=1 \) and sign-changing \( c \).
          \end{center}
          In the case (I) \cref{prop:ONB} and \cref{cor:ONB} show that the sequence of eigenvalues is unbounded from above and from below.
          This is also true for (II), see the Propositions 1.10 and 1.11 in~\cite{deFig_Pos}.
          On the other hand, there are subtle differences.
          As we will see in \cref{lem:sort_eigs_1d}, in our one-dimensional model example for case (I) there is precisely one positive eigenfunction \( \phi_0 \) with associated eigenvalue \( \lambda_0 \), and that one  might not have the smallest  absolute value among all eigenvalues.
          In fact, \( \abs{\lambda_0} \) can be much larger than \( \abs{\lambda_{-1}} \), see \cref{rem:ordering}.
          In particular, there is little hope to prove the existence of positive eigenvalues via some straightforward application of the Krein-Rutman theorem.
          This is different for the case (II) where Manes-Micheletti~\cite{ManMich_UnEstensione} (see also~\cite[Theorem~1.13]{deFig_Pos}) proved the existence of one positive and one negative principal eigenvalue, \emph{i.e.}, algebraically simple eigenvalues of the smallest absolute value among the positive and negative eigenvalues, respectively, coming with positive eigenfunctions.
          So here the two models exhibit different phenomena.
  \end{enumerate}
\end{rem}

\section{Proof of \texorpdfstring{\cref{thm:bif}}{\ref{thm:bif}} and \texorpdfstring{\cref{cor:bif1D}}{\ref{cor:bif1D}}}\label{sec04}

We now prove the theoretical bifurcation results  with the aid of   known bifurcation results for equations
of the form \( F(u, \lambda) = 0 \)  where \( F \in \scrC^2(H \times \R, H) \) for some   Hilbert space  \( H \). We will consider bifurcation from the trivial solution branch, so \( F \) is supposed  to satisfy \( F(0, \lambda) = 0 \) for all \( \lambda \in \R \). To this end we proceed as follows:
First, we present two abstract bifurcation theorems that allow to detect local respectively global
bifurcation from the trivial solution.  Next, we show how to apply these results to prove \cref{thm:bif}, which is
straightforward. Finally, we sharpen our results in the one-dimensional case of \cref{eq:NLeq1D} by proving
\cref{cor:bif1D}.

\subsection{Known abstract Bifurcation Theorems}

\subsubsection{Local Variational Bifurcation}

We first present a simplified version of~\cite[Theorem~2.1(i)]{PejWat_Bifurcation} (see also~\cite[Corollary~3]{FitzPejRech}) that allows to detect local bifurcation for equations of the form \( \nabla
\Psi_\lambda(u)=0 \) for some \( \Psi \in \scrC^2(H \times \R, \R) \) and \( \Psi_\lambda \coloneqq  \Psi(\cdot,\lambda) \).
We denote the Fr\'{e}chet derivative of \( \Psi_\lambda \) at \( u \) by \( \Psi_\lambda'(u):H\to H,\, \phi\mapsto
\skp{\nabla \Psi_\lambda(u)}{\phi} \).

\begin{thm}\label{thm:VariationalBifurcation}
  Suppose \( H \) is a separable real Hilbert space and \( \Psi \in \scrC^2(H \times \R, \R) \) satisfies
  \( \nabla \Psi_\lambda(u) = Lu - \lambda K u- \mathcal{G}(u) \)  where
  \begin{itemize}
    \item[(i)] \( L:H\to H \) is a linear invertible self-adjoint Fredholm operator,
    \item[(ii)] \( K:H\to H \) is a linear compact and positive self-adjoint operator,
    \item[(iii)] \( \mathcal{G}:H\to H \) satisfies \( \mathcal{G}(0)=0 \) and \( \mathcal{G}'(0)=0 \).
  \end{itemize}
  Then each \( \lambda_\star \in \R \) such that \( \ker(\Psi_{\lambda_\star}''(0)) \neq \{0\} \) is a
  bifurcation point for \( \nabla \Psi_\lambda(u)=0 \).
\end{thm}
\begin{proof}
  Our assumptions (i), (ii), (iii) imply that \( {(\Psi_\lambda)}_{\lambda\in\R} \) is a continuous family of
  \( \scrC^2 \)-functionals in the sense of~\cite[p.537]{PejWat_Bifurcation}.
  If \( \lambda_\star \in \R \) is as required, then Theorem~2.1(i) in~\cite{PejWat_Bifurcation} proves that
  the interval \( [\lambda_\star-\eps,\lambda_\star+\eps] \) contains a bifurcation point provided that the
  Hessians \( \Psi_{\lambda_\star\pm\eps}''(0) \) are invertible and the spectral flow of this family over the
  interval \( I \coloneqq [\lambda_\star-\eps,\lambda_\star+\eps] \) is non-zero.
  In fact, since \( L \) is invertible and \( K \) is compact, the linear operator \( \Psi_{\lambda}''(0) = L-\lambda
  K \) has a nontrivial kernel only for \( \lambda \) belonging to a discrete subset of \( \R \). So we may
  choose \( \eps > 0 \) so small that \( \Psi_{\lambda_\star+\eps}''(0),\Psi_{\lambda_\star-\eps}''(0) \) are
  invertible and \( \lambda_\star \) is the only candidate for bifurcation in \( I \) by the Implicit Function Theorem.
  Using then the positivity of \( K \) we get from Remark~(3) in~\cite{PejWat_Bifurcation} that the spectral
  flow over \( I \) is the dimension of \( \ker(\Psi_{\lambda_\star}''(0)) \), which is positive by
  assumption.
  So \( \lambda_\star \) is a bifurcation point.
\end{proof}

The more classical variational bifurcation theorems by Marino~\cite{Marino_Bif},
B\"ohme~\cite[Satz~II.1]{Boehme} and Rabinowitz~\cite[Theorem~11.4]{Rab_Minimax} apply if there is
\( \mu\in\R \) such that the self-adjoint operator \( L+\mu K \) generates a norm.
This assumption holds in the context of nonlinear elliptic boundary value problems
involving divergence-form operators with diffusion coefficients \( \sigma \) having a fixed sign. In our
setting with sign-changing \( \sigma \), this is not the case.

\subsubsection{Global Bifurcation}

Rabinowitz' Global Bifurcation Theorem~\cite{Rab_SomeGlobal}  states  that
the   solutions bifurcating from an eigenvalue of odd algebraic multiplicity
lie on solution continua that are unbounded or return to the trivial
solution branch \( \setwt{(0, \lambda)}{\lambda \in \R} \) at some other bifurcation point.
Here, a solution continuum is a closed and connected set consisting of solutions. Given that the
proof of this bifurcation theorem uses Leray-Schauder degree theory, more restrictive compactness properties
are required to be compared to \cref{thm:VariationalBifurcation}. On the other hand, no variational
structure is assumed. In order to avoid technicalities, we state a simplified variant of this result from
Theorem~II.3.3 in~\cite{Kielhoefer_Bifurcation}. The set \( \mathcal{S} \subset H \times \R \) denotes the
closure of nontrivial solutions of \( F(u, \lambda) = 0 \) in \( H \times \R \). In particular, the statement
\( (0,\lambda_\star)\in \mathcal{S} \) is equivalent to saying that \( (0,\lambda_\star) \) is a bifurcation point for
\( F(u,\lambda)=0 \), i.e., there are solutions \( (u^n,\lambda^n) \in H \sm \{0\}\times\R \)
such that \( (u^n,\lambda^n)\to (0,\lambda_\star) \) in \( H\times\R \) as \( n\to\infty \).

\begin{thm}[Rabinowitz]\label{thm:Rabinowitz}
  Suppose \( H \) is a separable real Hilbert space and that \( F\in \scrC^1(H\times \R,H) \) is given by \( F(u,
  \lambda) = L u - \lambda K u - \mathcal{G}(u) \) where
  \begin{itemize}
    \item[(i)] \( L : H \to H \) is  a linear invertible self-adjoint Fredholm operator,
    \item[(ii)] \( K : H \to H \) is a linear compact and positive self-adjoint operator,
    \item[(iii)] \( \mathcal{G}: H \to H \) is compact with \( \mathcal{G}(0)=0 \) and \( \mathcal{G}'(0) = 0 \).
  \end{itemize}
  Suppose that \( \lambda_\star \in \R \) is such that the dimension of
  \( \ker(L-\lambda_\star K ) \) is odd.
  Then \( (0, \lambda_\star) \in \mathcal{S} \).
  Moreover, if \( \mathcal{C} \) denotes the connected component of \( (0,\lambda_\star) \) in \( \mathcal{S} \), then
  \begin{itemize}
    \item[(I)] \( \mathcal{C} \) is unbounded or
    \item[(II)] \( \mathcal{C} \) contains a point \( (0, \lambda^\star) \) with \( \lambda^\star \neq
          \lambda_\star \).
  \end{itemize}
\end{thm}

A more general version of this result  holds  in Banach spaces and does not involve any self-adjointness
assumption. It then claims the above-mentioned properties of  \( \mathcal{C} \) assuming that \( \lambda_\star \) is an eigenvalue of odd algebraic multiplicity.  Under our more restrictive assumptions
including self-adjointness the algebraic multiplicity of \( \lambda_\star \) is equal to its geometric
multiplicity and hence to the dimension of the corresponding eigenspace.

\subsection{Proof of \texorpdfstring{\cref{thm:bif}}{\ref{thm:bif}}}

It suffices to show that \cref{eq:NLeq} fits into the abstract framework required by the above theorems.
As a Hilbert space we choose \( H \coloneqq \rmH_0^1(\Omega) \) with inner product \( \skpH{\cdot}{\cdot} \).
By \cref{prop:H1Theory} a solution \( (u,\lambda) \) of \cref{eq:NLeq} is nothing but a solution of \( F(u,\lambda)=0 \) where
\begin{equation}\label{eq:F}
  F(u,\lambda) \coloneqq \As u - (\lambda+\las) C u - \mathcal{G}(u).
\end{equation}
By \cref{prop:Aproperties}, \( \As : \rmH_0^1(\Omega) \to \rmH_0^1(\Omega) \) is a bounded,
linear, self-adjoint and invertible operator and \( C : \rmH_0^1(\Omega) \to \rmH_0^1(\Omega) \) is a
linear, compact and self-adjoint operator. Moreover, by Sobolev's Embedding
\( \rmH_0^1(\Omega)\hookrightarrow\rmL^4(\Omega) \)
for \( N \in \{1,2,3,4\} \), the mapping \( \mathcal{G}:\rmH_0^1(\Omega)\to \rmH_0^1(\Omega) \)
given by \( \skpH{\mathcal{G}(u)}{\phi}= \int_\Omega \kappa(x) u^3 \phi \di{x} \) is well-defined and
smooth with \( \mathcal{G}(0)=0 \) and \( \mathcal{G}'(0)=0 \).
The Rellich-Kondrachov Theorem implies that \( \mathcal{G} \) is compact due to \( N \in \{ 1, 2
,3 \} \). Note that for \( N=4 \) the exponent \( 4=\frac{2N}{N-2} \) is the Sobolev-critical exponent where
compactness fails. We thus conclude that the assumptions (i), (ii), (iii) of both
\cref{thm:VariationalBifurcation} and \cref{thm:Rabinowitz} hold for \( L \coloneqq \As \), \( K \coloneqq C \) with bifurcation parameter \( \widetilde{\lambda} \coloneqq \lambda + \las \).

\medskip

The energy functional \( \Psi(\cdot, \lambda) \coloneqq \Psi_\lambda : \rmH_0^1(\Omega) \to \R \) required
for \cref{thm:VariationalBifurcation} is given by \cref{eq:def_Phi}.
Then \( \Psi \in \scrC^2(\rmH_0^1(\Omega) \times \R, \R) \) with
\begin{equation}\label{eq:FunctionalDerivative}
  \begin{aligned}
    \Psi_\lambda'(u)[\phi]
     & = \int_\Omega  \sigma(x)\nabla u\cdot\nabla \phi \di{x}
    - \lambda \int_\Omega c(x)u\phi \di{x}
    - \int_\Omega \kappa(x) u^3 \phi \di{x}
    \\
     & =  \skpH{ A u - \lambda C u - \mathcal{G}(u)}{\phi},
  \end{aligned}
\end{equation}
so \( F(u,\lambda)=\nabla\Psi_\lambda(u)= Au-\lambda Cu-\mathcal{G}(u) \). Moreover,
\[
  \ker(\Psi_{\lambda_\star}''(0))
  = \ker(\As - (\lambda_\star+\las)  C)
  = \spa\setwt{\phi_j}{\lambda_j=\lambda_\star}
\]
for \( \lambda_j \) as in \cref{cor:ONB}. So \cref{thm:VariationalBifurcation} implies that each \( \lambda_j \) is a
bifurcation point for \cref{eq:NLeq}. Finally, \cref{thm:Rabinowitz} shows that every such eigenvalue with
odd-dimensional eigenspace satisfies Rabinowitz' Alternative (I) or (II) from above, which finishes the proof
of \cref{thm:bif}. \qed{}

\subsection{Proof of \texorpdfstring{\cref{cor:bif1D}}{\ref{cor:bif1D}}}

We now sharpen our results from \cref{thm:bif} for the one-dimensional boundary value problem
\[
  -\frac{\di{}}{\di{x}}\plr{\sigma(x) \, u'(x)} - \lambda \, c(x) \, u = \kappa(x) u^3
  \quad \text{in }\Omega,
  \qquad u \in \rmH_0^1(\Omega)
\]
from \cref{eq:NLeq1D}.
The assumptions on \( \sigma \), \( c \),  \( \kappa \) and \( \Omega = (a_-, a_+) \subset \R \) were specified
in the Introduction.
We want to verify that \cref{hyp:sigma_c,hyp:w_T_cor} are satisfied in this context.
While \cref{hyp:sigma_c} is trivial, the verification \cref{hyp:w_T_cor}
dealing with the  weak \( \ttT \)-coercivity of \( (u,v) \mapsto a(u, v) \coloneqq \int_\Omega \sigma(x) \, u' v' \di{x} \) requires some work. The following result seems to be well-known to experts, but a reference
appears to be missing in the literature.

\begin{lem}\label{lem:VerificationHypotheses1}
  Let \( \Omega \), \( \sigma \)  and \( c \) be given as in \cref{cor:bif1D}.
  Then the bilinear form \( a \) is weakly \( \ttT \)-coercive.
  In particular, \cref{hyp:w_T_cor} holds.
\end{lem}
%------------
\begin{proof}
  Let \( \chi\in \scrC_0^\infty(\Omega) \) with \( \chi(x) = 1 \) for \( x \) close to \( 0 \).
  Then define
  \[
    \ttT u(x) = u(x)
    \quad \text{if } 0 < x < a_+, \qquad
    \ttT u(x) = 2\chi(x) u(-m x) - u(x)
    \quad \text{if } a_- < x <0,
  \]
  where \( m \in \big( 0, \frac{a_+}{|a_-|}\big) \) will be chosen sufficiently small.
  Then \( \ttT \) is a well-defined bijective operator on \( \rmH_0^1(\Omega) \) because of \( \ttT \circ \ttT = \mathrm{id} \).

  Moreover,
  \begin{align*}
    a(u, \ttT u)
     & = \normH{u}^2
    + 2 m \int_{a_-}^0 \abs{\sigma_-} \chi(x) u'(-m x) u'(x) \di{x}
    - 2 \int_{a_-}^0 \abs{\sigma_-} \chi'(x) u(-m x) u'(x)\di{x}
    \\
     & \geq \normH{u}^2
    - 2 m \norm{\chi}_\infty|\sigma_-| \norm{u'(-m \, \cdot)}_{\rmL^2([a_-, 0])} \norm{u'}_{\rmL^2([a_-, 0])}
    \\
     & - 2 \norm{\chi'}_\infty |\sigma_-|\norm{u(-m \, \cdot)}_{\rmL^2([a_-, 0])} \norm{u'}_{\rmL^2([a_-, 0])}
    \\
     & \geq \normH{u}^2
    - 2 \sqrt{m} \norm{\chi}_\infty|\sigma_-| \norm{u'}_{\rmL^2([0, m |a_-|])} \norm{u'}_{\rmL^2([a_-, 0])}
    \\ &-  2m^{-1/2}\norm{\chi'}_\infty |\sigma_-| \norm{u}_{\rmL^2([0, m |a_-|])}
    \norm{u'}_{\rmL^2([a_-, 0])}                                                                               \\
     & \geq \normH{u}^2
    -  \alpha_1 \sqrt{m}  \normH{u}^2  - \alpha_2m^{-1/2}   \normH{u} \normL{u}                                \\
     & \geq (1- \alpha_1 \sqrt{m} - \alpha_2 \sqrt{m})  \normH{u}^2  - \alpha_2 m^{-3/2}  \normL{u}^2
  \end{align*}
  for \( \alpha_1, \alpha_2 > 0 \). Choosing \( m \in \big( 0, \frac{a_+}{|a_-|}\big) \)
  such that \( (\alpha_1+\alpha_2)\sqrt m \leq \frac{1}{2} \) we obtain
  \[
    a(u, \ttT u) + \skpH{\ttK u}{u}
    \geq \frac{1}{2} \normH{u}^2
  \]
  where \( \ttK \coloneqq \alpha_2 m^{-3/2} C : \rmH_0^1(\Omega) \to  \rmH_0^1(\Omega) \) is a compact operator,
  see \cref{eq:def_op_AC}. Hence, \( a \) is weakly \( \ttT \)-coercive.
\end{proof}

\begin{rem}\label{rem:special_T_1d}
  In higher-dimensional settings the verification of the weak \( \ttT \)-coercivity
  condition is much more sensitive with respect to the data. This concerns \( \Omega_+,\Omega_- \), the
  geometric properties of the interface \( \Gamma=\ov\Omega_+\cap \ov\Omega_- \) and the coefficients
  \( \sigma_+,\sigma_- \). In particular weak \( \ttT \)-coercivity may break down for critical ranges of the
  contrast \( \sigma_+(x)/\sigma_-(x) \), see for instance Theorem~5.1 and Theorem~5.4 in~\cite{BBDhia_MeshRequirements}.
\end{rem}

We thus conclude that the assumptions of \cref{thm:bif} are verified and hence the existence of infinitely many bifurcating branches \( \mathcal{C}_j \) is ensured.

To finish the proof of \cref{cor:bif1D} it now remains to prove the nodal characterization of all nontrivial solutions \( (u,\lambda) \in \mathcal{C}_j \) emanating from \( \lambda_j \).
Choosing an appropriate numbering of the eigenvalue sequence \( (\lambda_j) \) we need to prove that
nontrivial solutions \( (u,\lambda)\in\mathcal{C}_j \) have the stated nodal pattern.
To prove this we first compute and analyze the eigenpairs of the linear problem.

\medskip

\paragraph{Step 1: Nodal characterization of the eigenfunctions}

\begin{lem}\label{lem:sort_eigs_1d}
  Let \( \Omega \), \( \sigma \)  and \( c \) be given as in \cref{cor:bif1D} and let
  \( \plr{\phi_j,\lambda_j}_{j \in \Z} \) denote the sequence of eigenpairs for the one-dimensional boundary
  value problem \cref{eq:NLeq1D} as in \cref{cor:ONB}, set \( k_- \coloneqq  \sqrt{c_-/|\sigma_-|} \), \( k_+
  \coloneqq \sqrt{c_+/\sigma_+} \). Then each eigenvalue \( \lambda_j \) is simple and in particular
  \[
    -\infty \swarrow \cdots  < \lambda_{-2} < \lambda_{-1}   < \lambda_0   < \lambda_1 <  \lambda_2 < \cdots
    \nearrow +\infty.
  \]
  This sequence can be ordered in the following way:

  \begin{itemize}
    \item[(i)] For \( j \geq 1 \), \( \lambda_j > 0 \) is the only eigenvalue in the interval
          \( \plr{ \frac{j^2 \pi^2}{k_+^2 a_+^2}, \frac{{(2 j + 1)}^2 \pi^2}{4 k_+^2 a_+^2} } \) and \( \phi_j \) has
          \( j \) interior zeros in \( \Omega_+ \)
          with \( \abs{\phi_j'} > 0 \) on \( \ov{\Omega_-} \).

    \item[(ii)] For \( j \leq -1 \), \( \lambda_j < 0 \) is the only eigenvalue in the interval
          \( \plr{ -\frac{{(2 j - 1)}^2 \pi^2}{4 k_-^2 a_-^2}, -\frac{j^2 \pi^2}{k_-^2 a_-^2} } \)
          and \( \phi_j \) has \( \abs{j} \) interior zeros in \( \Omega_- \)  with \( \abs{\phi_j'} > 0 \) on
          \( \ov{\Omega_+} \).

    \item[(iii)] \( \lambda_0 \) is the only eigenvalue in the interval \( \plr{ -\frac{\pi^2}{4 k_-^2 a_-^2},
            \frac{\pi^2}{4 k_+^2 a_+^2} } \) and \( \phi_0 \) has no interior zeros in \( \Omega \) with
          \( \abs{\phi_0'} > 0 \) on \( \ov{\Omega_\pm} \). Moreover,
          \begin{align*}
            \lambda_0>0 \; \Leftrightarrow \; \frac{\sigma_+a_-}{a_+\sigma_-} > 1,\quad
            \lambda_0<0 \; \Leftrightarrow \; \frac{\sigma_+a_-}{a_+\sigma_-} < 1,\quad
            \lambda_0 = 0 \; \Leftrightarrow \;  \frac{\sigma_+a_-}{a_+\sigma_-} = 1.
          \end{align*}
  \end{itemize}
\end{lem}
%------------
\begin{proof}
  Any eigenpair \( (\phi, \lambda) \in \rmH_0^1(\Omega) \times \R \) satisfies
  \begin{align*}
     & -\phi''(x) = -\lambda k_-^2 \phi(x)
    \quad \text{on } (a_-, 0),
    \\
     & -\phi''(x) = +\lambda k_+^2 \phi(x)
    \quad \text{on } (0, a_+),
    \\
     & \phi, \sigma\phi' \in \scrC([a_-,a_+]),
    \qquad \phi(a_-)=\phi(a_+)=0.
  \end{align*}
  In the following, \( \tau_j:=\sqrt{|\lambda_j|} \).

  \medskip

  \paragraph{(i)\; Positive eigenvalues}
  Solving the ODE and exploiting the continuity of eigenfunctions as well as the homogeneous Dirichlet boundary conditions, we find the following formula for eigenfunctions \( \phi_j \) associated with positive eigenvalues (\( \lambda_j > 0 \))
  \begin{align}\label{eq:eigenfunctionsI}
    \phi_j(x) & = \alpha_j \begin{dcases}
      \frac{\sinh(\tau_j k_- (a_- - x))}{\sinh(\tau_j k_- a_-)}
       & \text{if } x \in (a_-, 0) \\
      \frac{\sin(\tau_j k_+ (a_+ - x))}{\sin(\tau_j k_+ a_+)}
       & \text{if } x \in (0, a_+)
    \end{dcases}
  \end{align}
  The parameter \( \alpha_j \in \R \sm \{0\} \) is chosen such that \( \norm{\phi_j}_c = 1 \).
  The equation for \( \lambda_j \) now results from the condition that \( \sigma\phi_j' \) has to be continuous.
  This means
  \begin{equation}\label{eq:lambdaj>0}
    \frac{\tan\plr{\tau_j k_+ a_+}}{\tanh\plr{\tau_j k_- a_-}} \cdot \frac{\sigma_- k_-}{\sigma_+ k_+} = 1,
    \qquad (\lambda_j > 0).
  \end{equation}
  By elementary monotonicity considerations one finds that this equation has a unique solution such that  \( \tau_j k_+ a_+ \in \plr{j \pi, (j+\frac{1}{2}) \pi} \) for \( j \geq 1 \).
  Moreover, it has a unique solution such that  \( \tau_0 k_+a_+ \in \plr{0, \frac{1}{2} \pi} \) if
  and only if \( \frac{\sigma_+ a_-}{ \sigma_- a_+}>1 \).  No further solutions  exist.
  We thus obtain:
  \begin{itemize}
    \item For \( j \geq 1 \), there is a unique solution \( \lambda_j \) in the interval \( \plr{ \frac{j^2 \pi^2}{k_+^2 a_+^2}, \frac{{(2 j + 1)}^2 \pi^2}{4 k_+^2 a_+^2} } \) and \( \phi_j \) has   \( j \) interior zeros in \( \Omega_+ \) with \( \abs{\phi_j'} > 0 \) on \( \ov{\Omega_-} \).

    \item If \( \frac{\sigma_+ a_-}{ \sigma_- a_+}>1 \), then there is a unique solution \( \lambda_0 \) in the interval \( \plr{ 0, \frac{\pi^2}{4 k_+^2 a_+^2} } \) and  \( \phi_0 \) has no interior zeros in \( \Omega \) with \( \abs{\phi_0'} > 0 \) on \( \ov{\Omega_\pm} \).
  \end{itemize}

  \medskip

  \paragraph{(ii)\; Negative eigenvalues}
  Similarly, we obtain for the negative eigenvalues (\( \lambda_j < 0 \))
  \begin{equation}\label{eq:eigenfunctionsII}
    \phi_j(x)  = \alpha_j \begin{dcases}
      \frac{\sin(\tau_j k_- (a_- - x))}{\sin(\tau_j k_- a_-)}
       & \text{if } x \in (a_-, 0) \\
      \frac{\sinh(\tau_j k_+ (a_+ - x))}{\sinh(\tau_j k_+ a_+)}
       & \text{if } x \in (0, a_+)
    \end{dcases}
  \end{equation}
  and
  \begin{equation}\label{eq:lambdaj<0}
    \frac{\tan\plr{\tau_j k_- a_-}}{\tanh\plr{\tau_j k_+ a_+}}
    \cdot \frac{\sigma_+ k_+}{\sigma_- k_-} = 1,
    \qquad (\lambda_j < 0).
  \end{equation}
  As for the positive eigenvalues one finds:
  \begin{itemize}
    \item For \( j \geq 1 \), there is a unique solution \( \lambda_{-j} \) in the interval
          \( \plr{ -\frac{{(2 j + 1)}^2 \pi^2}{4 k_-^2 a_-^2}, -\frac{j^2 \pi^2}{k_-^2 a_-^2} } \) and
          \( \phi_j \) has \( \abs{j} \) interior zeros in \( \Omega_- \)  with \( \abs{\phi_j'} > 0 \) on
          \( \ov{\Omega_+} \).
    \item  If \( \frac{\sigma_+ a_-}{ \sigma_- a_+}<1 \), then there is a unique solution \( \lambda_0 \) in
          the interval \( \plr{ -\frac{\pi^2}{4 k_-^2 a_-^2}, 0 } \) and  \( \phi_0 \) has no interior zeros in \( \Omega \) with
          \( \abs{\phi_0'} > 0 \) on \( \ov{\Omega_\pm} \).
  \end{itemize}

  \medskip

  \paragraph{(iii)\; Zero eigenvalue}
  The eigenvalue zero only occurs if \( \frac{\sigma_-}{\sigma_+} = \frac{a_-}{a_+} \).
  Here the associated eigenfunction is given by (\( \lambda_0 = 0 \))
  \begin{equation}  \label{eq:eigenfunctionsIII}
    \phi_0(x) = \alpha_0 \begin{dcases}
      1-\frac{x}{a_-}
       & \text{if } x \in (a_-, 0) \\
      1-\frac{x}{a_+}
       & \text{if } x \in (0, a_+)
    \end{dcases}
    \qquad (\lambda_0=0).
  \end{equation}
  \begin{itemize}
    \item  If \( \frac{\sigma_+ a_-}{ \sigma_- a_+}=1 \), then \( \lambda_0=0 \)   and  \( \phi_0 \) has no
          interior zeros in \( \Omega \) with \( \abs{\phi_0'} > 0 \) on \( \ov{\Omega_\pm} \).
  \end{itemize}
\end{proof}

\begin{rem}\label{rem:ordering}
  This ordering allows for configurations
  where \( \lambda_0 \) is not of the least absolute value, say
  \( \ldots<\lambda_{-M} < \cdots < \lambda_{-1}<0<\lambda_0<\lambda_1<\ldots \) with
  \( \lambda_0>|\lambda_{-M}| > \cdots > |\lambda_{-1}| \) for any given \( M\in\N \).
  In fact, choose \( a_- = -1 \), \( a_+ = 1 \), \( c_+ = c_- = 1 \), \( \sigma_- = -m^{-1} \) and define \( \sigma_+ \) via \cref{eq:lambdaj>0}.
  This means that we choose \( \sigma_+ = \sigma_+(m) \in (\frac{4\lambda_0}{\pi^2},\infty) \) as the largest
  positive solution of
  \[
    \frac{\tan(\sqrt{\lambda_0/\sigma_+})}{\sqrt{\sigma_+}}
    = \frac{\tanh(\sqrt{\lambda_0/|\sigma_-|})}{\sqrt{|\sigma_-|}}
    = \sqrt{m}\tanh(\sqrt{\lambda_0 m}).
  \]
  Then \cref{lem:sort_eigs_1d}~(iii) and \( |k_-a_-|=m^{1/2} \) imply that all negative eigenvalues converge
  to 0 and \( \sigma_+(m)\to \frac{4\lambda_0}{\pi^2} \)
  as \( m\to\infty \) whereas \( \lambda_0>0 \) is invariant with respect to \( m \) by choice of \( \sigma_+ \). In this case clustering of eigenvalues at \( 0 \) occurs.
\end{rem}

\paragraph{Step 2: Nodal characterization close to the bifurcation points}

Next we deduce that the nontrivial solutions \( (u, \lambda) \in \mathcal{C}_j \) sufficiently close to \( (0, \lambda_j) \) have the same nodal pattern as \( \phi_j \).
To see this, we first prove that if \( (u^n, \lambda^n) \in \mathcal{C}_j \) converges
to \( (0, \lambda_j) \) in \( \rmH_0^1(\Omega) \), then \( \tilde{u}^n \) converges uniformly on \( \ov\Omega \) to \( \phi_j/\normH{\phi_j} \)  where
\[
  \tilde{u}^n \coloneqq  \gamma_n u^n /   \normH{u^n}\qquad\text{where }
  \gamma_n:= \sign(\skpH{u^n}{\phi_j}).
\]
By the subsequence-of-subsequence argument, it suffices to prove that a subsequence of \( (\tilde u^n) \) has this
property. Since \( (\tilde u^n) \) is bounded, there is a weakly convergent subsequence with limit \( \phi \in
\rmH_0^1(\Omega) \).
Since \( u^n \) solves \cref{eq:NLeq1D} we have by \cref{prop:H1Theory}
\[
  \tilde{u}^n = \As^{-1}(\lambda^n+\las)C \tilde{u}^n  + \normH{u^n}^2 \As^{-1}\mathcal{G}(\tilde{u}^n)
\]
where \( \As^{-1} \) is bounded and \( C,\mathcal{G} \) are compact, see
the proof of \cref{thm:bif}. So \( \tilde{u}^n\rightharpoonup  \phi \) and \( \normH{u_n}^2\to 0 \) implies \( \tilde{u}^n\to \phi \) in \( \rmH_0^1(\Omega) \) where \( \phi = \As^{-1}(\lambda_j+\las)C\phi \).
In other words, \( \phi \) is an eigenfunction of \( \phi\mapsto -{c(x)}^{-1}\diver(\sigma(x)\nabla \phi) \)
associated with the eigenvalue \( \lambda_j \). Since the eigenspaces are one-dimensional, \( \phi \) is a multiple
of \( \phi_j \). Moreover,
from \( \normH{\tilde{u}^n}=1 \) we infer \( \normH{\phi}=1 \), so \( \skpH{\tilde{u}^n}{\phi_j}\geq
0 \) (by choice of \( \gamma_n \)) implies   \( \phi= + \phi_j /\normH{\phi_j} \).
We thus conclude that \( \tilde{u}^n\to \phi_j/\normH{\phi_j} \) in \( \rmH_0^1(\Omega) \) and hence uniformly
on \( \overline\Omega \) as \( n\to\infty \).  Integrating \cref{eq:NLeq1D} once, one finds that the convergence even holds in \( \scrC^1(\ov{\Omega_+}) \) and \( \scrC^1(\ov{\Omega_-}) \). So if infinitely many \( u^n \) had more than \( j \) interior zeros in \( \Omega_\pm \), then the collapse of zeros would cause at least one double zero of \( \phi_j \), but this is
false in view of our formulas for these eigenfunctions from above.
So almost all \( u^n \) have at most \( j \) interior zeros in \( \Omega_\pm \).
Similarly, almost all \( u^n \) have at least \( j \) zeros. So we conclude that the solutions close to the
bifurcation point have exactly \( j \) interior zeros in \( \Omega_\pm \) and are strictly monotone in \( \Omega_\mp \).

\medskip

\paragraph{Step 3: Nodal characterization along the whole branch}

We finally claim that this nodal property is preserved on connected subsets of \( \mathcal{S} \) that do not contain the trivial solution.
Indeed, the set of solutions on \( \mathcal{C}_j \sm \{(0, \lambda_j)\} \) with this property is open in \( \mathcal{S} \) with respect to the topology of \( \rmH_0^1(\Omega) \times \R \).
It is also closed in \( \mathcal{S} \) since double zeros cannot occur  (by the same arguments as above) and zeros cannot converge to the interface at \( x = 0 \) as the solutions evolve along the branch.
Indeed, in the latter case the equation on the monotone part would imply that the solution has to vanish  identically there, whence \( u \equiv  0 \) on \( \Omega \), which is impossible.
So we conclude that all elements on \( \mathcal{C}_j \sm \{(0, \lambda_j)\} \) have the claimed property and the proof is finished. \qed{}

\section{Variational methods}\label{sec:var_method}

We want to show that variational methods can be used to prove further existence and multiplicity results
for \cref{eq:NLeq} under  stronger assumptions than before.
Our aim is to prove the existence of infinitely many nontrivial solutions of
\begin{equation}\label{eq:NLeq_general}
  -\diver(\sigma(x) \, \nabla u) - \lambda \, c(x) \, u = \kappa(x)u^3,
  \qquad u\in \rmH_0^1(\Omega)
\end{equation}
for any given \( \lambda\in\R \). This means that any vertical line in a bifurcation diagram for
\cref{eq:NLeq}, say \cref{fig:1d:bif}, hits infinitely many nontrivial solutions of \cref{eq:NLeq_general}.
To show this we apply the generalized Nehari manifold approach from~\cite[Chapter~4]{SzuWet_Nehari}.
We start by considering the general case  where \( \Omega \subset \R^N \) is an arbitrary bound domain and
\( N \in \{1,2,3\} \). In this setting we will need more information about
the orthonormal basis of eigenfunctions from \cref{cor:ONB},
see \cref{hyp:sum_eigs} below. Since the latter can be checked in our one-dimensional model example,
the general result applies and leads to infinitely many solutions for any given \( \lambda \in \R \)
for \cref{eq:NLeq1D}, see \cref{cor:var}.

\subsection{The general case}

The variational approach aims at proving the existence of critical points of energy functionals. In our
case such a functional is given by  \( \Psi_\lambda : H \to \R \) where
\begin{equation*}
  \Psi_\lambda(u)
  \coloneqq \frac{1}{2} \int_\Omega \sigma(x) \abs{\nabla u(x)}^2  \di{x}
  - \frac{\lambda}{2} \int_\Omega c(x) \, {u(x)}^2 \di{x}
  -   \frac{1}{4} \int_\Omega \kappa(x) {u(x)}^4  \di{x}
\end{equation*}
and \( \lambda \in \R \) is fixed, see \cref{eq:def_Phi}. Of course, \( \Psi_\lambda \) is a well-defined smooth
functional on \( \rmH_0^1(\Omega) \), but the more natural setting for our analysis involves
another Hilbert space \( H \) that will be smaller or equal to \( \rmH_0^1(\Omega) \) under our assumptions. To define it,
denote by \( M \) the subspace of \( \rmH_0^1(\Omega) \) consisting of all finite
linear combinations of the eigenfunctions \( \plr{\phi_j}_{j\in\Z} \) from \cref{cor:ONB}.
Those exist by \cref{hyp:sigma_c,hyp:w_T_cor}. We recall
\begin{equation}\label{eq:muj}
  \skpL{\phi_i}{\phi_j} = \delta_{i,j}, \qquad
  a(\phi_i,\phi_j) = \lambda_j \delta_{i,j}, \qquad
  \pm\lambda_j \nearrow +\infty \text{ as } j \to \pm\infty.
\end{equation}
Then we define \( H \) to be the completion of \( M \) with respect to the norm \( \norm{\cdot} \coloneqq \skp{\cdot}{\cdot}^{1/2} \)
generated by the positive definite bilinear form
\begin{equation}\label{eq:def_norm}
  \skp{\sum_{i\in\Z} c_i\phi_i}{\sum_{j\in\Z} \tilde c_j\phi_j}
  \coloneqq \sum_{\lambda_i-\lambda\neq 0} c_i \tilde c_i |\lambda_i-\lambda|
  +  \sum_{\lambda_i-\lambda=0} c_i \tilde c_i.
\end{equation}
We will need \( H\subset \rmH_0^1(\Omega) \) in order to benefit both from Sobolev's Embedding Theorem and the
Rellich-Kondrachov Theorem in our analysis. So we have to ensure that the norm \( \norm{\cdot} \) dominates the
norm on \( \rmH_0^1(\Omega) \).  For this reason we add the following hypothesis.

\begin{assumption}\label{hyp:sum_eigs}
  There is a \( D>0 \) such that for all sequences \( {(c_j)}_{j\in\Z} \) with only finitely many non-zero entries  we
  have
  \[
    \sum_{i,j\in\Z} c_i c_j \skpH{\phi_i}{\phi_j}
    \leq D \sum_{i\in\Z} \abs{c_i}^2 (1+\abs{\lambda_j})
  \]
  where \( {(\phi_j)}_{j\in\Z} \) denotes the orthonormal basis from \cref{cor:ONB}.
\end{assumption}

\begin{prop}\label{prop:innerproduct}
  Let  \( \Omega \subset \R^N \) be a bounded domain, \( \lambda\in\R \), let
  \cref{hyp:sigma_c,hyp:w_T_cor,hyp:sum_eigs} hold. Then \( (H, \skp{\cdot}{\cdot}) \) is a
  Hilbert space such that \( H\subset \rmH_0^1(\Omega) \) is dense and \( \normH{\cdot} \les \norm{\cdot} \).
\end{prop}

We emphasize that we do not require the opposite bound \( \norm{\cdot}\les \normH{\cdot} \), which
appears to be much harder to verify. Note that this would imply the opposite
inclusion \( H\supset \rmH_0^1(\Omega) \). We now implement the
variational approach in the Hilbert space \( (H, \skp{\cdot}{\cdot}) \) and prove the existence of critical points of \( \Psi_\lambda \) in this space. Given
that \( H \) is dense in \( \rmH_0^1(\Omega) \), \( \Psi_\lambda'(u)=0 \) and \( u\in H \) in fact implies that \( u\in
\rmH_0^1(\Omega) \) is a weak solution in the \( \rmH_0^1(\Omega) \)-sense. In other words, critical points of
\( \Psi_\lambda:H\to\R \) provide solutions of \cref{eq:NLeq_general}.

\medskip

We first provide the functional analytical framework required by the Critical Point Theory
from~\cite{SzuWet_Nehari}. Being given the inner product from \cref{eq:def_norm}, we have an orthogonal
decomposition \( H = E^+ \oplus_\perp E^0 \oplus_\perp E^- \) where
\begin{align*}
  E^+ & \coloneqq \ov{\spa\setwt{\phi_j}{\lambda_j -\lambda > 0}}^{\norm{\cdot}},
  \\
  E^0 & \coloneqq \spa\setwt{\phi_j}{\lambda_j-\lambda = 0},
  \\
  E^- & \coloneqq \ov{\spa\setwt{\phi_j}{\lambda_j-\lambda < 0}}^{\norm{\cdot}}.
\end{align*}
The subspaces \( E^+, E^- \) are infinite-dimensional whereas \( E^0 \) is finite-dimensional, which is a consequence of \( |\lambda_j| \to \infty \) as \( |j|\to\infty \),
see \cref{eq:muj}. Here, \( E^0 = \{0\} \) is admissible.
Let \( \Pi^\pm : H  \to E^\pm \) denote the corresponding orthogonal projectors, and we will
write \( u^\pm \coloneqq \Pi^\pm u \) in the following. The whole point about the inner
product \cref{eq:def_norm} is that \( a(u,u)= \|u^+\|^2 - \|u^-\|^2 \), so we may rewrite the functional \( \Psi_\lambda \) from~\eqref{eq:def_Phi} as
\begin{equation}\label{eq:def_Phi2}
  \Psi_\lambda(u)
  = \frac{1}{2} \norm{u^+}^2 - \frac{1}{2} \norm{u^-}^2 - I(u)
  \qquad\text{where }I(u) \coloneqq \frac{1}{4} \int_\Omega \kappa(x) {u(x)}^4  \di{x}.
\end{equation}
Now, \( \Psi_\lambda \) has the right form to apply the critical point theorem from~\cite[Theorem~35]{SzuWet_Nehari}.

\begin{thm}[Szulkin, Weth]\label{thm:SzulkinWeth}
  Let \( (H, \skp{\cdot}{\cdot}) \) be a Hilbert space and suppose that the functional \( \Psi_\lambda : H
  \to \R \) satisfies
  \begin{itemize}
    \item[(i)] \( \Psi_\lambda(u) = \frac{1}{2} \norm{u^+}^2 - \frac{1}{2} \norm{u^-}^2 - I(u) \) where \( I(0)=0 \),
          \( \frac{1}{2} I'(u)[u] > I(u)>0 \) for all \( u \neq 0 \) and \( I \) is weakly lower semicontinuous.
    \item[(ii)] For each \( w \in E \sm (E^0 \oplus E^-) \) there exists a unique nontrivial
          critical point of \( \Psi_\lambda \) restricted to \( \R^+ w \oplus E^0 \oplus E^- \),
          %\( \restr{\Psi_\lambda}{\R^+ w \oplus E^0 \oplus E^-} \),
          which is the unique global maximizer.
    \item[(iii)] \( I'(u) = \oo(\norm{u}) \) as \( u \to 0 \).
    \item[(iv)] \( I(su)/s^2 \to \infty \) uniformly for \( u \) on weakly compact  subsets of \( H \setminus
          \{0\} \) as \( s \to \infty \).
    \item[(v)] \( I' \) is completely continuous.
  \end{itemize}
  Then \( \Psi_\lambda'(u) = 0 \) has a least energy solution.
  Moreover, if \( I \) is even, then this equation has infinitely many pairs of solutions.
\end{thm}

Applying this result in our setting, we get the following.

\begin{thm}\label{thm:var_general}
  Let  \( \Omega \subset \R^N \) be a bounded domain for \( N \in \{1,2,3\} \) and let
  \cref{hyp:sigma_c,hyp:w_T_cor,hyp:sum_eigs}  hold with \( \kappa(x)\geq \alpha>0 \)  for almost all
  \( x\in\Omega \). Then \cref{eq:NLeq_general} has infinitely many solutions in \( H \), among
  which a least energy solution \( u^*\in H \sm \{0\} \)   characterized by
  \[
    \Psi_\lambda(u^*) = \min\setwt{\Psi_\lambda(u)}{u \in H \sm \{0\} \text{ solves
      }\eqref{eq:NLeq_general}}.
  \]
\end{thm}
Note that we can also treat \( \kappa(x)\leq -\alpha<0 \) by considering \( (-\sigma, -\lambda) \). In that case, \cref{eq:NLeq_general} has infinitely many solutions in \( H \) among which a maximal energy solution \( u^* \) characterized by \( \Psi_\lambda(u^*) = \max\setwt{\Psi_\lambda(u)}{u \in H \sm \{0\} \text{ solves
  }\eqref{eq:NLeq_general}} \).

\begin{proof}
  We check the assumptions (i)-(v) from \cref{thm:SzulkinWeth}.
  First note that the functional \( I \) is continuously differentiable with
  Fr\'{e}chet derivative at \( u\in H \) given by \( I'(u)[v]=\int_\Omega \kappa(x) {u(x)}^3 v(x) \di{x} \) for \( v\in H \). This is a consequence of \( H\subset \rmH_0^1(\Omega)\subset \rmL^4(\Omega) \) with continuous
  injections by \cref{prop:innerproduct} and Sobolev's Embedding Theorem for \( N \in \{1,2,3\} \).
  From \cref{eq:def_Phi2} and \( I'(u)[u] = 4 I(u)>0 \) for \( u\neq 0 \) we infer the first part of (i). The
  weak lower semicontinuity of \( I \) follows from Fatou's Lemma given that \( u_k\rightharpoonup u \) in \( H \) implies \( u_k\rightharpoonup u \) in \( \rmH_0^1(\Omega) \) and thus
  \( u_k\to u \) pointwise almost everywhere. Hypothesis (iii) is immediate and (iv) holds because \( I(su)/s^2 = s^2 I(u) \)
  where \( \inf_B I > 0 \) for any weakly compact  subset \( B\subset H \setminus \{0\} \). Assumption (v) follows from
  the above formula for \( I' \) and the compactness of the embedding \( H\hookrightarrow \rmH_0^1(\Omega)\hookrightarrow \rmL^4(\Omega) \) due to \( N\in \{1,2,3\} \). The property (ii)
  is more difficult to prove. We refer to~\cite[pp. 31-32]{SzuWet_Nehari} where this has been carried out even
  in the case of more general nonlinearities.
\end{proof}

\subsection{An example in 1D}

We now show that the general result from above applies in the one-dimensional setting that we already discussed in our bifurcation analysis from \cref{cor:bif1D}.
So we consider the problem \cref{eq:NLeq1D}, namely
\[
  -\frac{\di{}}{\di{x}} \plr{\sigma(x) \, u'(x)} - \lambda \, c(x) \, u = \kappa(x) u^3
  \quad \text{in } \Omega,
  \qquad u \in \rmH_0^1(\Omega).
\]

\begin{lem}\label{lem:VerificationHypotheses2}
  Let \( \Omega, \sigma, c, \kappa \) be given as in \cref{cor:bif1D} with \( \kappa(x)\geq \alpha>0 \)
  for almost all \( x\in\Omega \). Then \cref{hyp:sigma_c,hyp:w_T_cor,hyp:sum_eigs} hold.
\end{lem}
%------------
\begin{proof}

  The assumptions of \cref{cor:bif1D} imply \cref{hyp:sigma_c} and \cref{lem:VerificationHypotheses1}
  yields \cref{hyp:w_T_cor}. So it remains to verify \cref{hyp:sum_eigs}, which is based on the explicit
  formulas for the eigenpairs from the proof of \cref{lem:sort_eigs_1d}. Moreover, we use
  \cref{eq:lambdaj>0},\cref{eq:lambdaj<0} for \( \tau_j:=\sqrt{|\lambda_j|} \), i.e.,
  \begin{align} \label{eq:lambdaj}
    \begin{aligned}
      \frac{\tan\plr{\tau_j k_+ a_+}}{\tanh\plr{\tau_j k_- a_-}} \cdot \frac{\sigma_-
        k_-}{\sigma_+ k_+} = 1  \quad (\lambda_j>0),\qquad
      \frac{\tan\plr{\tau_j k_- a_-}}{\tanh\plr{\tau_j k_+ a_+}}
      \cdot \frac{\sigma_+ k_+}{\sigma_- k_-} = 1,
      \quad (\lambda_j<0).
    \end{aligned}
  \end{align}
  Recall \( k_\pm= \sqrt{c_\pm/|\sigma_\pm|} \) and
  that \( a_-,\sigma_- \) are negative whereas \( a_+,\sigma_+,c_+,c_- \) are positive.
  All the following  computations of integrals can be  checked using the Python library Sympy~\cite{Sympy}
  and can be found in the Jupyter notebook \texttt{symbolic\_1d} (with format \texttt{IPYNB},
  \texttt{HTML}, and \texttt{PDF}) supply in the zenodo archive
  \href{https://doi.org/10.5281/zenodo.5140020}{\texttt{10.5281/zenodo.5592105}}.

  \medskip

  \noindent
  \textbf{1st step: Asymptotics for \( \tau_j \).}\; In view of \cref{eq:lambdaj} we have
  \begin{align} \label{eq:WeylLaw1D}
    \begin{aligned}
      \tau_j
       & = \frac{j\pi}{k_+a_+} + \frac{1}{k_+a_+} \arctan\left(\frac{\sigma_+k_+}{|\sigma_-|k_-}\right) + \oo(1)
       &                                                                                                         & (j\to\infty),   \\
      \tau_j
       & = \frac{|j|\pi}{k_-|a_-|} + \frac{1}{k_-|a_-|} \arctan\left(\frac{|\sigma_-|k_-}{\sigma_+k_+}\right) +
      \oo(1)
       &                                                                                                         & (j\to -\infty).
    \end{aligned}
  \end{align}
  In particular, plugging in the definition of \( k_+,k_- \) and using \( \sin(\arctan(z))= z/\sqrt{1+z^2} \) we
  find
  \begin{align} \label{eq:LimitsTauj}
    \begin{aligned}
      \sin^2(\tau_j k_+ a_+)
       &
      %= \frac{\sigma_+^2k_+^2}{(\sigma_+k_+)^2+(\sigma_-k_-)^2}
      %+ \oo(1)
      = \frac{\sigma_+ c_+}{\sigma_+ c_+ + |\sigma_-|c_-}
      + \oo(1)
       &   & (j\to\infty),   \\
      \sin^2(\tau_j k_- |a_-|)
       &
      %= \frac{\sigma_-^2k_-^2}{(\sigma_+k_+)^2+(\sigma_-k_-)^2}
      %+ \oo(1)
      = \frac{|\sigma_-| c_-}{\sigma_+ c_+ + |\sigma_-|c_-}
      + \oo(1)
       &   & (j\to -\infty).
    \end{aligned}
  \end{align}

  \noindent
  \textbf{2nd step: Formulas and asymptotics for \( \alpha_j \).}\;
  We have for \( \lambda_j > 0 \)
  \begin{align*}
    \alpha_j^{-2}
     & = \left[ \frac{a_+ c_+}{2 \sin^2(\tau_j k_+ a_+)}
      - \frac{c_+}{2k_+\tau_j\tan(\tau_j k_+ a_+)}\right]
    + \left[  \frac{a_- c_-}{2 \sinh^2(\tau_j k_- a_-)}
      - \frac{c_-}{2k_-\tau_j\tan(\tau_j k_- a_-)}
    \right]                                                                    \\
     & \stackrel{\eqref{eq:lambdaj}}= \frac{a_+ c_+}{2 \sin^2(\tau_j k_+ a_+)}
    + \frac{a_- c_-}{2 \sinh^2(\tau_j k_- a_-)}
    \;\stackrel{\eqref{eq:LimitsTauj}}=\; \frac{a_+(\sigma_+ c_+ + |\sigma_-|c_-)}{2\sigma_+}
    + \oo(1)\qquad (j\to\infty).
  \end{align*}
  Similarly, for \( \lambda_j<0 \),
  \begin{align*}
    \alpha_j^{-2}
     & =  \frac{|a_-| c_-}{2 \sin^2(\tau_j k_- |a_-|)}
    - \frac{a_+ c_+}{2 \sinh^2(\tau_j k_+ a_+)}
    \stackrel{\eqref{eq:LimitsTauj}}=\frac{|a_-|(\sigma_+ c_+ + |\sigma_-|c_-)}{2|\sigma_-|}+ \oo(1)\qquad
    (j\to -\infty).
  \end{align*}

  \noindent
  \textbf{3rd step: Formulas and asymptotics for \( \normH{\phi_j} \).}\;
  For \( \lambda_j>0 \) we compute
  \begin{align*}
    \frac{\normH{\phi_j}^2}{\alpha_j^2}
     & = \tau_j^2\left[ \frac{a_+ k_+^2 \sigma_+}{2 \sin^2(\tau_j k_+ a_+)}
    + \frac{k_+\sigma_+}{2\tau_j\tan(\tau_j k_+ a_+)}\right]                  \\
     & + \tau_j^2 \left[  \frac{a_- k_-^2\sigma_-}{2 \sinh^2(\tau_j k_- a_-)}
      + \frac{k_-\sigma_-}{2\tau_j \tanh(\tau_j k_- a_-)}
    \right]                                                                   \\
     & \stackrel{\eqref{eq:lambdaj}}=
    \tau_j^2\left[ \frac{a_+ c_+}{2 \sin^2(\tau_j k_+ a_+)}
      + \frac{|a_-| c_- }{2 \sinh^2(\tau_j k_- |a_-|)}
      + \frac{c_-}{k_-\tau_j\tanh(\tau_j k_- |a_-|)}\right].
  \end{align*}
  Similarly, for \( \lambda_j<0 \),
  \begin{align*}
    \frac{\normH{\phi_j}^2}{\alpha_j^2}
    =  \tau_j^2\left[\frac{|a_-| c_-}{2 \sin^2(\tau_j k_- |a_-|)}
      + \frac{a_+ c_+}{2 \sinh^2(\tau_j k_+ a_+)}
      + \frac{c_+}{k_+\tau_j\tanh(\tau_j k_+ a+)}\right].
  \end{align*}
  Using the precise expressions for \( \alpha_j^{-2} \) from the second step and \cref{eq:WeylLaw1D} we get
  \begin{align*}
    \normH{\phi_j}^2
     & =  \tau_j^2 + \OO(\tau_j)
    = \frac{\pi^2}{k_+^2a_+^2}\,j^2   + \OO(j)
     &
     &
    (j\to\infty),
    \\
    \normH{\phi_j}^2
     & =  \tau_j^2 + \OO(\tau_j)
    = \frac{\pi^2}{k_-^2 |a_-|^2}\,j^2 + \OO(|j|)
     &
     & (j\to -\infty).
  \end{align*}
  By the first step,   there is \( F>0 \) such that \( \normH{\phi_j}^2\leq F(1+|\lambda_j|) \) for all
  \( j\in\Z \).

  \medskip

  \noindent
  \textbf{4th step: Formulas and bounds for \( \skpH{\phi_i}{\phi_j} \).}\; For \( i\neq j \) explicit
  computations exploiting \cref{eq:lambdaj} reveal
  \begin{align*}
    \frac{\skpH{\phi_i}{\phi_j}}{\alpha_i \alpha_j}
     & = \frac{2\tau_i \tau_j}{\tau_i^2 - \tau_j^2} \clr{
      \frac{k_{-} \sigma_{-} \tau_{i}}{\tanh{\left(a_{-} k_{-} \tau_{j} \right)}}
      - \frac{k_{-} \sigma_{-} \tau_{j}}{\tanh{\left(a_{-} k_{-} \tau_{i} \right)}}
    },
     &
     & \text{for } \lambda_i > 0, \lambda_j > 0,
    \\
    \frac{\skpH{\phi_i}{\phi_j}}{\alpha_i \alpha_j}
     & =   \frac{2\tau_i \tau_j}{\tau_i^2 - \tau_j^2} \clr{
      \frac{k_{+} \sigma_{+} \tau_{i}}{\tanh{\left(a_{+} k_{+} \tau_{j} \right)}}
      - \frac{k_{+} \sigma_{+} \tau_{j}}{\tanh{\left(a_{+} k_{+} \tau_{i} \right)}}
    },
     &
     & \text{for } \lambda_i < 0, \lambda_j < 0,
    \\
    \frac{\skpH{\phi_i}{\phi_j}}{\alpha_i \alpha_j}
     & =  \frac{2\tau_i \tau_j}{\tau_i^2 + \tau_j^2} \clr{
      \frac{k_{+} \sigma_{+} \tau_{j}}{\tanh{\left(a_{+} k_{+} \tau_{i} \right)}}
      + \frac{k_{-} \sigma_{-} \tau_{i}}{\tanh{\left(a_{-} k_{-} \tau_{j} \right)}}
    },
     &
     & \text{for } \lambda_i < 0, \lambda_j > 0,
    \\
    \frac{\skpH{\phi_0}{\phi_j}}{\alpha_0 \alpha_j}
     & = \frac{\sigma_{-}}{a_{-}} + \frac{\sigma_{+}}{a_{+}},
     &
     & \text{for } \lambda_0 = 0, \lambda_j \neq 0.
  \end{align*}
  To estimate the first of these terms set \( A(z):= z\tanh(z) \).
  Using the Lipschitz continuity of \( A \), the estimate \( \tau_i \leq\sqrt{1+ \lambda_i} \)
  and~\eqref{eq:WeylLaw1D} we get for \( \lambda_i,\lambda_j>0 \)
  \begin{align*}
    \left|\frac{\skpH{\phi_i}{\phi_j}}{ \alpha_i \alpha_j}\right|
     & =  \frac{\tau_i\tau_j}{\tau_i+\tau_j}\,
    \left| \frac{2\sigma_-}{a_-\tanh\left(a_{-} k_{-} \tau_{j} \right)\tanh\left(a_{-} k_{-} \tau_i
      \right)}\right|\,\left|
    \frac{A(k_-\sigma_-\tau_i)-A(k_-\sigma_-\tau_j)}{ \tau_i -  \tau_j}\right| \\
     & \leq C\,\frac{\sqrt{1+\lambda_i}\sqrt{1+ \lambda_j}}{1+i+j}
  \end{align*}
  for some \( C>0 \). Using the asymptotics for \( (\alpha_i) \) from the second step  and performing analogous estimates
  in the other cases we find that there is \( G>0 \) independent of \( i,j \) such that
  \[
    |\skpH{\phi_i}{\phi_j}|
    \leq G\, \frac{\sqrt{1+\abs{\lambda_i}}\sqrt{1+\abs{\lambda_j}}
    }{1+|i|+|j|}
    \qquad \text{whenever } i\neq j.
  \]

  \medskip

  \noindent
  \textbf{5th step: Conclusion.}\;
  For \( \plr{c_j} \in \ell^\infty(\Z) \) with a finite number of non-zero entries, we define \( d_j :=
  \sqrt{1+\abs{\lambda_j}}\, \abs{c_j} \).
  The third and fourth step yield
  \[
    \sum_{i, j \in \Z} c_i c_j \skpH{\phi_i}{\phi_j}
    \leq F \sum_{j \in \Z} d_j^2
    + G \sum_{i \neq j \in \Z}  \frac{d_i d_j }{1+|i|+|j|}.
  \]
  Applying Hilbert's inequality, see for instance Eq.~(2) in~\cite{KufMalPer_Hardy}, gives
  \[
    \sum_{i, j \in \Z} c_i c_j \skpH{\phi_i}{\phi_j}
    \leq F \sum_{j \in \Z} d_j^2
    +  \tilde G  \sum_{j \in \Z}  d_j^2
    = (F+ \tilde G)\sum_{j \in \Z}  (1+|\lambda_j|)|c_j|^2
  \]
  for some \( \tilde G>0 \), which is all we had to show. So \cref{hyp:sum_eigs} holds.
  %      We recall from the proof
  %   of \cref{lem:VerificationHypotheses1}
  %   \[
  %     \ttT u(x) = u(x)
  %     \quad \text{if } 0 < x < a_+, \qquad
  %     \ttT u(x) = 2\chi(x) u(-m x) - u(x)
  %     \quad \text{if } a_- < x <0
  %   \]
  %   for some \( m>0 \)  and \( \chi\in \scrC_0^\infty(\Omega) \). Clearly, \( \ttT \) coincides with a bounded linear
  %   operator from \( \rmL^2(\Omega) \) to \( \rmL^2(\Omega) \), so an operator \( \tilde{\ttT} \) with the
  %   required properties exists. In \cref{lem:VerificationHypotheses1},
  %   the operator \( \ttK:\rmH_0^1(\Omega)\to \rmH_0^1(\Omega) \) was introduced via
  %   \( \alpha_2 m^{-3/2}  \normL{u}^2=\skpH{\ttK u}{u} \). Hence, we may choose \( \tilde{\ttK}:= \alpha_2
  %   m^{-3/2} \id:\rmL^2(\Omega)\to \rmL^2(\Omega) \) as required by \cref{hyp:sum_eigs}.
\end{proof}

Combining \cref{lem:VerificationHypotheses2} and \cref{thm:var_general}  we thus obtain:

\begin{cor}\label{cor:var}
  Let \( \Omega, \sigma, c, \kappa \) be given as in \cref{cor:bif1D} and \( \lambda \in \R \).
  Then equation \cref{eq:NLeq1D} has infinitely many nontrivial solutions in \( H \), among which a least energy solution.
\end{cor}

\begin{rem}\label{rem:RieszBasis}
  In the one-dimensional case of \cref{eq:NLeq1D} numerical investigations (cf. \cref{app:riesz_basis_1d}) indicate that not only
  \( \normH{\cdot}\les\norm{\cdot} \) but also \( \norm{\cdot}\les\normH{\cdot} \) holds. 
  As a consequence,   \( \norm{\cdot} \) and \( \normH{\cdot} \) are equivalent norms, so \(
  H=\rmH_0^1(\Omega) \) and the family \( \plr{\phi_j / \sqrt{1+\abs{\lambda_j}}}_{j \in \Z} \) is a
  so-called Riesz basis of \( \rmH_0^1(\Omega) \).
\end{rem}

\section{Open problems}

We finally address some open problems that we believe to be interesting to study:
\begin{enumerate}[(1)]
  %   \item [(1)] From a more applied point of view, it would be interesting to specify our results in the
  %   special case of a cylindrical waveguide containing a metamaterial in some of its layers. According to
  %   \cref{subsec:Physics}, this leads to the choice \( \Omega= \{x\in\R^2: |x|<R\} \) for some \( R>0 \) where
  %   \( \sigma,c,\kappa \) may be chosen to be constant on one centered ball and on annuli exhausting \( \Omega \).

  \item In \cref{cor:ONB} we showed that the eigenvalues satisfy the bounds
        $1+|\lambda_j|\geq c \plr{1+|j|}^{2/N}$ for all \( j \in \N \). An upper bound for the eigenvalues is
        unfortunately missing, let alone precise Weyl law asymptotics for the counting function
        \( \Lambda\mapsto \Card\setst{\lambda_j}{\abs{\lambda_j} \leq \Lambda} \) that one might compare to
        well-known ones for the Laplacian. In the 1D case this may be based on \cref{eq:WeylLaw1D}.
        In the higher-dimensional case we expect new difficulties due to ``plasmonic'' eigenvalues
        coming from a concentration near the interface.
        %Setting up a reasonable numbering of all eigenvalues even appears to be difficult.
        %             \item What happens to the eigenvalues \( \lambda_j \) as the parameters of the given linear problem
        %             approach the borderline of weak \( \ttT \)-coercivity?

  \item In view of the physical context described in \cref{subsec:Physics}, the case of sign-changing \( \sigma \) and \( c \) and even \( \kappa \) is relevant.
        As explained in \cref{rem:T1_assumption}, there is no hope to get an analogous spectral
        theory for all \( c \in \rmL^\infty(\Omega) \), but it would be interesting to identify
        those functions that admit such a one.

  \item \Cref{cor:bif1D} relies on the precise knowledge of eigenfunctions in the 1D case,
        especially regarding their nodal structure.
        Is there a way to find similar properties in more general one-dimensional problems  or even
        in higher-dimensional settings?

  \item In \cref{rem:RieszBasis}  we commented on the numerical evidence that 
        \( \plr{\phi_j / \sqrt{1+\abs{\lambda_j}}}_{j \in \Z} \) is a Riesz basis of \( \rmH_0^1(\Omega) \).
        In our one-dimensional case, a theoretical verification seems doable by extensive analysis of
        explicit formulas as in \cref{lem:VerificationHypotheses2}, but we have not found a reasonably
        short and elegant proof. We believe that such a one could be of interest. 
        
        %   \item[(6)] In general, variational bifurcation theory does not allow to prove bifurcation in form of a
        %   curve, but only the existence of a sequence of nontrivial solutions converging to the bifurcation point.
        %   This is due to B\"ohme's counterexample from~\cite[Section~6]{Boehme}. Do we generically observe a
        %   bifurcations in form of a curve, for instance do we generically have simple eigenvalues?
        %   \item[(7)] In classical variational bifurcation theory a \( k \)-dimensional kernel of \( \Psi_\lambda'(0) \)
        %   in \cref{thm:VariationalBifurcation} leads to the existence of \( k \) pairs of nontrivial solutions
        %   \( (u,\lambda),(-u,\lambda) \) that converge to the bifurcation point, not just one pair,
        %   see~\cite[Corollary~11.30]{Rab_Minimax}. Does this also hold in the strongly indefinite case?
\end{enumerate}

\appendix

\section{Riesz basis property for the 1d example}\label{app:riesz_basis_1d}

We numerically illustrate that  \( {\bigl(\phi_j\big/\sqrt{1+\abs{\lambda_j}}\bigr)}_{j\in \Z} \) indeed satisfies
the Riesz' basis property in the one-dimensional setting studied earlier.
In other words, we numerically show that both inequalities \( \normH{\cdot}\les \norm{\cdot} \) and
\( \norm{\cdot}\les \normH{\cdot} \) may be expected to hold.
Note that the first inequality was rigorously proven in \cref{prop:innerproduct}.
In the following, we use the same notation as in \cref{sec:var_method}.

\begin{figure}
  \includegraphics[scale=0.47]{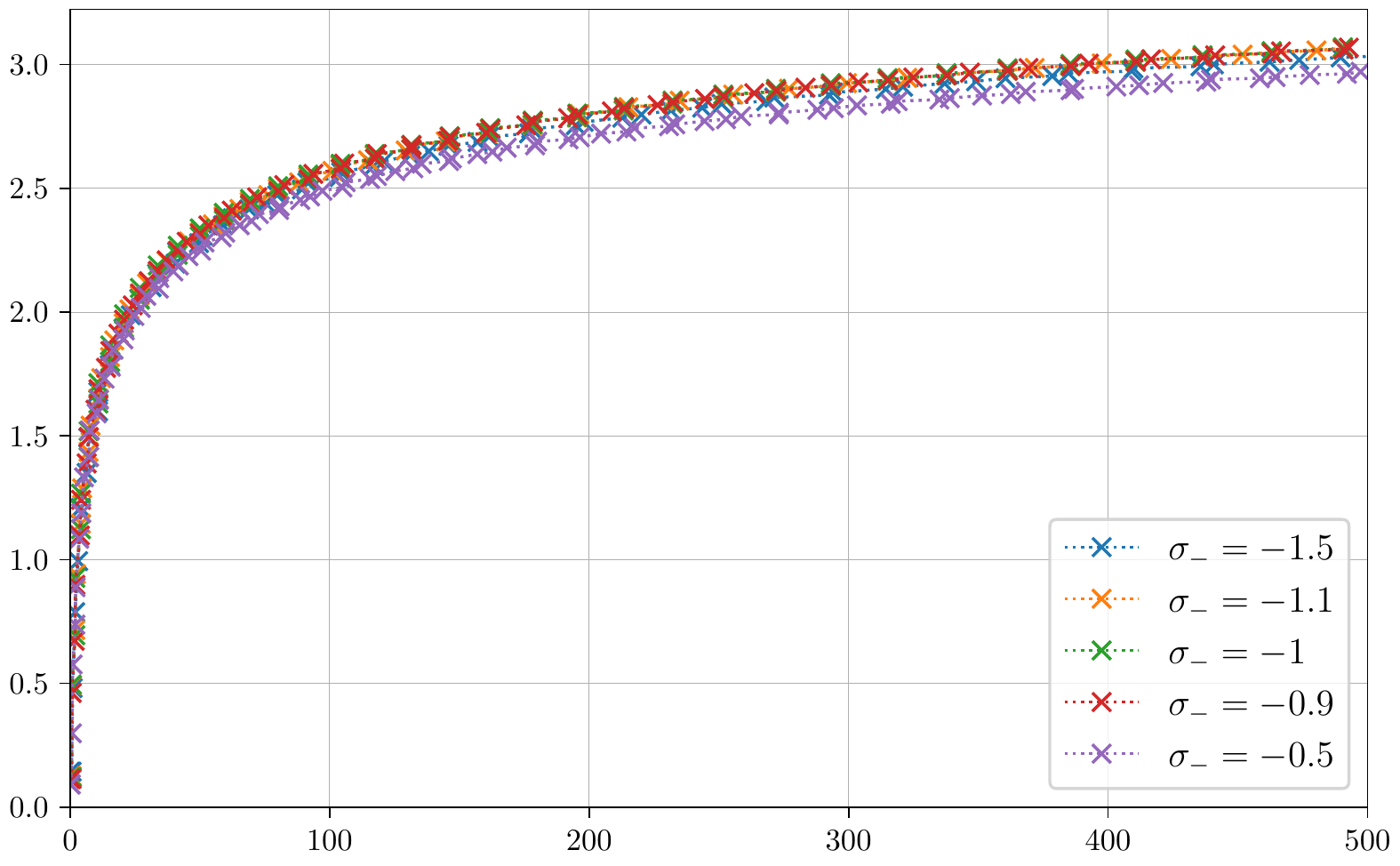}%
  \hspace*{2ex}%
  \includegraphics[scale=0.47]{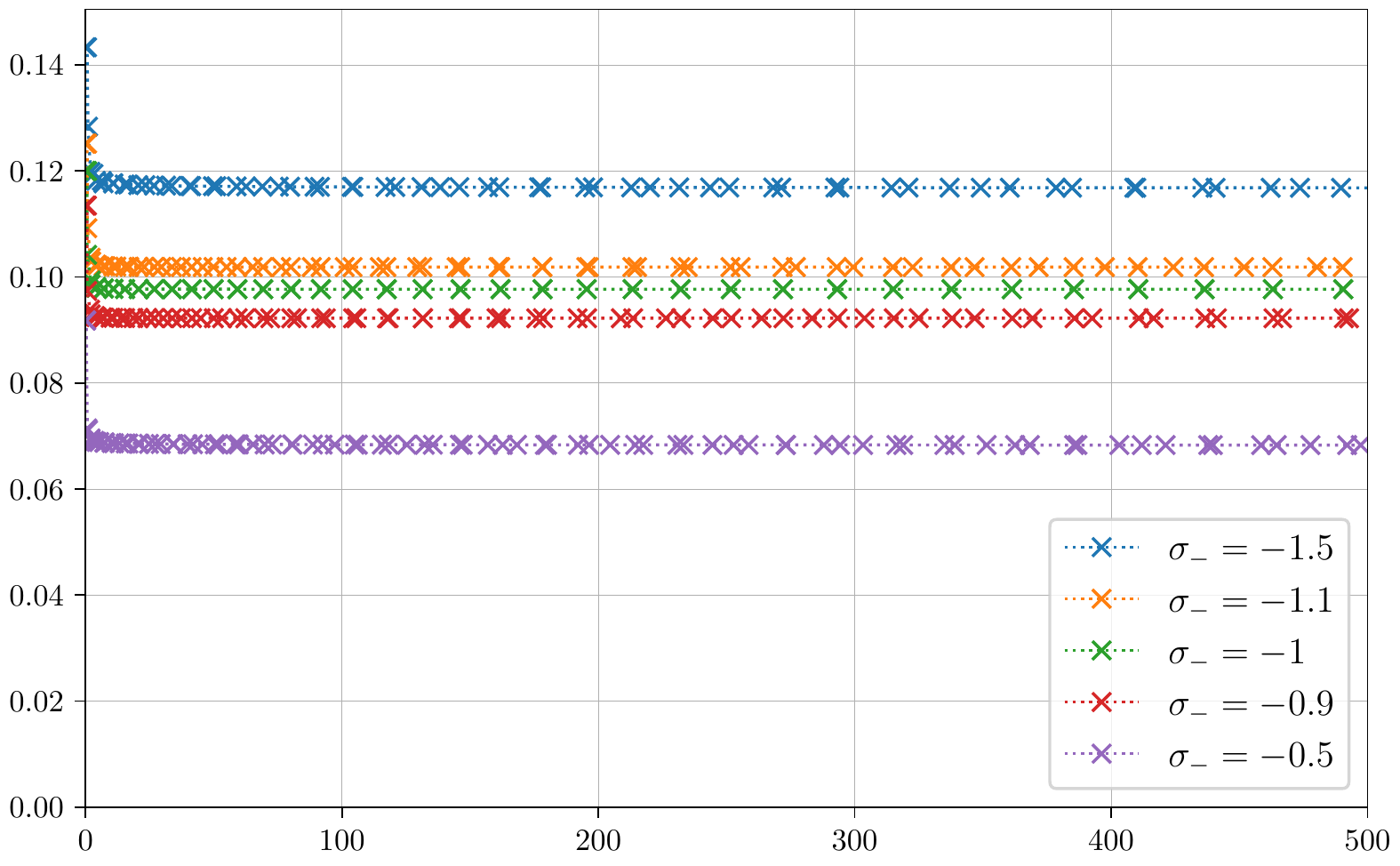}
  \caption{Maximal (left) and minimal (right) numerical eigenvalues of \( M_\Lambda \) in dependence on \( \Lambda \) in the one-dimensional setting.}%
  \label{fig:riesz}
\end{figure}

\medskip

For the numerical investigation, we choose the one-dimensional setting as in \cref{sec02}: We consider
the linear operator
\( \phi\mapsto - \frac{\di{}}{\di{x}}\big(\sigma(x) \, \phi'\big) \)  on
\( \Omega=(-5,5) \) with \( \sigma(x)=\sigma_+=1 \) on \( \Omega_+=(0,5) \) and various choices for \( \sigma(x)=\sigma_-<0 \)
on \( \Omega_- = (-5,0) \). We numerically compute the eigenvalues of the matrix
\[
  M_\Lambda =
  \plr{\frac{\skpH{\phi_i}{\phi_j}}{\sqrt{1+\abs{\lambda_i}}\sqrt{1+\abs{\lambda_j}}}}_{\abs{\lambda_i}, \abs{\lambda_j}\leq \Lambda}
\]
where the \( (\lambda_j,\phi_j) \) denote the eigenpairs from \cref{cor:ONB}.
The code is available on Zenodo with DOI
\href{https://doi.org/10.5281/zenodo.5707422}{\texttt{10.5281/zenodo.5707422}}.
\Cref{fig:riesz} shows the maximal and minimal eigenvalues in dependence on \( \Lambda \) for different choices of \( \sigma_- \) on the left and right, respectively.
For all considered choices of \( \sigma_- \), we clearly observe that the maximal and minimal eigenvalues asymptotically tend to a finite, non-zero value.
As discussed above, this behavior was rigorously proved for the maximal eigenvalue in \cref{prop:innerproduct} so that we focus on the minimal eigenvalue.
The asymptotic value is reached quickly. Note that for \( \abs{\sigma_-}\to 0 \) the minimal eigenvalue is
expected to degenerate to zero, which explains the dependence on \( \sigma_- \) observable in \cref{fig:riesz}
(right). This is in line with \cref{rem:ordering}.

% 
% %============================
% \section*{Acknowledgments}
% %============================
% 
% The authors wish to thank the reviewers for very careful proofreading and several suggestions that helped to
% improve the paper. They are also indepted to Monique Dauge (Univ. Rennes) for pointing out an error in a
% previous version of this manuscript.

%============================
\section*{Funding}
%============================

Funded by the Deutsche Forschungsgemeinschaft (DFG, German Research Foundation) --- Project-ID 258734477 --- SFB 1173.

\bibliographystyle{abbrv}
\bibliography{document}

\end{document}